\documentclass{amsart}

\usepackage{subfigure}
\usepackage{amsmath,amssymb, amsthm}
\usepackage{enumerate}
\usepackage{fancybox}
\usepackage{graphicx,color}
\usepackage{amsfonts}
\usepackage{multicol}
\usepackage{hyperref}
\usepackage{thmtools,thm-restate}
\usepackage[all]{xy}
\usepackage{amscd}

\numberwithin{equation}{section}

\newcommand{\dashedrightarrow}{\mathrel{-\,}\rightarrow}

\newcommand{\h}{\mathfrak{H}}

\newcommand{\sm}{\operatorname{sm}}
\newcommand{\id}{\operatorname{id}}

\newcommand{\C}{\mathbb{C}}
\newcommand{\R}{\mathbb{R}}
\newcommand{\Z}{\mathbb{Z}}
\newcommand{\CP}{\mathbb{CP}}
\newcommand{\Pa}{\partial}
\newcommand{\Diff}{\operatorname{Diff}}
\newcommand{\Int}{\operatorname{Int}}
\newcommand{\Crit}{\operatorname{Crit}}
\newcommand{\Ker}{\operatorname{Ker}}
\newcommand{\Sing}{\operatorname{Sing}}
\renewcommand{\Im}{\operatorname{Im}}

\newcommand{\rank}{\operatorname{rank}}
\newcommand{\reg}{\operatorname{reg}}

\newcommand{\Hom}{\operatorname{Hom}}

\newcommand{\NS}{\operatorname{NS}}

\newcommand{\Mod}{\operatorname{Mod}}
\newcommand{\Hess}{\operatorname{Hess}}

\theoremstyle{plain}
\newtheorem{theorem}{Theorem}[section]
\newtheorem{corollary}[theorem]{Corollary}
\newtheorem{lemma}[theorem]{Lemma}

\newtheorem{problem}[theorem]{Problem}

\theoremstyle{definition}

\newtheorem{example}[theorem]{Example}
\newtheorem{conjecture}[theorem]{Conjecture}
\newtheorem{remark}[theorem]{Remark}

\title{Topology of holomorphic Lefschetz pencils on the four-torus}

\author{Noriyuki Hamada}
\address{Graduate School of Mathematical Sciences, The University of Tokyo, Komaba, Meguro-ku, Tokyo 153-8914, Japan}
\email{nhamada@ms.u-tokyo.ac.jp}

\author{Kenta Hayano}
\address{Department of Mathematics, Graduate School of Science, Hokkaido University, Sapporo, Hokkaido 060-0810, Japan}
\email{k-hayano@math.sci.hokudai.ac.jp}

\begin{document}

\begin{abstract}
In this paper we discuss topological properties of holomorphic Lefschetz pencils on the four-torus. 
Relying on the theory of moduli spaces of polarized abelian surfaces, we first prove that, under some mild assumption, the (smooth) isomorphism class of a holomorphic Lefschetz pencil on the four-torus is uniquely determined by its genus and divisibility. 
We then explicitly give a system of vanishing cycles of the genus-$3$ holomorphic Lefschetz pencil on the four-torus due to Smith, and obtain those of holomorphic pencils with higher genera by taking finite unbranched coverings. 
One can also obtain the monodromy factorization associated with Smith's pencil in a combinatorial way. 
This construction allows us to generalize Smith's pencil to higher genera, which is a good source of pencils on the (topological) four-torus.
As another application of the combinatorial construction, for any torus bundle over the torus with a section we construct a genus-$3$ Lefschetz pencil whose total space is homeomorphic to that of the given bundle. 
\end{abstract}
\maketitle


\section{Introduction}

Lefschetz pencils on smooth four-manifolds are closely related to symplectic structures by Donaldson's construction \cite{Donaldson_1999} of Lefschetz pencils on symplectic manifolds and Gompf's generalization \cite{Gompf_2004} of Thurston's construction \cite{Thurston_1976} of symplectic structures on surface bundles. 
Moreover, Kas \cite{Kas_1980} and Matsumoto \cite{Matsumoto_1996} gave a combinatorial interpretation of isomorphism classes of Lefschetz fibrations, in particular their results enable us to construct Lefschetz fibrations (and symplectic four-manifolds) using simple closed curves on oriented surfaces (these results are generalized to that for Lefschetz pencils in \cite{BaykurHayano_isomHurwitz}).  
For these reasons Lefschetz pencils and fibrations have attracted a lot of interest from four-dimensional topologists in the last two decades. 
On the other hand, Lefschetz originally introduced Lefschetz pencils as generic pencils (i.e.~linear $1$-systems) of very ample line bundles in order to study the topology of algebraic varieties (e.g.~\cite{Lamotke_1981}). 
It is therefore natural to pay our attention to holomorphic Lefschetz pencils as well as smooth ones. 
In this paper we study holomorphic Lefschetz pencils on the four-torus from a topological point of view. 

In order to explain our main result, we first introduce two invariants for Lefschetz pencils. 
The \emph{genus} of a Lefschetz pencil is the genus of the closure of a regular fiber, and the \emph{divisibility} of a Lefschetz pencil is the maximum integer by which we can divide the integral homology class represented by the closure of a regular fiber. 
Two Lefschetz pencils on the same four-manifold have the same genus and divisibility if these are isomorphic, but the converse does not hold in general (the reader can find a counterexample for the converse in \cite{BaykurHayano_multisection}, for example). 
Our main result states that the converse becomes true for holomorphic Lefschetz pencils on the four-torus under some assumptions: 

\begin{theorem}\label{T:main_uniquenessLP}

Let $f_0,f_1$ be holomorphic Lefschetz pencils on the four-torus. 
Suppose either that the genus of $f_0$ is greater than $5$ or that the divisiblity of $f_0$ is greater than $1$. 
Then $f_0$ and $f_1$ are isomorphic if and only if they have the same genus and divisibility. 

\end{theorem}

\noindent
The assumption on the genus and the divisibility of $f_0$ in the theorem above is needed for some technical reasons and we believe that the theorem still holds without the assumption (see the last paragraph of Section~\ref{S:uniqueness_holLP}).

As we mentioned in the first paragraph, Lefschetz pencils are not only objects in complex geometry but also related to symplectic topology. 
It is especially important to find out how smooth Lefschetz pencils differ from holomorphic ones, which is related to the difference between complex (or K\"ahler) surfaces and symplectic four-manifolds. 
Since there exist non-complex symplectic four-manifolds, we can easily obtain Lefschetz pencils on non-complex four-manifolds using Donaldson's construction \cite{Donaldson_1999}. 
While it is in general hard to obtain monodromy factorizations of Lefschetz pencils coming from Donaldson's construction, several ingenious techniques, such as fiber sum operations and substitution operations, have been employed in order to give non-holomorphic Lefschetz pencils and fibrations (on possibly non-complex four-manifolds) with explicit monodromy factorizations (e.g.~\cite{Baykur_preprint_genus3LP,BaykurKorkmaz_smallLF,FS_1998_const_4mfd,HamadaKobayashiMonden,Korkmaz_2001,Monden_2014,OS_2000,Smith_2001_LP_divisor}). 
Furthermore, Li~\cite{Li_2008} constructed non-holomorphic Lefschetz pencils on minimal K\"ahler surfaces of general type. 
The construction in \cite{Li_2008} relies not only on Donaldson's result \cite{Donaldson_1999} but also on the differences between cohomology K\"ahler cones and symplectic cones. 
Since the cohomology K\"ahler cone of the four-torus coincides with its symplectic cone (see \cite[Proposition 4.10]{Li_2008}), this construction cannot give the affirmative answer to the following question:

\begin{problem}\label{P:exist_nonholLP_4torus}

Does there exist a non-holomorphic Lefschetz pencil on the four-torus? 

\end{problem}

\noindent
Problem~\ref{P:exist_nonholLP_4torus} is also important in complex geometry since it might be related to existence of non-K\"ahler symplectic forms on the four-torus. (Here, a symplectic form $\omega$ is said to be \textit{non-K\"ahler} if there do not exist complex structures compatible with $\omega$.) 
Indeed, for a holomorphic Lefschetz pencil we can take a symplectic form on the total space taming the complex structure by using \cite[Theorem 2.11 b)]{Gompf_2004}. 
Such a symplectic form is K\"ahler if it is further compatible with the complex structure. 
Theorem \ref{T:main_uniquenessLP}, together with explicit examples we will construct, gives rise to several constraints on monodromy factorizations of holomorphic Lefschetz pencils on the four-torus, in particular it might be possible to construct a non-holomorphic Lefschetz pencil on the four-torus using Theorem~\ref{T:main_uniquenessLP} (see Remark~\ref{R:relation_generalizationSmith_nonholLP}). 

As we mentioned earlier, a system of vanishing cycles of a Lefschetz pencil completely determines its isomorphism class. 
Thus we can find a non-holomorphic Lefschetz pencil on the four-torus using Theorem~\ref{T:main_uniquenessLP} once we can get vanishing cycles of a holomorphic Lefschetz pencil on the four-torus with sufficiently large genus or divisibility, and find another system of simple closed curves (associated with a Lefschetz pencil on the four-torus) which is not Hurwitz equivalent to the system of the vanishing cycles. 
In this paper we first analyze the simplest example of a holomorphic pencil on the four-torus: a genus-$3$ Lefschetz pencil due to Smith~\cite{Smith_2001_torus}. 

\begin{theorem}\label{T:VC_Smith}

The simple closed curves in Figure~\ref{F:VC_Smith} are vanishing cycles of a genus-$3$ Lefschetz pencil constructed in \cite{Smith_2001_torus}. 

\end{theorem}

\noindent
We can obtain holomorphic Lefschetz pencils on the four-torus with larger genera and divisibilities using finite unbranched coverings; The composition of a Lefschetz pencil and a finite unbranched covering of its total space is again a Lefschetz pencil, and any finite unbranched covering of the four-torus is also the four-torus. 
We will indeed prove that any holomorphic Lefschetz pencil on the four-torus with odd genus satisfying the assumption in Theorem~\ref{T:main_uniquenessLP} is isomorphic to the composition of the genus-$3$ Lefschetz pencil in \cite{Smith_2001_torus} and a finite unbranched covering. (See Lemma~\ref{L:relation_covering_polarization} and the observation following it.)

Baykur \cite{Baykur_preprint_genus3LP} constructed genus-$3$ Lefschetz pencils on symplectic Calabi-Yau four-manifolds (i.e. symplectic manifolds with trivial canonical classes) with positive $b_1$ relying on combinatorial techniques. 
The family of Lefschetz pencils given in \cite{Baykur_preprint_genus3LP} covers all possible rational homology types of symplectic Calabi-Yau four-manifolds with $b_1>0$ (cf.~\cite{Li_2006}), in particular it contains a four-manifold \textit{homeomorphic} to the four-torus. 
We will also construct a genus-$3$ Lefschetz pencil in a similar manner (by giving vanishing cycles, see Figure~\ref{F:SmithsLPConfiguration}) and prove that our pencil is isomorphic to both the pencil with vanishing cycles in Figure~\ref{F:VC_Smith}, that is, Smith's pencil, and the pencil given by Baykur~\cite{Baykur_preprint_genus3LP} (Lemma~\ref{L:HurwitzEquivalence} and Remark~\ref{R:HurwitzEquivalenceBaykur}).  
Our construction of the genus-$3$ pencil can be generalized to that of a genus-$g$ symplectic Calabi-Yau Lefschetz pencil $f_g$ for any $g\geq 3$. 
We will prove that the pencils with odd genera are compositions of Smith's pencil with finite unbranched coverings, and thus these are holomorphic pencils on the four-torus (Lemma~\ref{L:generalization_coveringSmith}). 
We further expect that the family of the pencils $\{ f_g~|~g-1 \text{~is prime}\}$ is a candidate for \emph{all essential} holomorphic Lefschetz pencils on the four-torus, where the tentative term essential means that
they cannot be decomposed as the composition of a holomorphic pencil and a finite unbranched covering of the four-torus (Conjectures~\ref{C:polarization_generalization} and~\ref{C:covering_generalisedSmithsLP}).
Applying a combinatorial operation to our genus-$3$ pencil, we will obtain a family of genus-$3$ Lefschetz pencils $\{f_{\alpha,\beta}\}$ parametrized by $\alpha,\beta\in \Mod(\Sigma_1^1;U)$ with $[\alpha,\beta]=1$, where $U=\{u\}\subset \Pa \Sigma_1^1$ and $\Mod(\Sigma_1^1;U)$ is the mapping class group of the one-holed torus $\Sigma_1^1$ fixing $U$ (for the precise definition of this mapping class group, see Subsection~\ref{S:monodromy_LFLP}). 

\begin{restatable}{theorem}{LPtorusbdl}\label{T:LPtorusbdl}

The total space of $f_{\alpha,\beta}$ is homeomorphic to that of the torus bundle over the torus with a section whose monodromy representation sends two elements generating $\pi_1(T^2)$ to $\alpha$ and $\beta$. 

\end{restatable}

\noindent
We will give a monodromy factorization of $f_{\alpha,\beta}$ explicitly in \eqref{eq:LPonT2bdl}.
Note that Smith~\cite{Smith_2001_torus} observed that any torus bundle over the torus with a section admits a genus-$3$ Lefschetz pencil. 
We believe that this pencil is isomorphic to ours, especially the total space of $f_{\alpha,\beta}$ is \emph{diffeomorphic} to that of a torus bundle over the torus. 

The constructions of Lefschetz pencils in the previous paragraph are related to the smooth classification problem of symplectic four-manifolds with Kodaira dimension $0$, which is one of the central concerns in symplectic topology. 
It is conjectured that any Kodaira diemnsion $0$ symplectic manifold is diffeomorphic to the K3 surface, the Enriques surface or a torus bundle over the torus. 
The family of symplectic Calabi-Yau manifolds given in \cite{Baykur_preprint_genus3LP} contains potential counter-examples of the conjecture. 
Furthermore, we would obtain a new symplectic four-manifold with Kodaira dimension $0$ once we can apply partial conjugations to any of the Lefschetz pencils in the previous paragraph so that the fundamental group of the total space of the resulting pencil is not a $4$-dimensional solvmanifold group (see \cite[Remark 18]{Baykur_preprint_genus3LP}).

In Section~\ref{S:uniqueness_holLP} we will prove Theorem~\ref{T:main_uniquenessLP} relying on the theory of moduli spaces of polarized abelian surfaces. 
In Section~\ref{S:example_holLP} we will first prove Theorem~\ref{T:VC_Smith}, that is, we will obtain vanishing cycles of the genus-$3$ holomorphic Lefschetz pencil due to Smith~\cite{Smith_2001_torus}. 
We will then discuss compositions of this pencil with finite unbranched coverings. 
In Section~\ref{Section:CombinatorialApproach} we will first re-construct Smith's pencil from a combinatorial point of view, and generalize the construction to obtain Lefschetz pencils with higher genera. 
Utilizing the technique Appendix~\ref{A:2ndhomology} we will prove that the divisibilities of these Lefschetz pencils are all $1$ (Lemma~\ref{L:divisibility_generalization}). 
We will further modify Smith's pencil to prove Theorem~\ref{T:LPtorusbdl}.

\section{Preliminaries}\label{S:preliminaries}

Throughout the paper, we will always assume that all the manifolds are smooth, oriented and connected unless otherwise noted. 

\subsection{Lefschetz pencils and fibrations}
Let $X$ be a closed $4$-manifold and $B\subset X$ a non-empty discrete set.
A smooth map $f:X\setminus B\to \CP^1$ is called a \text{Lefschetz pencil} if it satisfies the following conditions: 

\begin{enumerate}\label{P:definition_LP}

\item
the restriction $f|_{\Crit(f)}$ is injective, 

\item
each $x\in \Crit(f)$ is of \textit{Lefschetz type}, that is, there exists a complex coordinate $(U,\varphi:U\to \C^2)$ (resp.~$(V,\psi:V\to \C)$) of $x$ (resp.~$f(x)$) compatible with the orientation such that $\psi\circ f \circ \varphi^{-1}(z,w)$ is equal to $z^2+w^2$, 

\item
for any $b\in B$ there exist a complex coordinate $(U,\varphi)$ of $x$ compatible with the orientation and an orientation preserving self-diffeomorphism $\xi:\CP^1\to \CP^1$ such that $\xi\circ f \circ \varphi^{-1}(z,w)$ is equal to $[z:w]\in \CP^1$.

\end{enumerate}
Each point in $B$ is called a \textit{base point} of $f$. 
A smooth map $f:X\to \CP^1$ satisfying the conditions (1) and (2) above is called a \textit{Lefschetz fibration}. 
A Lefschetz pencil or fibration $f$ is said to be \textit{holomorphic} if there exists a complex structure of $X$ such that $f$ is holomorphic and we can take biholomorphic $\varphi,\psi$ and $\xi$ in the conditions above. 

\begin{remark}\label{R:complexity_Lefsing}
Since a Lefschetz singularity germ has infinite $\mathcal{A}_e$-codimension as a real germ, it is not finitely determined in smooth category, especially it is basically hard to determine whether a given smooth germ is of Lefschetz type or not. 
However, in complex category there is a useful criterion for a critical point to be of Lefschetz type: a critical point $x\in \C^2$ of a holomorphic function $f:\C^2\to \C$ is of Lefschetz type if and only if the \textit{complex Hessian} $\Hess(f)_x = \det\left(\left(\dfrac{\Pa^2 f }{\Pa z_k\Pa z_l}(x)\right)_{1\leq k,l\leq 2}\right)$ is not equal to $0$ (\cite[Lemma 2.11]{Voisin_2003}). 
\end{remark}

For a Lefschetz pencil or fibration $f:X\setminus B\to \CP^1$, the genus of the closure $\overline{f^{-1}(\ast)}$ of a regular fiber is called the \textit{genus} of $f$, which is denoted by $g(f)$. 
We further define the following number using a regular fiber: 
\[
d(f) = \sup\left\{n\in \Z~\left|~{}^\exists \alpha\in H_2(X;\Z)\text{ s.t. }\left[\overline{f^{-1}(\ast)}\right] = n\alpha\right.\right\} \in \Z_{>0}\cup \{\infty\}. 
\]
This number is called the \textit{divisibility} of $f$. 
Two Lefschetz pencils or fibrations $f_0:X_0\setminus B_0\to \CP^1$ and $f_1:X_1\setminus B_1\to \CP^1$ are said to be \textit{isomorphic} if there exist diffeomorphisms $\Phi:X_0\to X_1$ and $\phi:\CP^1\to \CP^1$ which make the following diagram commute:
\[
\begin{CD}
X_0\setminus B_0  @> \Phi >>  X_1\setminus B_1 \\
@V f_0 VV  @VV f_1 V\\
\CP^1 @> \phi >> \CP^1.
\end{CD}
\]
Obviously two isomorphic Lefschetz pencils or fibrations have the same number of base points, genus and divisibility, but the converse does not hold in general (a pair $f_{(2,2)}$ and $f_{(3,1)}$ in \cite{BaykurHayano_multisection}, for example, is a counterexample of the opposite direction).

\subsection{Monodromy factorizations of Lefschetz fibrations/pencils}\label{S:monodromy_LFLP}
Let $\Sigma=\Sigma_g^p$ be a compact genus-$g$ surface with $p$ boundary components. 
We take points $u_1,\ldots,u_p\in \Pa \Sigma$ from each of the component of $\Pa \Sigma$ and let $\delta_i\subset \Int(\Sigma)$ be a simple closed curve parallel to the boundary component containing $u_i$. 
Let $U$ be the set $\{u_1,\ldots,u_p\}$ and $\Diff(\Sigma;U)$ the group of orientation-preserving diffeomorphisms of $\Sigma$ which preserve the set $U$. 
We call the set $\pi_0(\Diff(\Sigma;U))$ the \textit{mapping class group of $\Sigma$} and denote it by $\Mod(\Sigma;U)$. 
An element of $\Mod(\Sigma;U)$ is the isotopy class of an element in $\Diff(\Sigma;U)$, where isotopies fix the set $U$. 
The group structure of $\Mod(\Sigma;U)$ is induced by compositions of maps, that is, $[\varphi_1]\cdot [\varphi_2] = [\varphi_1\circ \varphi_2]$ for $\varphi_1,\varphi_2 \in \Diff(\Sigma;U)$.

Now let $f:X\setminus B\to \CP^1$ be a Lefschetz pencil or fibration with $n$ critical points. 
Set $f(\Crit(f)) = \{a_1,\ldots,a_{n}\}$, and take paths $\alpha_1,\ldots,\alpha_{n}\subset \CP^1$ with a common initial point  $a_0 \in \CP^1\setminus f(\Crit(f))$ such that
\begin{itemize}
\item
$\alpha_1,\ldots,\alpha_{n}$ are mutually disjoint except at $p_0$, 
\item
$\alpha_i$ connects $a_0$ with $a_i$, 
\item
$\alpha_1,\ldots,\alpha_{n}$ are ordered counterclockwise around $a_0$, i.e.~there exists a small loop around $a_0$ oriented counterclockwise, hitting each $\alpha_i$ only once in the given order.
\end{itemize}  

\noindent
We take a loop $\widetilde{\alpha}_i$ with the base point $a_0$ by connecting $\alpha_i$ with a small counterclockwise circle with center $a_i$. 
We call a system of paths $\widetilde{\alpha}_1,\ldots,\widetilde{\alpha}_n$ obtained by the procedure above a \textit{Hurwitz path system} of $f$. 
For each $b\in B$, let $D_b$ be a sufficiently small $4$-ball neighborhood of $b$ and $\nu B$ the disjoint union $\sqcup_{b\in B}D_b$. 
For each $b$ we take a section $S_b\subset \Pa D_b$ of $f$. 
Let $\mathcal{H}$ be a horizontal distribution of $f|_{X\setminus (\nu B\cup \Crit(f))}$, that is, $\mathcal{H} = \{\mathcal{H}_x\}_{x\in X\setminus (\nu B\cup\Crit(f))}$ is a plane field such that $\Ker(df_x) \oplus \mathcal{H}_x = T_xX$ for any $x\in X\setminus (\nu B\cup\Crit(f))$. 
We assume that $\mathcal{H}_x = T_x S_b$ for any $x\in S_b$ and $\mathcal{H}_x \subset T_x \Pa D_b$ for any $x\in\Pa D_b$. 
Using $\mathcal{H}$, we can take a lift of the direction vector field of $\widetilde{\alpha}_i$ and a flow of this lift gives rise to a self-diffeomorphism of $f^{-1}(a_0)$. 
We call this diffeomorphism a \textit{parallel transport} of $\widetilde{\alpha}_i$ and its isotopy class a \textit{local monodromy} around $a_i$. 
Note that a local monodromy does not depend on the choice of $\mathcal{H}$. 

Under an identification of the pair $(f^{-1}(a_0)\setminus \nu B, f^{-1}(a_0)\cap \sqcup_{b\in B}S_b)$ with the pair $(\Sigma_g^p,U)$, we can regard a parallel transport as a diffeomorphism in $\Diff(\Sigma_g^p;U)$, and thus, a local monodromy as a mapping class in $\Mod(\Sigma_g^p;U)$. 
A local monodromy around $a_i$ is a Dehn twist $t_{c_i}$ along some simple closed curve $c_i \subset \Int\Sigma_g^p$ (see \cite{Kas_1980}). 
The curve $c_i$ is called a \textit{vanishing cycle} of $f$. 
Since the concatenation $\widetilde{\alpha}_1\cdots \widetilde{\alpha}_{n}$ is null-homotopic in $\CP^1\setminus f(\Crit(f))$ and each $S_b$ has the self-intersection $-1$, the composition $t_{c_n}\cdots t_{c_1}$ is equal to $t_{\delta_1}\cdots t_{\delta_p}$ in $\Mod(\Sigma_g^p;U)$ (which is the identity if $B=\emptyset$). 
The factorization 
\[
t_{c_n}\cdots t_{c_1} =t_{\delta_1}\cdots t_{\delta_p}
\]
is called a \textit{monodromy factorization} of $f$. 
Two factorizations $t_{c_n}\cdots t_{c_1} = t_{d_n}\cdots t_{d_1} =t_{\delta_1}\cdots t_{\delta_p}$ are said to be \textit{Hurwitz equivalent} if one can obtained from the other by successive applications of the following two kinds of moves: 

\begin{itemize}

\item
\textit{Elementary transformation}: 
$t_{c_n}\cdots t_{c_{i+1}}t_{c_i}\cdots t_{c_1} \rightarrow t_{c_n}\cdots t_{t_{c_{i+1}}(c_i)}t_{c_{i+1}}\cdots t_{c_1}$, 

\item
\textit{Simultaneous conjugation}: 
$t_{c_n}\cdots t_{c_1} \rightarrow t_{\varphi(c_n)}\cdots t_{\varphi(c_1)}$ for $\varphi \in \Mod(\Sigma_g^p;U)$.  

\end{itemize}

\begin{theorem}[\cite{Kas_1980},\cite{Matsumoto_1996} and \cite{BaykurHayano_isomHurwitz}]\label{T:isom_Hurwitzeq}

Assume that $2-2g-p$ is negative. 
Two Lefschetz pencils or fibrations of genus $g$ with $p$ base points are isomorphic if and only if the corresponding monodromy factorizations are Hurwitz equivalent. 

\end{theorem}

\subsection{Moduli spaces of polarized abelian surfaces}\label{S:moduli_abelian_surface}

By an \textit{abelian surface}, we mean a complex torus of dimension $2$ which can be holomorphically embedded into $\CP^N$ for sufficiently large $N$. 
For a complex torus $T$, a \textit{polarization} of $T$ is a cohomology class $H\in H^2(T;\Z)$ which is the first Chern class of an ample line bundle. 
Let $\Lambda\subset \C^2$ be a lattice and $T=\C^2/\Lambda$. 
We can canonically identify the group $H_1(T;\Z)$ with the lattice $\Lambda$. 
Using this identification we can regard polarizations of $T$ as integer-valued alternating forms on $\Lambda$. 
For any polarization $E$ we can take a basis $\mu_1,\mu_2,\lambda_1,\lambda_2$ of $\Lambda$ such that $E$ is represented by the following matrix with respect to this basis: 
\[
E = \begin{pmatrix}
0&D\\-D&0
\end{pmatrix}\text{, where }D=(\boldsymbol{d}_1,\boldsymbol{d}_2)=\begin{pmatrix}
d_1&0\\0&d_2
\end{pmatrix},\hspace{.3em}d_i>0,\hspace{.3em}d_1|d_2. 
\]
We call the pair $(d_1,d_2)$ or the matrix $D$ the \textit{type} of the polarization $E$.

We denote by $\h_2$ the set of symmetric complex $2\times 2$ matrices with positive definite imaginary part, which is a connected complex manifold of dimension $3$. 
For $Z=(\boldsymbol{z}_1,\boldsymbol{z}_2)\in \h_2$, we denote the ordered set $(\boldsymbol{z}_1,\boldsymbol{z}_2,\boldsymbol{d}_1,\boldsymbol{d}_2)$ by $\Lambda_Z$. 
The set $\Lambda_Z$ is a basis of a lattice in $\C^2$, which we denote by $\overline{\Lambda_Z}$. 
In particular, $\Lambda_Z$ gives rise to a complex torus $T_Z = \C^2/\overline{\Lambda_Z}$. 
Let $H_Z$ be the imaginary part of a hermitian form represented by the matrix $\Im(Z)^{-1}$ with respect to the standard basis of $\C^2$. 
The form $H_Z$ is a real-valued alternating form on $\C^2$. 
It is easy to check that the representation matrix of $H_Z|_{\overline{\Lambda_Z}}$ with respect to the basis $\Lambda_Z$ is $\begin{pmatrix}
0&D\\-D&0
\end{pmatrix}$. 
Thus, $H_Z$ is a $(d_1,d_2)$-polarization of $T_Z$. 
Conversely, any polarized abelian surface can be obtained by the construction above. 
More precisely, it is known that for any complex torus $T=\C^2/\overline{\Lambda}$, its polarization $H$ and a basis $\Lambda$ of the lattice $\overline{\Lambda}$ with respect to which the representation of $H|_{\overline{\Lambda}}$ is $\begin{pmatrix}
0&D\\-D&0
\end{pmatrix}$, there exists a matrix $Z\in \h_2$ and a biholomorphic map $\Psi:\C^2\to \C^2$ such that two  triples $(T,H,\Lambda)$ and $(T_Z,H_Z,\Lambda_Z)$ correspond by $\Psi$ (see \cite[\S.8.1]{LangeBirkenhake_1992}). 
In particular, $\h_2$ is a moduli space of $(d_1,d_2)$-polarized abelian surfaces with a symplectic basis of the lattice. 
The followings are basic properties of this moduli space which will be used in this paper: 

\begin{lemma}\label{L:nongeneric_subset_moduli}

We fix the pair $(d_1,d_2)$ and we regard $\h_2$ as a moduli space of $(d_1,d_2)$-polarized abelian surfaces as explained above. 

\begin{enumerate}

\item
the subset $S_0=\{Z\in \h_2~|~\NS(T_Z)\not\cong \Z\}$ is contained in a countable union of proper analytic subsets of $\h_2$, where $\NS(T_Z)= \Im(c_1:H^1(T_Z;\mathcal{O}^\ast_{T_Z})\to H^2(T_Z;\Z))$ is the N\'eron-Severi group of $T_Z$. 

\item
the subset $S_1=\{Z\in \h_2~|~T_Z \cong E_1\times E_2\text{ for some elliptic curves }E_1,E_2\}$ is contained in $S_0$. 

\end{enumerate}

\end{lemma}

\begin{proof}
The first statement is in \cite[Exercise 8.1]{LangeBirkenhake_1992} and the details are left to the reader. 
In order to prove the second one, assume that there exists elliptic curves $E_1,E_2$ such that $T_Z$ is biholomorphic to $E_1\times E_2$.
The cohomology classes represented by the divisors $E_1\times \{0\}$ and $\{0\}\times E_2$ are both contained in $\NS(T_Z)$. 
Thus the rank of $\NS(T_Z)$ is at least two. 
\end{proof}

For a holomorphic line bundle $L$ we denote the set of holomorphic sections by $\Gamma(L)$, which is a finite-dimensional complex vector space. 
In the rest of this subsection we will construct an ample line bundle $L_Z$ with $c_1(L_Z)=H_Z$ and a basis of $\Gamma(L_Z)$ explicitly (for more systematic constructions of line bundles on complex tori and their sections, see \cite[Chapters 2 and 3]{LangeBirkenhake_1992}).
For $Z\in \h_2$ we denote the submodules $\left<\boldsymbol{z}_1,\boldsymbol{z}_2\right>$ and $\left<\boldsymbol{d}_1,\boldsymbol{d}_2\right>$ of the lattice $\overline{\Lambda_Z}$ by $\Lambda_Z^1$ and $\Lambda_Z^2$, respectively. 
Let $V_Z^i$ be a real subspace of $\C^2$ generated by $\Lambda_Z^i$. 
It is easy to see that $\C^2$ is equal to the direct sum $V_Z^1\oplus V_Z^2$. 
Using this decomposition we define a map $\chi_Z:\C^2\to S^1$ as follows: 
\[
\chi_Z(v_1+v_2)=\exp\left(\pi i H_Z(v_1,v_2)\right), 
\]
where $v_i \in V_i$. 
We further define a map $a_Z:\overline{\Lambda_Z}\times \C^2\to \C^\times$ as follows: 
\[
a_Z(\lambda,v) = \chi_Z(\lambda)\exp\left(\pi \Im(Z)^{-1}(v,\lambda)+\frac{\pi}{2}\Im(Z)^{-1}(\lambda,\lambda)\right), 
\]
where $\Im(Z)^{-1}$ is regarded as a hermitian form on $\C^2$. 
We then define a line bundle $L_Z$ as follows: 
\[
L_Z = (\C^2\times \C)/\sim, 
\]
where the equivalence relation $\sim$ is generated by the following relation: 
\[
(v+\lambda,z) \sim (v,a(\lambda,v)z),\text{ for }\lambda\in \overline{\Lambda_Z}\text{ and }v\in \C^2. 
\]
By the assumption the alternating form $H_Z$ is trivial on $V_Z^2$. 
Thus the restriction $\Im(Z)^{-1}|_{V_Z^2}$ is symmetric. 
Since $\C$-extension of $V_Z^2$ is the whole space $\C^2$, we can define a symmetric form $B_Z$ on $\C^2$ by extending $\Im(Z)^{-1}|_{V_Z^2}$. 
We define a holomorphic map $\vartheta_Z^{00}:\C^2\to \C$ as follows: 
\[
\vartheta_Z^{00}(v) =  \exp\left(\frac{\pi}{2}B_Z(v,v)\right)\sum_{\lambda\in \Lambda_Z^1}\exp\left(\pi(\Im(Z)^{-1}-B_Z)(v,\lambda)-\frac{\pi}{2}(\Im(Z)^{-1}-B_Z)(\lambda,\lambda)\right).
\]
We can verify that the map $T_Z\ni [v]\mapsto [(v,\vartheta_Z^{00}(v))]\in L_Z$ is well-defined, especially $\vartheta_Z^{00}$ gives rise to a section of $L_Z$ (\cite[Lemma 3.2.4]{LangeBirkenhake_1992}). 
For two integers $0\leq i< d_1$ and $0\leq j < d_2$ we define the map $\vartheta_Z^{ij}:\C^2\to \C$ as follows: 
\[
\vartheta_Z^{ij}(v) = a_Z(w_{ij},v)^{-1}\vartheta_Z^{00}(v+w_{ij}), 
\]
where $w_{ij} = \dfrac{i}{d_1}\boldsymbol{z}_1+\dfrac{j}{d_2}\boldsymbol{z}_2$. 
This also gives rise to a section of $L_Z$ for each $i,j$ (\cite[Corollary 3.2.6]{LangeBirkenhake_1992}). 

\begin{theorem}[{\cite[Theorem 3.2.7]{LangeBirkenhake_1992}}]\label{T:basis_sections}

The set $\{\vartheta_Z^{ij}\in \Gamma(L_Z)~|~0\leq i<d_1, 0\leq j<d_2\}$ is a basis of $\Gamma(L_Z)$. 

\end{theorem}

\section{Uniqueness of holomorphic Lefschetz pencils on the four-torus}\label{S:uniqueness_holLP}

In this section we prove Theorem~\ref{T:main_uniquenessLP}. 
Let $L$ be a holomorphic line bundle on $T^4$. 
For $s_0,s_1\in \Gamma(L)$, we define a meromorphic map $[s_0:s_1]:X\dashedrightarrow \CP^1$ as follows: 
for $x\in T^4$, take a trivialization $\pi_L^{-1}(U)\cong U\times \C$ and regard the restriction $s_i|_{U}$ as a holomorphic function, and define $[s_0:s_1](x) = [s_0(x):s_1(x)]$. 
It is easy to see that a point $[s_0(x):s_1(x)]\in \CP^1$ does not depend on the choice of a trivialization of $L$ around $x$.  
The map $[s_0:s_1]$ is defined on the complement of $s_0^{-1}(0)\cap s_1^{-1}(0)$. 

\begin{lemma}\label{L:holLP_fromlinebdl}

For any holomorphic Lefschetz pencil $f$ on $T^4$, there exists an ample line bundle $L$ and its sections $s_0,s_1$ such that $f$ is equal to $[s_0:s_1]$. 

\end{lemma}

\begin{proof}
Let $V_i = \{[z_0:z_1]\in\CP^1~|~ z_i\neq 0\}$ and $\psi_i:V\to \C$ a map defined as $\psi_i([z_0:z_1])=z_j/z_i$ ($j\neq i$). 
For each $b\in B$ we take a $4$-ball neighborhood $D_b$ and a biholomorphic map $\Phi_b:D_b\to \C^2$ so that $D_b$ and $D_{b'}$ are disjoint if $b\neq b'$ and $f\circ \Phi_b^{-1}(z,w)$ is equal to $[z:w]$. 
We put $\Phi_b(x) = (\Phi_b^0(x),\Phi_b^1(x))$. 
We define a space $L$ as follows: 
\[
L = (f^{-1}(V_0)\times \C)\sqcup (f^{-1}(V_1)\times \C)\sqcup_{b\in B} (D_b\times \C)/\sim,  
\]
where the equivalence relation $\sim$ is defined as follows: 
\begin{itemize}

\item
for $x\in f^{-1}(V_0\cap V_1)$, $f^{-1}(V_0)\times \C\ni (x,z)\sim (x,\psi_1(f(x))z)\in f^{-1}(V_1)\times \C$, 

\item
for $x\in f^{-1}(V_0)\cap D_b$, $f^{-1}(V_0)\times \C\ni (x,z)\sim (x,\Phi_b^0(x)z)\in D_b\times \C$, 

\item
for $x\in f^{-1}(V_1)\cap D_b$, $f^{-1}(V_1)\times \C\ni (x,z)\sim (x,\Phi_b^1(x)z)\in D_b\times \C$. 

\end{itemize}
It is easy to see that $L$ together with the projection $\pi_f:L\to T^4$ onto the first component is a holomorphic line bundle on $T^4$. 
We define two sections $s_0,s_1:T^4\to L$ of $L$ as follows: 
\begin{align*}
s_0(x) = & \begin{cases}
(x,1)\in f^{-1}(V_0)\times \C & (x\in f^{-1}(V_0)) \\
(x,\psi_1(f(x)))\in f^{-1}(V_1)\times \C & (x\in f^{-1}(V_1)) \\
(x,\Phi_b^0(x))\in D_b\times \C & (x\in D_b),
\end{cases} \\
s_1(x) = & \begin{cases}
(x,\psi_0(f(x)))\in f^{-1}(V_0)\times \C & (x\in f^{-1}(V_0)) \\
(x,1)\in f^{-1}(V_1)\times \C & (x\in f^{-1}(V_1)) \\
(x,\Phi_b^1(x))\in D_b\times \C & (x\in D_b). 
\end{cases} 
\end{align*}
It is easy to verify that the map $[s_0:s_1]$ is equal to $f$. 
The line bundle $L$ has non-trivial holomorphic sections $s_0,s_1$ and $c_1^2(L) = |B|>0$. 
Thus $L$ is ample by \cite[Proposition 4.5.2]{LangeBirkenhake_1992}. 
\end{proof}

\begin{remark}

The type of the polarization $c_1(L)$ for $L$ in Lemma~\ref{L:holLP_fromlinebdl} is not equal to $(1,1)$ since $s_0$ and $s_1$ are linearly independent (cf.~Theorem~\ref{T:basis_sections}). 

\end{remark}

We can easily prove the following lemma using the inverse function theorem for holomorphic maps:

\begin{lemma}\label{L:criterion_basepoint_Lefschetz}

Let $x\in T^4$ be a base point of $h=[s_0:s_1]$. 
Then the condition (3) in p.\pageref{P:definition_LP} holds at $x$ if and only if $(ds_0)_x$ and $(ds_1)_x$ are linearly independent. 

\end{lemma}

\begin{remark}\label{R:independence_choice_linebundle}

For a polarization $H$ of a complex torus $T$, the following map is surjective (\cite[Corollary 2.5.4]{LangeBirkenhake_1992}):
\[
T \to c_1^{-1}(H)\subset H^1(T;\mathcal{O}_L^\ast),\hspace{.3em} v \mapsto t_v^\ast L, 
\]
where $L$ is an ample line bundle with $c_1(L)=H$ and $t_v:T\to T$ is the translation $x\mapsto x+v$. 
In particular the set of isomorphism classes of holomorphic Lefschetz pencils obtained from an ample line bundle $L$ depends only on the class $c_1(L)$. 
Thus, any holomorphic Lefschetz pencil on $T^4$ is isomorphic to a pencil obtained from a pair of sections of a line bundle $L_Z$ we constructed in Subsection~\ref{S:moduli_abelian_surface}. 

\end{remark}

\subsection{A condition for pencils to be Lefschetz}\label{S:condition_pencil_Lefschetz}

As we explained in Subsection~\ref{S:moduli_abelian_surface}, $\h_2$ is a moduli space of $(d_1,d_2)$-polarized abelian surfaces with a symplectic basis of the associated lattice for each type $(d_1,d_2)$. 
Using a holomorphic function $\vartheta_Z^{ij}$ ($0\leq i<d_1,0\leq j<d_2$) on $\C^2$, we define a map $\varphi_Z:T_Z\dashedrightarrow \CP^N$ ($N=d_1d_2-1$) as follows:
\[
\varphi_Z(\overline{x}) = [\cdots:\vartheta_Z^{ij}(x):\cdots], 
\] 
where $\overline{x}\in T_Z$ is a point represented by $x\in \C^2$. 
This map is well-defined by double-periodicity of $\vartheta_Z^{ij}$ and is defined on the complement of the intersection $\cap_{i,j}{\vartheta_Z^{ij}}^{-1}(0)$.  
We denote the set of hyperplanes in $\CP^N$ by $(\CP^N)^\ast$, which is canonically biholomorphic to $\CP^N$. 
For any projective line $P\subset (\CP^N)^\ast$, we define a pencil $f_P:T_Z\dashedrightarrow P$ as follows: 
\[
f_P(\overline{x}) = H \in P \text{ if }\overline{x}\in \varphi_Z^{-1}(H). 
\]
Let $H_0,H_1\in P$ be distinct hyperplanes and $\sum_{i,j}a_{ij}^kX_{ij}$ a defining polynomial of $H_k$. 
It is easily verify that $f_P$ is defined on the complement of $\varphi_Z^{-1}(H_0\cap H_1)$ and is isomorphic to $[s_0:s_1]$, where $s_k=\sum_{i,j}a_{ij}^k\vartheta_Z^{ij} \in \Gamma(L_Z)$. 
Thus, by Lemma~\ref{L:holLP_fromlinebdl} any holomorphic Lefschetz pencil on $T^4$  is isomorphic to $f_P:T_Z\dashedrightarrow P$ for some $P\subset (\CP^N)^\ast$. 
In this subsection, we will discuss when $f_P$ becomes a Lefschetz pencil. 
Note that the ideas of the proofs in this subsection are based on the arguments in \cite[\S{.2.1.1}]{Voisin_2003}

\begin{lemma}\label{L:(1,d)not product}

If $d_1=1$ and $f_P$ is a Lefschetz pencil for some $P\subset (\CP^N)^\ast$, then $T_Z$ is not biholomorphic to a product of elliptic curves. 

\end{lemma}

\begin{proof}
Suppose that $T_Z$ would be a product $E_1\times E_2$. 
The line bundle $L_Z$ would be a tensor product $p_1^\ast L_1\otimes p_2^\ast L_2$, where $p_i$ is the projection onto the $i$-th component, $L_1$ is a line bundle on $E_1$ of degree $1$ and $L_2$ is a line bundle on $E_2$ of degree $d$. 
The space $\Gamma(L_Z)$ would be generated by $s\cdot t_1,\ldots,s\cdot t_d$, where $\Gamma(L_1)=\left<s\right>$ and $\Gamma(L_2)=\left<t_1,\ldots,t_d\right>$. 
Thus, for any line $P\subset( \CP^N)^\ast$ the base locus of a pencil $f_P$ would contain $s^{-1}(0)$, especially $f_P$ would not be a Lefschetz pencil, which contradicts the assumption. 
\end{proof}

\noindent
In what follows, we assume that $T_Z$ is not a product of elliptic curves if $d_1=1$.
Note that a generic $Z\in \h_2$ satisfies this assumption by Lemma~\ref{L:nongeneric_subset_moduli}.  
For a homogeneous linear polynomial $q\in \C[X]$ we denote the zero-set of $q$ by $H_q\in (\CP^N)^\ast$. 
We define a subset $W_Z\subset T_Z\times (\CP^N)^\ast$ as follows:
\[
W_Z= \left\{\left(\overline{x},H_{\sum l_{ij}X_{ij}}\right)\in T_Z\times (\CP^N)^\ast~\left|~\sum_{i,j} l_{ij}\vartheta^{ij}_Z(x)=0\text{ and }\sum_{i,j} l_{ij}\frac{\Pa \vartheta^{ij}_Z}{\Pa z_k}(x)=0\hspace{.5em} (k=1,2)\right.\right\}.
\]
We can prove the following lemma by direct calculation. 

\begin{lemma}\label{L:property_W_Z}

Let $P\subset (\CP^N)^\ast$ be a line. 
Suppose that $\overline{x}$ is not a base point of $f_P$. 
The following conditions are equivalent: 

\begin{enumerate}

\item
$(\overline{x},H_{\sum l_{ij}X_{ij}}) \in W_Z$, 

\item
$f_P(\overline{x}) = H_{\sum l_{ij}X_{ij}}$ and $\overline{x}$ is a critical point of $f_P$, 

\item
$\varphi_Z$ is not transverse to $H_{\sum l_{ij}X_{ij}}$ at $\overline{x}$. 

\end{enumerate}

\end{lemma}

In what follows, we assume that $d_1d_2$ is greater than or equal to $3$. 
In this case $\varphi_Z$ is defined on $T_Z$ (see \cite[\S.10.1]{LangeBirkenhake_1992}). 
We define a subset $R_i\subset T_Z$ ($i=0,1,2$) as follows: 
\begin{equation}\label{E:definitionR_i}
R_i = \left\{x\in T_Z~|~ \rank(d\varphi_Z)_x= i
\right\}. 
\end{equation}
We denote the union $\cup_{j\leq i}R_j$ by $S_i$.\label{P:definition_S0}
The set $S_i$ is an analytic subset of $T_Z$, especially the dimension of $S_i$ makes sense.   

\begin{lemma}\label{L:dimension_S_i}

The dimension of $S_1$ is at most $1$. 
Furthermore, if $\NS(T_Z)\cong \Z$ and $S_0\neq \emptyset$, $\dim(S_0)$ is equal to $0$. 

\end{lemma}

\begin{proof}
Since $T_Z$ is compact, the image $\varphi_Z(T_Z)$ is an analytic set by \cite[Theorem 5.8]{Chirka_1989}. 
Assume that $\dim(S_1)=2$. 
Since $T_Z$ is irreducible, $S_1$ is equal to $T_Z$. 
Thus the dimension of $\varphi_Z(T_Z)$ is $1$ by the rank theorem (see \cite[Theorem A2.2.2]{Chirka_1989}). 
By Chow's theorem (see \cite[Theorem 7.1]{Chirka_1989}) $\varphi_Z(T_Z)$ is an algebraic curve. 
If the degree of $\varphi_Z(T_Z)$ is $1$, $\varphi_Z(T_Z)$ would be contained in some $H_{\sum l_{ij}X_{ij}}\in (\CP^N)^\ast$, but it would imply that the section $\sum l_{ij}\vartheta_Z^{ij}$ is the zero-section, contradicting that $\{\vartheta_Z^{ij}\}_{i,j}$ is a basis of $\Gamma(L_Z)$.  
Thus the degree of $\varphi_Z(T_Z)$ is at least $2$, especially $\varphi_Z(T_Z)$ intersects with a generic hyperplane in $\CP^N$ at more than one point. 
This would imply that a generic divisor in $|L_Z|$ is reducible, which contradicts \cite[Theorem 4.3.5]{LangeBirkenhake_1992}. 

The map $\varphi_Z$ is constant on each component of $S_0$. 
Since $\varphi_Z$ is not a constant map, $\dim (S_0)$ is less than $2$. 
Suppose that $\dim(S_0)$ is equal to $1$. 
We take a one-dimensional component $C$ of $S_0$ and denote the point in $\varphi_Z(C)$ by $c\in\CP^N$.
Since there exists a hypersurface $H\in (\CP^N)^\ast$ away from $c$, the intersection number $[C]\cdot H_Z$ is equal to $0$. 
On the other hand, the self-intersection $H_Z^2$ is positive. 
Since both of the class $[C]$ and $H_Z$ are contained in $\NS(T_Z)$, $[C]$ should be equal to $0$ by the assumption, but it cannot happen since $C$ is an algebraic curve.  
\end{proof}

In what follows we assume that $\dim(S_0)$ is equal to $0$ if $S_0$ is not empty. 
By Lemmas~\ref{L:nongeneric_subset_moduli} and \ref{L:dimension_S_i} this assumption holds for generic $Z\in \h_2$. 
Note also that any pencil $f_P$ would not be a Lefschetz pencil if $\dim(S_0)>0$. 
Indeed, any point in a one-dimensional component of $S_0$ is either a base point or a critical point of $f_P$ for any $P$. 

\begin{lemma}\label{L:dimension_W_Z}

The dimension of $W_Z$ is at most $N-1$. 

\end{lemma}

\begin{proof}[Proof of Lemma~\ref{L:dimension_W_Z}]
Let $p_1:W_Z\to T_Z$ be the projection onto the first component. 
By Lemma~\ref{L:property_W_Z}, the restriction $p_1|_{p_1^{-1}(R_i)}$ is a fiber bundle with fiber $\CP^{N-1-i}$. 
Since the dimension of $R_0 =S_0$ is $0$, it is a finite set (see \cite[Proposition 3.4]{Chirka_1989}). 
Thus $p_1^{-1}(R_0)$ is a manifold and its dimension is $N-1$ if it is not empty. 
Since $R_2\subset T_Z$ is open, $p_1^{-1}(R_2)$ is also a manifold and its dimension is $N-1$ provided that $p_1^{-1}(R_0)$ is not empty. 
Suppose that the dimension of the locally analytic set $p_1^{-1}(R_1)$ is greater than $N-1$. 
There exists an open set $U\subset T_Z\times (\CP^N)^\ast$ such that the intersection $p_1^{-1}(R_1)\cap U$ is a manifold with dimension greater than $N-1$. 
Since $\Sing (R_1)$ is nowhere dense in $R_1$, $U\cap p_1^{-1}(\reg(R_1))$ is not empty. 
Since $\reg(R_1)$ is open in $R_1$, $U\cap p_1^{-1}(\reg(R_1))$ is a manifold with dimension greater than $N-1$. 
However, it is impossible since $p_1^{-1}(\reg(R_1))$ is a fiber bundle over $\reg(R_1)$, which is a $1$-dimensional manifold if it is not empty, with fiber $\CP^{N-2}$. 
Thus the dimension of $p_1^{-1}(R_1)$ is at most $N-1$. 
Since $W_Z$ is the union $p_1^{-1}(R_0)\cup p_1^{-1}(R_1)\cup p_1^{-1}(R_2)$, its dimension is at most $N-1$. 
\end{proof}

\noindent
Let $p_2:W_Z \to (\CP^N)^\ast$ be the projection onto the second component and $\mathcal{D}_Z$ the image of $p_2$. 
Since $p_2$ is a proper map, $\mathcal{D}_Z$ is an analytic set of dimension at most $\dim(W_Z)$ (\cite[Theorem 5.8]{Chirka_1989}). 

\begin{lemma}\label{L:dimension_D_Z}

The dimensions of $\mathcal{D}_Z$ and $W_Z$ are both $N-1$. 

\end{lemma}

\begin{proof}
Since $\dim(\mathcal{D}_Z)\leq \dim(W_Z)$, it is enough to prove $\dim(\mathcal{D}_Z)=N-1$ by Lemma~\ref{L:dimension_W_Z}. 
Since $\dim(\mathcal{D}_Z)$ is at most $N-1$, there exists a point $H \in (\CP^N)^\ast$ away from $\mathcal{D}_Z$. 
Let $\pi_H:(\CP^N)^\ast\setminus \{H\} \to \CP^{N-1}$ be the projection from $H$. 
The image $\pi_H(\mathcal{D}_Z)$ is an analytic set since the restriction $\pi_H|_{\mathcal{D}_Z}$ is proper.  
If $\dim(\mathcal{D}_Z)$ were less than $N-1$, the dimension of $\pi_H(\mathcal{D}_Z)$ would also be less than $N-1$.  
Thus we could take a point $x\in \CP^{N-1}$ away from $\pi_H(\mathcal{D}_Z)$. 
We denote the closure $\overline{\pi_H^{-1}(x)}$ by $P_x$, which is a line in $(\CP^N)^\ast$. 
Using Lemma~\ref{L:criterion_basepoint_Lefschetz} we can verify that $f_{P_x}$ is a Lefschetz pencil on $T_Z$ without critical points. 
This would imply that a blow-up of $T_Z$ admits a surface bundle over $\CP^1$, which is impossible. 
\end{proof}

We define the subset $W_Z^0\subset W_Z$ as follows: 
\[
W_Z^0 = \left\{(\overline{x},H_{\sum l_{ij}X_{ij}})\in W_Z~\left|~\det\left(\left(\sum_{i,j}l_{ij}\frac{\Pa^2 \vartheta_Z^{ij}}{\Pa z_k\Pa z_l}(x)\right)_{1\leq k,l\leq 2}\right)\neq 0\right.\right\}. 
\]

\begin{lemma}\label{L:criterion_Lefsing_W_Z}

Suppose that $\overline{x}\in T_Z$ is not a base point of $f_P$ for a line $P\subset (\CP^N)^\ast$. 
The following conditions are equivalent: 

\begin{enumerate}

\item
$(\overline{x},H) \in W_Z^0$, 

\item
$f_P(\overline{x})=H$ and $\overline{x}$ is a Lefschetz type critical point of $f_P$. 

\end{enumerate}

\noindent
Furthermore, if $(\overline{x},H)\in W_Z^0$, then $(\overline{x},H)$ is regular, $\dim_{(\overline{x},H)}W_Z=N-1$ and $p_2$ is an immersion at $(\overline{x},H)$. 

\end{lemma}

\begin{proof}
We can easily prove equivalence of the two conditions (1) and (2) using the criterion in Remark~\ref{R:complexity_Lefsing}. 
The details are left to the reader. 
We take a point $(\overline{x},H_{\sum l_{ij}X_{ij}})\in W_Z^0$. 
One of the values $l_{00},\ldots,l_{ij},\ldots$ is not equal to $0$. 
For simplicity we assume $l_{00}\neq 0$ (we can deal with the other cases in the same way). 
We put $l_{ij}' = l_{ij}/l_{00}$. 
Define a holomorphic map $\Phi:\C^2\times \C^N \to \C^3$ as follows: 
\begin{align*}
&\Phi(x, y_{01},\ldots,y_{ij},\ldots) = \\
&\left(\vartheta^{00}_Z(x)+\sum_{(i,j)\neq (0,0)} y_{ij}\vartheta^{ij}_Z(x),\frac{\Pa \vartheta^{00}_Z}{\Pa z_1}(x)+\sum_{(i,j)\neq (0,0)} y_{ij}\frac{\Pa \vartheta^{ij}_Z}{\Pa z_1}(x),\frac{\Pa \vartheta^{00}_Z}{\Pa z_2}(x)+\sum_{(i,j)\neq (0,0)} y_{ij}\frac{\Pa \vartheta^{ij}_Z}{\Pa z_2}(x)\right). 
\end{align*}
The analytic set germ of $W_Z$ at $(\overline{x},H_{\sum l_{ij}X_{ij}})$ is biholomorphic to that of $\Phi^{-1}(0)$ at $(x,(l_{ij}'))$. 
It is easy to verify by direct calculation that the differential $(d\Phi)_{(x,(l_{ij}'))}$ is the following matrix:
\[
\begin{pmatrix}
0&0& \cdots & \vartheta_Z^{ij}(x)&\cdots \\
\frac{\Pa^2 \vartheta^{00}_Z}{\Pa z_1^2}(x)+\sum_{(i,j)\neq (0,0)} l_{ij}'\frac{\Pa^2 \vartheta^{ij}_Z}{\Pa z_1^2}(x) & \frac{\Pa^2 \vartheta^{00}_Z}{\Pa z_1\Pa z_2}(x)+\sum_{(i,j)\neq (0,0)} l_{ij}'\frac{\Pa^2 \vartheta^{ij}_Z}{\Pa z_1\Pa z_2}(x)&\cdots &\frac{\Pa \vartheta^{ij}_Z}{\Pa z_1}(x)& \cdots \\
\frac{\Pa^2 \vartheta^{00}_Z}{\Pa z_1\Pa z_2}(x)+\sum_{(i,j)\neq (0,0)} l_{ij}'\frac{\Pa^2 \vartheta^{ij}_Z}{\Pa z_1\Pa z_2}(x) & \frac{\Pa^2 \vartheta^{00}_Z}{\Pa z_2^2}(x)+\sum_{(i,j)\neq (0,0)} l_{ij}'\frac{\Pa^2 \vartheta^{ij}_Z}{\Pa z_2^2}(x)&\cdots &\frac{\Pa \vartheta^{ij}_Z}{\Pa z_2}(x)& \cdots 
\end{pmatrix}. 
\]
Since $(\overline{x},H_{\sum l_{ij}X_{ij}})$ is contained in $W_Z^0$, the determinant of the matrix
\[
\begin{pmatrix}
\frac{\Pa^2 \vartheta^{00}_Z}{\Pa z_1^2}(x)+\sum_{(i,j)\neq (0,0)} l_{ij}'\frac{\Pa^2 \vartheta^{ij}_Z}{\Pa z_1^2}(x) & \frac{\Pa^2 \vartheta^{00}_Z}{\Pa z_1\Pa z_2}(x)+\sum_{(i,j)\neq (0,0)} l_{ij}'\frac{\Pa^2 \vartheta^{ij}_Z}{\Pa z_1\Pa z_2}(x) \\
\frac{\Pa^2 \vartheta^{00}_Z}{\Pa z_1\Pa z_2}(x)+\sum_{(i,j)\neq (0,0)} l_{ij}'\frac{\Pa^2 \vartheta^{ij}_Z}{\Pa z_1\Pa z_2}(x) & \frac{\Pa^2 \vartheta^{00}_Z}{\Pa z_2^2}(x)+\sum_{(i,j)\neq (0,0)} l_{ij}'\frac{\Pa^2 \vartheta^{ij}_Z}{\Pa z_2^2}(x)
\end{pmatrix}
\]
is not equal to $0$. 
Since one of the values $\vartheta_Z^{01}(x),\ldots, \vartheta_Z^{ij}(x),\ldots$ is not equal to zero, $(x,(l_{ij}'))$ is a regular point of $\Phi$. 
Thus the analytic set germ of $W_Z$ at $(\overline{x},H_{\sum l_{ij}X_{ij}})$ is an $(N-1)$-dimensional submanifold germ, which means $(\overline{x},H)$ is regular and $\dim_{(\overline{x},H)}W_Z=N-1$. 
The tangent space $T_{(\overline{x},H)}W_Z$ is identified with $\Ker(d\Phi)_{(x,(l_{ij}'))}\subset T_x \C^2\oplus T_{(l_{ij}')}\C^N$. 
Under this identification, $(dp_2)_{(\overline{x},H)}$ coincides with the restriction of the projection $T_x \C^2\oplus T_{(l_{ij}')}\C^N\to T_{(l_{ij}')}\C^N$. 
Suppose that $c_1\left(\frac{\Pa}{\Pa z_1}\right)+c_2\left(\frac{\Pa}{\Pa z_2}\right)$ is contained in $\Ker(d\Phi)_{(x,(l_{ij}'))}$. 
The following equality holds for $k=1,2$: 
{\allowdisplaybreaks
\begin{align*}
&c_1\left(\frac{\Pa^2 \vartheta^{00}_Z}{\Pa z_1\Pa z_k}(x)+\sum_{(i,j)\neq (0,0)} l_{ij}'\frac{\Pa^2 \vartheta^{ij}_Z}{\Pa z_1\Pa z_k}(x)\right)
+c_2\left(\frac{\Pa^2 \vartheta^{00}_Z}{\Pa z_2\Pa z_k}(x)+\sum_{(i,j)\neq (0,0)} l_{ij}'\frac{\Pa^2 \vartheta^{ij}_Z}{\Pa z_2\Pa z_k}(x)\right)=0 \\
\Leftrightarrow & e_k\cdot\left(\frac{\Pa^2\vartheta^{00}_Z}{\Pa z_k\Pa z_l}(x)+\sum_{(i,j)\neq (0,0)}l_{ij}'\frac{\Pa^2\vartheta_Z^{ij}}{\Pa z_k\Pa z_l}(x)\right)_{1\leq k,l\leq 2}\cdot \begin{pmatrix}c_1\\c_2\end{pmatrix}=0, 
\end{align*}
}
where $e_k$ is the unit vector in $\C^2$. 
Since the determinant of $\left(\frac{\Pa^2\vartheta^{00}_Z}{\Pa z_k\Pa z_l}(x)+\sum_{(i,j)\neq (0,0)}l_{ij}'\frac{\Pa^2\vartheta_Z^{ij}}{\Pa z_k\Pa z_l}(x)\right)_{1\leq k,l\leq 2}$ is equal to $0$, $(c_1,c_2)$ must be equal to $(0,0)$. 
Thus $(dp_2)_{(\overline{x},H_{\sum l_{ij}X_{ij}})}$ is injective. 
\end{proof}

The proof of Lemma~\ref{L:criterion_Lefsing_W_Z} yields the following corollary: 

\begin{corollary}\label{C:property_image_dp_2}

Let $(\overline{x},H)\in W_Z^0$ and $P\subset (\CP^N)^\ast$ a line containing $H$. 
Then $T_HP$ is contained in $(dp_2)_{(\overline{x},H)}(T_{(\overline{x},H)}W_Z)$ if and only if $\varphi_Z^{-1}(H')$ contains $\overline{x}$ for any $H'\in P$. 

\end{corollary}

\begin{proof}
Let $H = H_{\sum l_{ij} X_{ij}}$ and suppose that $l_{kl}$ is not $0$. 
Let $H'=H_{\sum m_{ij} X_{ij}}\in P\setminus \{H\}$ be another line.
Since each line in $P$ has $\sum (\lambda_0l_{ij}+\lambda_1m_{ij}) X_{ij}$ as a defining equation, it suffices to show that $T_HP\subset (dp_2)_{(\overline{x},H)}(T_{(\overline{x},H)}W_Z)$ if and only if $\varphi_Z(\overline{x})$ is on $H'$. 
Using the affine coordinates in $\{H_{\sum y_{ij}X_{ij}}\in (\CP^N)^\ast~|~y_{kl}\neq 0\}$, we can identify the tangent space $T_{H}(\CP^N)^\ast$ with $\C^N$. 
Under this identification, the following equality holds: 
\[
T_HP = \left<\sum_{(i,j)\neq (k,l)} (m_{ij}-\frac{l_{ij}}{l_{kl}}m_{kl}) \left(\frac{\Pa}{\Pa y_{ij}}\right)\right>. 
\]
On the other hand, the proof of Lemma~\ref{L:criterion_Lefsing_W_Z} yields the following equality: 
\[
(dp_2)_{(\overline{x},H)}(T_{(\overline{x},H)}W_Z)= \left\{\sum_{(i,j)\neq(k,l)}r_{ij}\left(\frac{\Pa}{\Pa y_{ij}}\right)~\left|~\sum_{(i,j)\neq(k,l)}r_{ij}\vartheta_Z^{ij}(x)=0\right.\right\}. 
\]
Thus $T_HP\subset (dp_2)_{(\overline{x},H)}(T_{(\overline{x},H)}W_Z)$ if and only if $\sum_{(i,j)\neq(k.l)}(m_{ij}-\frac{l_{ij}}{l_{kl}}m_{kl})\vartheta_Z^{ij}(x)=0$. 
Since $\sum l_{ij}\vartheta_Z^{ij}(x)$ is equal to $0$, the latter condition is equivalent to the following equality: 
\[
\sum m_{ij} \vartheta_Z^{ij}(x)=0, 
\]
which is equivalent to the condition $\varphi_Z(\overline{x})\in H'$. 
\end{proof}

We define two subsets $\mathcal{D}_Z'$ and $\mathcal{D}_Z''$ of $\mathcal{D}_Z$ as follows: 
\begin{align*}
\mathcal{D}_Z' = & p_2(W_Z\setminus W_Z^0), \\
\mathcal{D}_Z'' = & \left\{H \in \mathcal{D}_Z\setminus \mathcal{D}_Z'~\left|~\sharp(p_2^{-1}(H))\neq 1\right.\right\}. 
\end{align*}
Since $W_Z\setminus W_Z^0$ is analytic and $p_2$ is proper, $\mathcal{D}_Z'$ is analytic by \cite[Theorem 5.8]{Chirka_1989}.

\begin{lemma}\label{L:positive_codimension_D''}

The subset $\mathcal{D}_Z''$ is contained in an analytic subset of $\mathcal{D}_Z$ with dimension at most $N-2$. 

\end{lemma}

\begin{proof}
The set $\Sing(\mathcal{D}_Z)$ of singular points of $\mathcal{D}_Z$ is an analytic set with dimension at most $N-2$ by \cite[Theorem 5.2.2]{Chirka_1989}. 
Thus it is enough to show that $\mathcal{D}_Z''$ is contained in $\Sing(\mathcal{D}_Z)$. 
For $H\in \mathcal{D}_Z''$ the preimage $p_2^{-1}(H)$ is contained in $W_Z^0$. 
By Lemma~\ref{L:criterion_Lefsing_W_Z}, $p_2$ is an immersion at any point in $W_Z^0$, in particular $p_2^{-1}(H)$ is a finite set and the analytic set germ of $\mathcal{D}_Z$ at $H$ contains $|p_2^{-1}(H)|>1$ submanifold components. 
By Lemma~\ref{C:property_image_dp_2}, any two of such components intersect transversely at $H$. 
In particular, the analytic set germ of $\mathcal{D}_Z$ at $H$ would not be a submanifold germ. 
Hence $H$ is contained in $\Sing(\mathcal{D}_Z)$. 
\end{proof}

\begin{theorem}\label{T:criterion_LP}

For a line $P\subset (\CP^N)^\ast$, the following conditions are equivalent: 

\begin{enumerate}

\item
the map $f_P$ is a Lefschetz pencil, 

\item
the line $P$ is away from $\mathcal{D}_Z'\cup \mathcal{D}_Z''$ and $(dp_2)_{(\overline{x},H)}(T_{(\overline{x},H)}W_Z)+T_{H}P=T_H(\CP^N)^\ast$ for any $(\overline{x},H) \in p_2^{-1}(\mathcal{D}_Z\cap P)$. 

\end{enumerate}

\end{theorem}

\begin{proof}
By Lemma~\ref{L:criterion_Lefsing_W_Z}, $P$ is away from $\mathcal{D}_Z'\cup \mathcal{D}_Z''$ if and only if $f_P$ satisfies the conditions (1) and (2) in p.\pageref{P:definition_LP}. 
Thus it suffices to show that $T_{H}P\subset (dp_2)_{(\overline{x},H)}(T_{(\overline{x},H)}W_Z)$ for some $(\overline{x},H)\in p_2^{-1}(\mathcal{D}_Z\cap P)$ if and only if $f_P$ has a base point at which the condition (3) does not hold.
By Corollary~\ref{C:property_image_dp_2}, $T_{H}P\subset (dp_2)_{(\overline{x},H)}(T_{(\overline{x},H)}W_Z)$ if and only if $\varphi_Z^{-1}(H)$ contains $\overline{x}$ for any $H\in P$. 
Thus, if $T_{H}P$ is contained in $(dp_2)_{(\overline{x},H)}(T_{(\overline{x},H)}W_Z)$ for some $(\overline{x},H)\in p_2^{-1}(\mathcal{D}_Z\cap P)$, $\overline{x}$ is a base point of $f_P$ and the condition (3) does not hold at $\overline{x}$ by Lemma~\ref{L:criterion_basepoint_Lefschetz}. 
Suppose conversely that $f_P$ has a base point at which the condition (3) does not hold.
Let $H_{\sum l_{ij}^0X_{ij}}, H_{\sum l_{ij}^1X_{ij}}$ be distinct hyperplanes in $P$ and $s_k = \sum l_{ij}^k\vartheta_Z^{ij}$ a section of $L_Z$. 
By Lemma~\ref{L:criterion_basepoint_Lefschetz} and the assumption there exist a base point $x$ of $f_P$ and $\lambda_k \in \C$ such that $\lambda_0(ds_0)_x+\lambda_1(ds_1)_x=0$. 
It is easily verify that $(x, \tilde{H}=H_{\sum (\lambda_0l_{ij}^0 + \lambda_1l_{ij}^1)X_{ij}})$ is contained in $W_Z$. 
Since $x$ is a base point of $f_P$, $x$ is contained in $\varphi_Z^{-1}(H)$ for any $H\in P$. 
Hence $T_{\tilde{H}}P$ is contained in $(dp_2)_{(x,\tilde{H})}(T_{(x,\tilde{H})}W_Z)$ by Corollary~\ref{C:property_image_dp_2}. 
\end{proof}

\noindent
By Lemma~\ref{L:criterion_Lefsing_W_Z} the set $\mathcal{D}_Z^0=\mathcal{D}_Z\setminus (\mathcal{D}_Z'\cup \mathcal{D}_Z'')$ is a submanifold of $(\CP^N)^\ast$ of dimension $N-1$. 
We can easily deduce the following corollary from Theorem~\ref{T:criterion_LP}. 

\begin{corollary}\label{C:criterion_LP_projection}

Suppose that $H\in (\CP^N)^\ast$ is away from $\mathcal{D}_Z$. 
Let $\pi_H:\mathcal{D}_Z\to \CP^{N-1}$ be the projection from $H$. 
The map $f_{\overline{\pi_H^{-1}(x)}}:T_Z \dashedrightarrow \overline{\pi_H^{-1}(x)}$ is a Lefschetz pencil if and only if $\pi_H^{-1}(x)$ is away from $\mathcal{D}_Z'\cup \mathcal{D}_Z''$ and $x$ is a regular value of $\pi_H|_{\mathcal{D}_Z^0}$. 

\end{corollary}

\noindent
We denote the set of projective lines in $(\CP^N)^\ast$ by $\mathcal{L}$, which can be identified with the Grassmannian manifold $G_2(\C^{N+1})$. 

\begin{corollary}\label{C:genericness_injectivity}

Suppose that $f_P:T_Z\dashedrightarrow P$ satisfies the conditions (2) and (3) in p.\pageref{P:definition_LP}. 
For any open neighborhood $U\subset \mathcal{L}$ of $P$, there exists a line $P'\in U$ such that $f_{P'}$ is a Lefschetz pencil. 

\end{corollary}

\begin{proof}
The proof of Theorem~\ref{T:criterion_LP} implies that $P$ is away from $\mathcal{D}_Z'$ and $(dp_2)_{(\overline{x},H)}(T_{(\overline{x},H)}W_Z)+T_{H}P=T_H(\CP^N)^\ast$ for any $(\overline{x},H) \in p_2^{-1}(\mathcal{D}_Z\cap P)$. 
Let $H\in P$ be away from $\mathcal{D}_Z$. 
We take an open neighborhood $V\subset \CP^{N-1}$ of $\pi_H(H')$ for $H'\in P\setminus\{H\}$ so that $\overline{\pi_H^{-1}(x)}$ is in $U$ for any $x\in V$. 
Since both of the images $\pi_H(\mathcal{D}_Z')$ and $\pi_H(\Sing(\mathcal{D}_Z))$ are analytic sets and the dimension of $\pi_H(\Sing(\mathcal{D}_Z))$ is at most $N-2$ by Lemma~\ref{L:positive_codimension_D''}, there exists a point $y\in V$ away from $\pi_H(\mathcal{D}_Z'\cup \mathcal{D}_Z'')$ (note that $\mathcal{D}_Z''$ is contained in $\Sing(\mathcal{D}_Z)$). 
Furthermore, since $\pi_H(H')$ is a regular value of $\pi_H|_{\mathcal{D}_Z^0}$ we can take such a point $y\in V$ so that $y$ is a regular value of $\pi_H|_{\mathcal{D}_Z^0}$ by Sard's theorem for holomorphic maps (see \cite[Corollary V.1.2]{Lojasiewicz_1991}). 
By Corollary~\ref{C:criterion_LP_projection} the map $f_{\overline{\pi_H^{-1}(y)}}$ is a Lefschetz pencil. 
\end{proof}

\subsection{Existence of paths connecting two Lefschetz pencils}

Let $d_1,d_2$ be positive integers with $d_1|d_2$. 
Throughout this subsection, we assume $d_1d_2\geq 3$. 
As we observed in the beginning of Subsection~\ref{S:condition_pencil_Lefschetz}, any holomorphic Lefschetz pencil on $T^4$ with genus $d_1d_2+1$ and divisibility $d_1$ is isomorphic to $f_P:T_Z\dashedrightarrow P$ for $Z\in \h_2$ and $P\in \mathcal{L}$. 
The aim of this subsection is to take a path in $\h_2\times \mathcal{L}$ which connects two points associated with two given Lefschetz pencils. 

\begin{lemma}\label{L:openness_SpaceOfLP}

The subset of $\h_2\times \mathcal{L}$ consisting of points which yield Lefschetz pencils is open. 

\end{lemma}

\begin{proof}
Let $S_2(\C^{N+1})$ be the space of pairs of $\C$-linearly independent vectors in $\C^{N+1}$ endowed with the relative topology of $(\C^{N+1})^2$ (that is, $S_2(\C^{N+1})$ is a non-compact Stiefel manifold). 
Since the quotient map $\pi:S_2(\C^{N+1})\to \mathcal{L}$ is continuous and open, it is enough to show that the set 
\[
\left\{\left.(Z,(v_0,v_1))\in \h_2\times S_2(\C^{N+1})~\right|~(Z,\pi(v_0,v_1))\text{ yields a Lefschetz pencil}\right\}
\]
is an open subset. 

For $(Z,(v_0,v_1))\in \h_2\times S_2(\C^{N+1})$, we define diffeomorphisms $\phi_{v_0,v_1}:\CP^1\to \pi(v_0,v_1)$ and $\psi_Z:\R^4\to \C^2$ as follows: 
\begin{align*}
& \phi_{v_0,v_1}([l_0:l_1]) = H_{\sum (l_0v_0^{ij}+l_1v_1^{ij})X_{ij}}, \\
& \psi_Z(x) = (Z,D)x, 
\end{align*}
where we put $v_k = (\ldots,v_k^{ij},\ldots)$ and $D=\begin{pmatrix}
d_1&0\\0&d_2
\end{pmatrix}$. 
By Lemma~\ref{L:criterion_Lefsing_W_Z}, $f_{\pi(v_0,v_1)}(\psi_Z(x)) = \phi_{v_0,v_1}([l_0:l_1])$ and $\psi_Z(x)$ is a Lefschetz critical point if and only if the following three conditions are satisfied: 
{\allowdisplaybreaks
\begin{align*}
&\sum_{i,j} (l_0v_0^{ij}+l_1v_1^{ij})\vartheta^{ij}_Z(\psi_Z(x))=0, \\
&\sum_{i,j} (l_0v_0^{ij}+l_1v_1^{ij})\frac{\Pa \vartheta^{ij}_Z}{\Pa z_k}(\psi_Z(x))=0\hspace{.5em} (k=1,2), \\
&\det\left(\left(\sum_{i,j}(l_0v_0^{ij}+l_1v_1^{ij})\frac{\Pa^2 \vartheta_Z^{ij}}{\Pa z_k\Pa z_l}(\psi_Z(x))\right)_{1\leq k,l\leq 2}\right)\neq 0. 
\end{align*}
}
Furthermore, by Corollary~\ref{C:property_image_dp_2} the following two conditions are equivalent: 
\begin{itemize}

\item
$T_{\phi_{v_0,v_1}([l_0:l_1])}\pi(v_0,v_1)\subset(dp_2)_{(\overline{\psi_Z(x)},\phi_{v_0,v_1}([l_0:l_1]))}(T_{(\overline{\psi_Z(x)},\phi_{v_0,v_1}([l_0:l_1]))}W_Z)$ and $(\psi_Z(x),\pi(v_0,v_1))$ is contained in $W_Z^0$, 

\item
$(Z,(v_0,v_1))$ satisfies the three conditions above and $\sum_{i,j}v_k^{ij} \vartheta_Z^{ij}(\psi_Z(x))=0$ for $k=0,1$.

\end{itemize}
We define a map $\Phi(Z,(v_0,v_1)):\R^4\times \CP^1 \to \C^6$ as follows: 
\begin{align*}
& \Phi(Z,(v_0,v_1))(x,[l_0:l_1]) = \\
&\left(\sum_{i,j} (l_0v_0^{ij}+l_1v_1^{ij})\vartheta^{ij}_Z(\psi_Z(x)),\sum_{i,j} (l_0v_0^{ij}+l_1v_1^{ij})\frac{\Pa \vartheta^{ij}_Z}{\Pa z_1}(\psi_Z(x)), \sum_{i,j} (l_0v_0^{ij}+l_1v_1^{ij})\frac{\Pa \vartheta^{ij}_Z}{\Pa z_2}(\psi_Z(x)), \right.\\
&\left.\det\left(\left(\sum_{i,j}(l_0v_0^{ij}+l_1v_1^{ij})\frac{\Pa^2 \vartheta_Z^{ij}}{\Pa z_k\Pa z_l}(\psi_Z(x))\right)_{1\leq k,l\leq 2}\right),\sum_{i,j}v_0^{ij} \vartheta_Z^{ij}(\psi_Z(x)),\sum_{i,j}v_1^{ij} \vartheta_Z^{ij}(\psi_Z(x))
\right). 
\end{align*}
Let $V_1= \{0\}\times \{0\}\times \C^2, V_2=\{0\}\times \C\times \{0\} \subset \C^3\times\C\times \C^2 = \C^6$ and $\Delta=\{(x_1,x_2,x_3,x_4)\in \R^4~|~ |x_i| \leq 1/2\}$.  
By double-periodicity of the Theta functions, $f_{\pi(v_0,v_1)}:T_Z \dashedrightarrow \pi(v_0,v_1)$ satisfies the conditions (2) and (3) in p.\pageref{P:definition_LP} if and only if $\Phi(Z,(v_0,v_1))(\Delta\times \CP^1)\cap (V_1\cup V_2)$ is empty. 
Since $V_1$ and $V_2$ are closed and $\Delta\times \CP^1$ is compact, the subset 
\[
W_0=\left\{(Z,(v_0,v_1))\in \h_2\times S_2(\C^{N+1})~|~\Phi(Z,(v_0,v_1))(\Delta\times \CP^1)\cap (V_1\cup V_2)=\emptyset \right\}
\]
is open. 

We take a point $(Z,(v_0,v_1))\in W_0$ and suppose that $f_{\pi(v_0,v_1)}:T_Z\dashedrightarrow \pi(v_0,v_1)$ is a Lefschetz pencil. 
We put $\phi_{v_0,v_1}^{-1}(f_{\pi(v_0,v_1)}(\Crit(f_{\pi(v_0,v_1)}))) = \{y_1,\ldots,y_n\}$ ($n=6d_1d_2$ is the number of critical points of $f_{\pi(v_0,v_1)}$).  
For each $i$, we take a disk neighborhood $D_i$ of $y_i$ in $\CP^1$ so that $D_i\cap D_j=\emptyset$ if $i\neq j$. 
The former three components of $\Phi(Z,(v_0,v_1))(x,[l_0:l_1])$ never be equal to $0$ simultaneously for $(x,[l_0:l_1]) \in \Delta\times (\CP^1\setminus \sqcup_i D_i)$. 
Since $\Delta\times (\CP^1\setminus \sqcup_i D_i)$ is compact, there exists an open neighborhood $U\subset W_0$ of $(Z,(v_0,v_1))$ such that for any $(Z',(v_0',v_1'))\in U$ the former three components of $\Phi(Z',(v_0',v_1'))(x,[l_0:l_1])$ never be equal to $0$ simultaneously for $(x,[l_0:l_1]) \in \Delta\times (\CP^1\setminus \sqcup_i D_i)$. 
Thus, all the critical values of $f_{\pi(v_0,v_1)}:T_Z\dashedrightarrow \pi(v_0,v_1)$ are contained in the disjoint union $\sqcup_iD_i$. 
Furthermore, we can make $U$ sufficiently small so that the conjugacy class of a local monodromy of $f_{\pi(v_0',v_1')}:T_{Z'}\dashedrightarrow \pi(v_0',v_1')$ around $D_i$ are independent of the choice of $(Z',(v_0',v_1'))\in U$. 
In particular for any $(Z',(v_0',v_1'))\in U$ and $i$, the preimage $f_{\pi(v_0',v_1')}^{-1}(D_i)$ contains at least one critical point of $f_{\pi(v_0',v_1')}:T_{Z'}\dashedrightarrow \pi(v_0',v_1')$. 
Since $f_{\pi(v_0',v_1')}$ has genus $g=d_1d_2+1$ and $b=2d_1d_2$ base points, $0=\chi(T_{Z'})$ is equal to $4-4g+n'-b = n'-6d_1d_2$, where $n'$ is the number of critical points of $f_{\pi(v_0',v_1')}$. 
Thus $n'$ is equal to $n$ and $f_{\pi(v_0',v_1')}^{-1}(D_i)$ contains exactly one critical point for each $i$, which implys that $f_{\pi(v_0',v_1')}$ satisfies the condition (1) for any $(Z',(v_0',v_1'))\in U$.
We can eventually conclude that the set of points in $\h_2\times S_2(\C^{N+1})$ giving rise to a Lefschetz pencil is open. 
\end{proof}

\noindent
The proof of Lemma~\ref{L:openness_SpaceOfLP} implies the following corollaries:

\begin{corollary}\label{C:isomLP_component}

Let $\mathcal{W}$ be the set of $(Z,P)\in \h_2\times \mathcal{L}$ such that $f_P:T_Z\dashedrightarrow P$ is a Lefschetz pencil and $(Z_i,P_i)\in \mathcal{W}$. 
The two pencils $f_{P_0}$ and $f_{P_1}$ are isomorphic if $(Z_0,P_0)$ and $(Z_1,P_1)$ are contained in the same connected component of $\mathcal{W}$. 

\end{corollary}

\begin{proof}
Suppose that $(Z_0,P_0)$ and $(Z_1,P_1)$ are contained in the same connected component of $\mathcal{W}$. 
The proof of Lemma~\ref{L:openness_SpaceOfLP} shows that the monodromy factorizations of $f_{P_0}$ and $f_{P_1}$ are Hurwitz equivalent. 
Thus $f_{P_0}$ and $f_{P_1}$ are isomorphic by Theorem~\ref{T:isom_Hurwitzeq}. 
\end{proof}

\begin{corollary}\label{C:nonexistence_sepVC}

A genus-$g\geq 4$ holomorphic Lefschetz pencil on $T^4$ does not have a reducible fiber. 

\end{corollary}

\begin{proof}
Suppose that $f_P:T_Z\dashedrightarrow P$ has a reducible fiber $F=F_1+F_2$. 
By Lemma~\ref{L:nongeneric_subset_moduli}, Lemma~\ref{L:openness_SpaceOfLP} and Corollary~\ref{C:isomLP_component}, we may assume that $\NS(T_Z)$ is isomorphic to $\Z$ without loss of generality. 
Since $[F_i]\in \NS(T_Z)$ for $i=1,2$, $[F_i]=n_i \alpha$ for some $n_i\in \Z$ and $\alpha\in H^2(T_Z;\Z)$. 
Since $F_1$ and $F_2$ intersect on one point (which is a Lefschetz singularity of $f_P$), $n_1n_2\alpha^2$ is equal to $1$. 
Thus $2g-2 = [F]^2$ must be equal to $4$, which contradicts the assumption. 
\end{proof}

\begin{remark}\label{R:nonhol_reduciblefiber}

We can deduce from Corollary~\ref{C:nonexistence_sepVC} that a genus-$g\geq 4$ Lefschetz pencil on the four-torus $T^4$ with reducible fibers cannot be holomorphic. 

\end{remark}

\begin{lemma}\label{L:connectedness_lines}

Suppose that the dimension of $\mathcal{D}_Z'\subset \mathcal{D}_Z$ is at most $N-2$. 
Then the set 
\[
\mathcal{L}_Z^0=\{P\in \mathcal{L}~|~f_P:T_Z\dashedrightarrow P\text{ is a Lefschetz pencil}\}
\]
is non-empty and path-connected.  

\end{lemma}

\begin{proof}
Assume that $H\in (\CP^N)^\ast$ is away from $\mathcal{D}_Z$. 
By the assumption the image $\pi_{H}(\mathcal{D}_Z'\cup \mathcal{D}_Z'')$ is contained in an analytic set with dimension at most $N-2$. 
Furthermore, the set of critical values of $\pi_H|_{\mathcal{D}_Z^0}$ is a countable union of submanifolds of $\CP^{N-1}$ with dimension at most $N-2$ by Sard's theorem (see \cite[Corollary V.1.2]{Lojasiewicz_1991}). 
Thus, the image $\pi_{H}(\mathcal{D}_Z'\cup \mathcal{D}_Z''\cup \Crit(\pi_H|_{\mathcal{D}_Z^0}))$ is contained in a countable union of submanifolds of $\CP^{N-1}$. 
In particular we can take a point $x\in \CP^{N-1}$ away from $\pi_{H}(\mathcal{D}_Z'\cup \mathcal{D}_Z''\cup \Crit(\pi_H|_{\mathcal{D}_Z^0}))$. 
By Corollary~\ref{C:criterion_LP_projection} the line $\overline{\pi_H^{-1}(x)}$ is contained in $\mathcal{L}_Z^0$, especially $\mathcal{L}_Z^0$ is not empty. 

Let $P_0,P_1\in \mathcal{L}_Z^0$.
We take distinct hyperplanes $H_0,H_0'\in P_0$ away from $\mathcal{D}_Z$. 
Since $\pi_{H_0}(H_0')$ is not contained in  $\pi_{H_0}(\mathcal{D}_Z'\cup \mathcal{D}_Z''\cup \Crit(\pi_{H_0}|_{\mathcal{D}_Z^0}))$ by Corollary~\ref{C:criterion_LP_projection}, there exists an open neighborhood $U_0\subset (\CP^N)^\ast$ of $H_0'$ such that it is away from $\mathcal{D}_Z$ and $\pi_{H_0}(U_0)\cap \pi_{H_0}(\mathcal{D}_Z'\cup \mathcal{D}_Z''\cup \Crit(\pi_{H_0}|_{\mathcal{D}_Z^0})) = \emptyset$. 
We take a hyperplane $H_1\in P_1\setminus \mathcal{D}_Z$. 
Since $\pi_{H_1}$ is an open map, $\pi_{H_1}(U_0)$ is an open subset of $\CP^{N-1}$. 
In particular, the set $\pi_{H_1}(U_0) \setminus (\pi_{H_1}(\mathcal{D}_Z'\cup \mathcal{D}_Z''\cup \Crit(\pi_{H_1}|_{\mathcal{D}_Z^0})))$ is not empty. 
Let $x_1'$ be a point in this set and $P_1' = \overline{\pi_{H_1}^{-1}(x_1')}$.  
By the transversality theorem (\cite[Theorem 4.9]{GG_1973}), we can take a path in $\CP^{N-1}$ away from $\pi_{H_1}(\mathcal{D}_Z'\cup \mathcal{D}_Z''\cup \Crit(\pi_{H_1}|_{\mathcal{D}_Z^0}))$ which connects $x_1\in \pi_{H_1}(P_1)$ and $x_1'$. 
Taking preimages of this path, we can take a path in $\mathcal{L}_Z^0$ from $P_1$ to $P_1'$. 
We take a hyperplane $H_1'\in P_1'\cap U_0$ and denote the line going through $H_0$ and $H_1'$ by $P_1''$. 
Since $f_{P_1''}$ is a Lefschetz pencil, $\pi_{H_1'}(P_1''\setminus \{H_1''\})$ is not contained in $\pi_{H_1'}(\mathcal{D}_Z'\cup \mathcal{D}_Z''\cup \Crit(\pi_{H_1'}|_{\mathcal{D}_Z^0}))$. 
Thus, in the same way as above we can take a path in $\mathcal{L}_Z^0$ connecting $P_1'$ and $P_1''$. 
Lastly, since $\pi_{H_0}(P_1''\setminus \{H_0\})$ is not contained in $\pi_{H_0}(\mathcal{D}_Z'\cup \mathcal{D}_Z''\cup \Crit(\pi_{H_0}|_{\mathcal{D}_Z^0}))$, we can take a path in $\mathcal{L}_Z^0$ connecting $P_1''$ and $P_0$. 
\end{proof}

\begin{lemma}\label{L:existence_pathLP}

Suppose that the following condition \textnormal{($\ast$)} (for $(d_1,d_2)$) is satisfied: 

\begin{enumerate}

\item[\textnormal{($\ast$)}]
The set $\{Z\in \h_2~|~\dim(\mathcal{D}_Z')\geq N-1\}$ is contained in a countable union of analytic sets with positive codimensions. 

\end{enumerate}

\noindent
Then $\mathcal{W}$ defined in Corollary~\ref{C:isomLP_component} is path-connected. 

\end{lemma}

\begin{proof}
For $(Z_i,P_i)\in \mathcal{W}$ ($i=0,1$), we first take a path $\beta:[0,1]\to \h_2$ which satisfies the following properties: 

\begin{itemize}

\item
$\beta(i)=Z_i$ for $i=0,1$, 

\item
$\dim(\mathcal{D}_{\beta(t)}')<N-1$ for any $t\in (0,1)$,

\item
the group $\NS(T_{\beta(t)})$ is the infinite cyclic group for any $t\in (0,1)$. 

\end{itemize}

\noindent
We can take such a path by the transversality theorem, the assumption and Lemma~\ref{L:nongeneric_subset_moduli}. 
We may further assume that $T_{\beta(t)}$ is not a product of elliptic curves by Lemma~\ref{L:nongeneric_subset_moduli} and an analytic set $S_0\subset T_{\beta(t)}$ defined in p.\pageref{P:definition_S0} has dimension $0$ by Lemma~\ref{L:dimension_S_i}.
By Lemma~\ref{L:openness_SpaceOfLP} there exists $\varepsilon >0$ such that $f_{P_0}:Z_t\dashedrightarrow P_0$ and $f_{P_1}:Z_{1-t}\dashedrightarrow P_1$ are both Lefschetz pencils for $t\in [0,\varepsilon]$. 
We will prove that there exists a piecewise smooth path $\gamma:[t_0,t_1]\to \h_2 \times \mathcal{L}$ which satisfy the following conditions: 

\begin{enumerate}

\item
$\gamma_2(t_0) = P_0$, 

\item
$f_{\gamma_2(t)}:T_{\gamma_1(t)} \dashedrightarrow \gamma_2(t)$ is a Lefschetz pencil for any $t\in[t_0,t_1]$, 

\item
there exists a monotone non-decreasing function $\delta:[t_0,t_1]\to [\varepsilon,1-\varepsilon]$ such that $\delta(t_0)=\varepsilon$ and $\gamma_1(t)=\beta(\delta(t))$ for any $t\in [t_0,t_1]$, 

\item
$\delta(t_1)=1-\varepsilon$ and $\gamma_2(t_1)=P_1$,

\end{enumerate}
where $\gamma_i(t)$ is the $i$-th component of $\gamma(t)$.
In order to prove existence of such a path, we define a value $T\leq 1-\varepsilon$ as follows: 
\[
T = \sup \left\{t \leq 1-\varepsilon ~\left|~{}^\exists \gamma:[t_0,t_1]\to \mathcal{L}\text{ s.t. }\gamma\text{ satisfies (1), (2) and (3) and }\gamma_1(t_1)=\beta(t) \right.\right\}. 
\]
The value $T$ is equal to $1-\varepsilon$. 
To see this, suppose that $T$ is less than $1-\varepsilon$. 
By Lemma~\ref{L:connectedness_lines}, there exists a line $P\in \mathcal{L}$ such that $f_P:Z_{\beta(T)}\dashedrightarrow P$ is a Lefschetz pencil. 
By Lemma~\ref{L:openness_SpaceOfLP} we can take $\varepsilon'>0$ such that $f_P:Z_{\beta(t)}\dashedrightarrow P$ is a Lefschetz pencil for any $t\in [T-\varepsilon',T+\varepsilon']$. 
By the definition of $T$, there exists a path $\gamma:[t_0,t_1]\to \mathcal{L}$ and $s \in (T-\varepsilon',T]$ such that $\gamma$ satisfies the conditions (1)--(3) and $\gamma_1(t_1)=\beta(s)$. 
Using Lemma~\ref{L:connectedness_lines} we can take a path in $\mathcal{L}_{\beta(s)}^0$ which connects $\gamma_2(s)$ and $P$. 
We can then extend a path $\gamma$ so that the extended path $\widetilde{\gamma}$ satisfies the conditions (1)--(3) and $\widetilde{\gamma}(t_1)=T+\varepsilon'$, which contradicts the definition of $T$. 
Thus we can conclude that $T=1-\varepsilon$. 
In the same way as above, we can then take a path $\gamma$ which satisfies the conditions (1)--(4). 
Eventually we can obtain a path connecting $(Z_0,P_0)$ and $(Z_1,P_1)$ by concatenate the three paths $[0,\varepsilon] \ni t\mapsto (\beta(t),P_0)$, $\gamma$ obtained above and $[1-\varepsilon,1] \ni t\mapsto (\beta(t),P_1)$. 
\end{proof}

\noindent
We can eventually deduce the following theorem from Corollary~\ref{C:isomLP_component} and Lemma~\ref{L:existence_pathLP}:

\begin{theorem}\label{T:uniqueness_isomLP}

Suppose that the condition \textnormal{($\ast$)} in Lemma~\ref{L:existence_pathLP} is satisfied. 
Then any two holomorphic Lefschetz pencils on $T^4$ with genus-$(d_1d_2+1)$ and divisibility $d_1$ are isomorphic. 

\end{theorem}

\subsection{The condition ($\ast$) for a pair $(d_1,d_2)$}

As we proved in the last subsection, any two holomorphic Lefschetz pencils with genus-$(d_1d_2+1)$ and divisibility $d_1$ are isomorphic \textit{provided that the condition \textnormal{($\ast$)} in Lemma~\ref{L:existence_pathLP} is satisfied}. 
In this subsection we discuss for which pair $(d_1,d_2)$ satisfies this condition. 

\begin{lemma}\label{L:conditionast_d1d2geq5}

The condition \textnormal{($\ast$)} holds if $d_1d_2 \geq 5$. 

\end{lemma}

\begin{proof}
We first observe that if $d_1d_2\geq 5$, the set of $Z\in \h_2$ such that $L_Z$ is not very ample is contained in an algebraic set with positive codimension (see \cite[Theorem 4.5.1, \S.10.1, Theorem 10.4.1]{LangeBirkenhake_1992}).
Thus, it suffices to show that $\mathcal{D}_Z'$ has dimension at most $N-2$ when $L_Z$ is very ample. 

In what follows we assume that $L_Z$ is very ample. 
We will prove that the differential $(d\Phi)_{(x,(\ldots,l_{ij}',\ldots))}$ is surjective for any $(\overline{x},H_{\sum l_{ij}X_{ij}})\in W_Z$, where $\Phi:\C^2\times \C^N\to \C^3$ is defined in the proof of Lemma~\ref{L:criterion_Lefsing_W_Z}. 
Suppose contrary that $(d\Phi)_{(x,(\ldots,l_{ij}',\ldots))}$ is not surjective. 
The rank of the following matrix is at most $2$: 
\[
\begin{pmatrix}
\vartheta_Z^{01}(x) & \cdots & \vartheta_Z^{ij}(x)&\cdots \\
\frac{\Pa \vartheta^{01}_Z}{\Pa z_1}(x) & \cdots &\frac{\Pa \vartheta^{ij}_Z}{\Pa z_1}(x)& \cdots \\
\frac{\Pa \vartheta^{01}_Z}{\Pa z_2}(x) & \cdots &\frac{\Pa \vartheta^{ij}_Z}{\Pa z_2}(x)& \cdots 
\end{pmatrix}
\]
Thus there exist complex numbers $k,c_1,c_2\in \C$ such that $(c_1,c_2) \neq (0,0)$ and the following equality holds for any $(i,j)\neq (0,0)$: 
\[
k\vartheta_Z^{ij}(x) = c_1\frac{\Pa \vartheta^{ij}_Z}{\Pa z_1}(x)+c_2\frac{\Pa \vartheta^{ij}_Z}{\Pa z_2}(x). 
\]
Since $(\overline{x},H_{\sum l_{ij}X_{ij}})$ is contained in $W_Z$, $\vartheta^{00}_Z(x)+\sum_{(i,j)\neq (0,0)} l_{ij}'\vartheta^{ij}_Z(x)$ is equal to $0$. 
Since all the $\vartheta^{ij}_Z(x)$'s cannot be equal to $0$ simultaneously, there exists a pair $(i,j)\neq (0,0)$ such that $\vartheta_Z^{ij}(x)\neq 0$. 
Suppose that $\vartheta_Z^{01}(x)\neq 0$ for simplicity (we can deal with the other cases similarly). 
In some coordinates, the map $\varphi_Z:T_Z\to \CP^N$ is described as follows: 
\[
x \mapsto \left(\frac{\vartheta_{Z}^{00}(x)}{\vartheta_{Z}^{01}(x)},\ldots, \frac{\vartheta_{Z}^{ij}(x)}{\vartheta_{Z}^{01}(x)},\ldots\right).
\]
It is easy to see that the differential of this map sends the vector $c_1 \left(\frac{\Pa}{\Pa z_1}\right)+c_2 \left(\frac{\Pa}{\Pa z_2}\right)$ to $0$. 
Thus, $\varphi_Z$ is not an embedding, contradicting the assumption. 
We can eventually conclude that $(d\Phi)_{(x,(\ldots,l_{ij}',\ldots))}$ is surjective. 
In particular, $W_Z$ is a submanifold of dimension $N-1$ and the tangent space $T_{(\overline{x},H)}W_Z$ can be identified with $\Ker((d\Phi)_{(x,(\ldots,l_{ij}',\ldots))})$.  
In the same way as in the proof of Lemma~\ref{L:criterion_Lefsing_W_Z}, we can verify that $p_2:W_Z\to (\CP^N)^\ast$ is an immersion at $(\overline{x},H)$ if and only if $(\overline{x},H)$ is contained in $W_Z^0$. 
Since the dimension of $\mathcal{D}_Z$ is $N-1$ by Lemma~\ref{L:dimension_D_Z}, $W_Z^0$ is not empty and $W_Z\setminus W_Z^0$ is proper sub-analytic set of $W_Z$. 
Since $W_Z$ is irreducible, the dimension of $W_Z\setminus W_Z^0$ is at most $N-2$. 
Since $p_2$ is a proper map, the dimension of $\mathcal{D}_Z'$ is also at most $N-2$ by \cite[Theorem 5.8]{Chirka_1989}.
\end{proof}

\begin{lemma}\label{L:conditionast_22}

The condition \textnormal{($\ast$)} holds for $(d_1,d_2)=(2,2)$. 

\end{lemma}

\begin{proof}
For $Z\in \h_2$, the $(2,2)$-polarized abelian surface $(T_Z,H_Z)$ is isomorphic to $(T_{Z'},H_{Z'}^2)$, where $Z' = Z/2\in \h_2$ and $(T_{Z'},H_{Z'})$ is the $(1,1)$-polarized abelian surface corresponding $Z'$ (the isomorphism sends $\overline{x}\in T_Z$ to $\overline{x/2}\in T_{Z'}$). 
Suppose that $T_Z$ is not a product of elliptic curves (this condition holds for generic $Z$ by Lemma~\ref{L:nongeneric_subset_moduli}). 
The abelian surface $T_{Z'}$ is not also a product of elliptic curves. 
We denote the Kummer surface $T_Z/ \left<-1\right>$ associated with $T_Z$ by $K_Z$. 
Since $L_{Z'}$ is symmetric in the sense of \cite[\S.4.6]{LangeBirkenhake_1992}, there exists an embedding $\psi_Z:K_Z\to (\CP^3)^\ast$ such that the following diagram commutes (see \cite[\S.4.8]{LangeBirkenhake_1992}): 
\[
\xymatrix{
T_Z \ar[r]^{\varphi_Z}\ar[d]_{\pi}  & (\CP^3)^\ast \\
K_Z \ar[ur]_{\psi_Z} & 
},
\]
where $\pi:T_Z\to K_Z$ is the quotient map. 
Thus, the set $R_0\subset T_Z$ defined at \eqref{E:definitionR_i} consists of sixteen points which are preimages of singular points of $K_Z$ under $\pi$, and $R_2=T_Z\setminus R_0$. 
The preimage $p_1^{-1}(R_2)\subset W_Z$ is a manifold with dimension $N-1$.
If the dimension of $p_2(p_1^{-1}(R_2))$ is less than $N-1$, that of $\mathcal{D}_Z\cap p_2(p_1^{-1}(R_2))$ is also less than $N-1$. 
Suppose that the dimension of $p_2(p_1^{-1}(R_2))$ is $N-1$. 
In the same way as that in the proof of Lemma~\ref{L:conditionast_d1d2geq5}, we can verify that $p_1^{-1}(R_2)\cap (W_Z\setminus W_Z^0)$ is the set of points at which the restriction $p_2:p_1^{-1}(R_2)\to \mathcal{D}_Z$ is not an immersion, especially $p_1^{-1}(R_2)\cap W_Z^0$ is not empty. 
Since $p_1^{-1}(R_2)$ is irreducible and $p_1^{-1}(R_2)\cap (W_Z\setminus W_Z^0)$ is locally analytic, the dimension of $p_1^{-1}(R_2)\cap (W_Z\setminus W_Z^0)$ is at most $N-2$. 

For $\overline{x}\in R_0$, the preimage $p_1^{-1}(\overline{x})$ is a hyperplane in $\{\overline{x}\}\times (\CP^N)^\ast$. 
The proof of \cite[Theorem 4.8.1]{LangeBirkenhake_1992} implies that the following map is an isomorphism for any $x\in \C^2$ representing $\overline{x}\in R_0$: 
\[
S^2(T_x\C^2) \to \Hom(\varphi_Z(\overline{x}),\C),\hspace{.4em} \sum a_{ij}\frac{\Pa^2}{\Pa z_i\Pa z_j} \mapsto \left(\varphi_Z(\overline{x})\ni \vartheta \mapsto a_{ij}\frac{\Pa^2\vartheta}{\Pa z_i\Pa z_j}(x)\in \C\right), 
\]
where $S^2(T_x\C^2)$ is the symmetric product of $T_x\C^2$ and we identify $\varphi_Z(\overline{x})\in (\CP^N)^\ast$ with a hyperplane in $\Gamma(L_Z)$ using the basis $\{\vartheta_Z^{ij}\}$ of $\Gamma(L_Z)$. 
We denote the image of $\frac{\Pa^2}{\Pa z_i\Pa z_j}$ under the map above by $s_{ij}\in \Hom(\varphi_Z(\overline{x}),\C)$. 
The set $\{s_{11},s_{12},s_{22}\}$ is a basis of $\Hom(\varphi_Z(\overline{x}),\C)$. 
Thus we can take a dual basis $\{s_{11}^\ast,s_{12}^\ast,s_{22}^\ast\}$ of $\Hom(\varphi_Z(\overline{x}),\C)^\ast\cong \varphi_Z(\overline{x})$. 
Let $\vartheta_0 = s_{11}^\ast+s_{22}^\ast \in \varphi_Z(\overline{x})$. 
This Theta function satisfies the condition $\frac{\Pa^2\vartheta_0}{\Pa z_i\Pa z_j}(x)=\delta_{ij}$. 
In particular, $(\overline{x},\vartheta_0)$ is contained in $W_Z^0\cap p_1^{-1}(\overline{x})$. 
Since $p_1^{-1}(\overline{x})$ is irreducible, the dimension of $(W_Z\setminus W_Z^0)\cap p_1^{-1}(\overline{x})$ is at most $N-2$. 
We can eventually conclude that the dimension of 
$W_Z\setminus W_Z^0 = \left(p_1^{-1}(R_2)\cap (W_Z\setminus W_Z^0)\right) \cup \left((p_1^{-1}(R_0)\cap (W_Z\setminus W_Z^0)\right)$ is at most $N-2$. 
Thus the dimension of the image $\mathcal{D}_Z'=p_2(W_Z\setminus W_Z^0)$ is also at most $N-2$. 
\end{proof}

\noindent
Theorem~\ref{T:main_uniquenessLP} immediately follows from Theorem~\ref{T:uniqueness_isomLP}, Lemma~\ref{L:conditionast_d1d2geq5} and Lemma~\ref{L:conditionast_22}. 

We thus far cannot guarantee that the assumption ($\ast$) holds when $(d_1,d_2)=(1,3),(1,4)$.   
Furthermore, the arguments in this section do not work when $(d_1,d_2)=(1,2)$ (note that we assumed that $d_1d_2$ is at least $3$ in the paragraph followed by Lemma~\ref{L:dimension_S_i}). 
Still, we believe the following conjecture holds: 

\begin{conjecture}\label{C:uniqueness_holLP}

Two holomorphic Lefschetz pencils on the four-torus (with any genera and divisibilities) are holomorphic if and only if their genera and divisibilities coincide. 

\end{conjecture}

\section{Examples of Lefschetz pencils on the four-torus}\label{S:example_holLP}

As was shown in the previous section, for any pair of positive integers $d_1,d_2$ with $d_1|d_2$ and $d_1d_2\geq 5$, there exists a holomorphic Lefschetz pencils on $T^4$ with genus-$(d_1d_2+1)$ and divisibility $d_1$ (see Lemmas~\ref{L:connectedness_lines} and \ref{L:conditionast_d1d2geq5}). 
In this section we will construct some of them explicitly and determine their monodromy factorizations. 

We begin with observing the relation between (possibly branched) coverings and monodromies of mappings. 
Let $X$ be a closed four-manifold, $\Sigma$ a closed surface and $f:X\to \Sigma$ a smooth map with discrete critical value set. 
As we defined in Subsection~\ref{S:monodromy_LFLP}, we can define a local monodromy by taking a loop around a critical value of $f$ and a parallel transport along this loop with respect to a horizontal distribution of the submersion $f|_{X\setminus \Crit(f)}$.  
Let $q:\tilde{X}\to X$ be a covering branched at (possibly disconnected and empty) immersed surface $S$ with transverse self-intersections (the reader can refer to \cite[Chap.~7]{GS}, for example, for covering branched at such surfaces). 
We denote the set of self-intersections of $S$ by $D(S)\subset X$ and the critical point set of the restriction $f|_{S\setminus D(S)}$ by $T_f(S)$. 
In what follows we assume that the image $f(T_f(S))$ is a discrete set. 
It is easy to see that the critical value set of the composition $f\circ q:\tilde{X}\to \Sigma$ is contained in the image $f(\Crit(f)\cup D(S)\cup T_f(S))$, which is a discrete set by assumption. 
In particular, we can define a local monodromy of each critical value of $f\circ q$. 
We will discuss the relation of monodromies of $f$ and $f \circ q$ below. 

Let $a_0\in \Sigma$ be a point away from $f(\Crit(f)\cup D(S)\cup T_f(S))$. 
By the assumption the fiber $f^{-1}(a_0)$ is a submanifold of $X$ and intersects $S$ transversely. 
In particular the intersection $f^{-1}(a_0)\cap S$ is a finite set. 
Using this we can identify $f^{-1}(a_0)$ with a genus-$g$ closed surface $\Sigma_g$ with $p=\sharp(f^{-1}(a_0)\cap S)$ marked points, which we denote by $\Sigma_{g,p}$. 
For a point $a \in f(\Crit(f)\cup D(S)\cup T_f(S))$ we take a path $\alpha$ from $a_0$ to $a$. 
We further take a loop $\widetilde{\alpha}$ by connecting $\alpha$ with a small circle around $a$. 
Let $\mathcal{H}$ be a horizontal distribution of the restriction $f|_{X\setminus (\Crit(f)\cup D(S)\cup T_f(S))}$ such that $\mathcal{H}_x$ coincides with $T_xS$ for any $x\in S\setminus (\Crit(f)\cup D(S)\cup T_f(S))$.  
Using the identification $f^{-1}(a_0)\cong \Sigma_{g,p}$, we can regard the parallel transport $T_{\widetilde{\alpha},\mathcal{H}}$ along $\widetilde{\alpha}$ with respect to $\mathcal{H}$ as self-diffeomorphism of $\Sigma_{g,p}$ preserving the marked points setwise. 

By the assumption the restriction $q|_{f^{-1}(a_0)}$ is a covering branched at the finite set $f^{-1}(a_0)\cap S$. 
In particular we can take an identification of a fiber $(f\circ q)^{-1}(a_0)$ with a marked surface $\Sigma_{\tilde{g},p}$, which is a covering of $\Sigma_{g,p}$ branched at $p$ marked points. 
Since $\mathcal{H}$ is tangent to $S$ at any point in $S\setminus (\Crit(f)\cup D(S)\cup T_f(S))$, we can take a lift $\widetilde{\mathcal{H}}$ of $\mathcal{H}$ by the branched covering $q$, which is a horizontal distribution of $f\circ q$. 
It is easy to verify that the parallel transport $T_{\widetilde{\alpha},\widetilde{\mathcal{H}}}$ is a lift of $T_{\widetilde{\alpha},\mathcal{H}}$ by $q$, that is, the following diagram commutes: 
\[
\begin{CD}
\Sigma_{\tilde{g},p} @> T_{\widetilde{\alpha},\widetilde{\mathcal{H}}}>> \Sigma_{\tilde{g},p} \\
@Vq VV @VV qV \\
\Sigma_{g,p} @> T_{\widetilde{\alpha},\mathcal{H}}>> \Sigma_{g,p}. 
\end{CD}
\]
We eventually obtain the following lemma: 

\begin{lemma}\label{L:relation_monodromy_covering}

Let $\Mod(\Sigma_{g,p})$ be the group of isotopy classes of self-diffeomophisms of $\Sigma_{g,p}$ preserving the marked points setwise, and $\varphi_{\widetilde{\alpha}}\in \Mod(\Sigma_{g,p})$ the monodromy of $f$ along $\widetilde{\alpha}$. 
Then the monodromy of $f\circ q$ along $\widetilde{\alpha}$ is represented by a lift of a representative of $\varphi_{\widetilde{\alpha}}$ by $q$. 

\end{lemma}

\begin{remark}\label{R:ambiguity_lift}

Lemma~\ref{L:relation_monodromy_covering} \textit{does not} uniquely determine the monodromy of $f\circ q$ along $\widetilde{\alpha}$: a lift of a representative of $\varphi_{\widetilde{\alpha}}$ by $q$ is unique \textit{up to covering transformations of $q$}. 
Still, such an ambiguity would not cause any problems in the following subsections. 
Indeed, monodromies we will deal with must satisfy some additional conditions, which determine them uniquely. 

\end{remark}

\subsection{Genus-$3$ holomorphic pencils due to Smith} \label{Subsection:Smith's_Pencil}

In \cite{Smith_2001_torus} Smith gave a way to construct a genus-$3$ holomorphic Lefschetz pencil on $T^4$ by taking a branched covering of a singular projective variety. 
Although Smith showed that we can obtain a holomorphic pencil by his construction, he neither carried out the construction in practice nor obtained vanishing cycles of the resulting pencil (but mentioned the symplectic representation of the monodromy). 
In this subsection, we will construct a genus-$3$ holomorphic pencil of $T^4$ following the construction in \cite{Smith_2001_torus} and determine the vanishing cycles of the pencil. 

We begin with a brief review on Smith's construction. 
For homogeneous polynomials $q_1,\ldots,q_n$, we denote the zero-set of them by $V(q_1,\ldots,q_n)\subset \CP^m$. 
We put $Q= V(x^2+y^2+z^2) \subset \CP^3$ and $S= V(x^2+y^2+z^2) \subset \CP^2$. 
The set $Q$ is a singular variety with an $A_1$-singularity $[1:0:0:0]\in Q$ and $S$ is a sphere in $\CP^2$.   
Let $\pi:Q \setminus \{[1:0:0:0]\} \to S$ be the restriction of the projection $[t:x:y:z]\mapsto [x:y:z]$. 
We take six conics $C_1,\ldots,C_6$ with the following properties: 

\begin{itemize}

\item
for each $i$, $C_i$ is away from the singularity $[1:0:0:0]$,  

\item
two spheres $C_i$ and $C_j$ intersect on two points transversely for $i\neq j$, $i,j \leq 4$ or $(i,j)=(5,6)$,

\item
two spheres $C_5$ and $C_6$ are tangent to $C_i$ on one point for $i \leq 4$, 

\end{itemize}

\noindent
Let $q_1:Z\to Q$ be a double covering branched at $C_1\cup \cdots \cup C_4$.
The space $Z$ has two $A_1$-singularities in the preimage $q_1^{-1}([1:0:0:0])$. 
Let $r:Z_{\sm}\to Z$ be the resolution of these singularities. 
The space $Z_{\sm}$ is a manifold obtained by replacing neighborhoods of the two singularities of $Z$ with two disk bundles over the sphere with degree $-2$. 
In particular, $Z_{\sm}$ has two spheres $S_1, S_2$ with self-intersection $-2$. 
Since $C_1\cup \cdots\cup C_4$ has twelve transverse double points, $Z_{\sm}$ has other twelve spheres $S_3,\ldots,S_{14}$ with self-intersection $-2$. 
Furthermore, the preimage $q_1^{-1}(C_5\cup C_6)$ contains two disjoint sphere $S_{15}$ and $S_{16}$ with self-intersection $-2$. 
We can take a double covering $q_2:\widetilde{T}\to Z_{\sm}$ branched at the disjoint union $\sqcup_{i=1}^{16}S_i$. 
The space $\widetilde{T}$ has $16$ exceptional spheres in the preimage $q_2^{-1}(\sqcup_i S_i)$. 
We denote the blow-down of $\widetilde{T}$ along these spheres by $T$. 
The composition $\pi\circ q_1\circ r\circ q_2$ factors through $T$ and defines a pencil $f:T\dashedrightarrow S$ with four base points which satisfies the conditions (2) and (3) in p.\pageref{P:definition_LP}. 
If we take conics $C_1,\ldots,C_6$ so that the restriction of $\pi$ on the set of double points of $C_1\cup \cdots \cup C_4$ is injective, the resulting pencil $f$ becomes Lefschetz. 

In what follows, we consider the following conics in $Q$: 
{\allowdisplaybreaks
\begin{align*}
& C_1=V(x^2+y^2+z^2,t), \hspace{.5em} C_2=V(x^2+y^2+z^2,t+x),\\
&C_3=V\left(x^2+y^2+z^2,t+\frac{x}{2}+\frac{y}{2}\right),\hspace{.5em}C_4=V\left( x^2+y^2+z^2,t+\frac{x}{2}-\frac{y}{2}\right), \\
&C_5=V\left(x^2+y^2+z^2, t+\frac{x}{2}+\frac{iz}{2}\right), \hspace{.5em}C_6=V\left(x^2+y^2+z^2, t+\frac{x}{2}-\frac{iz}{2}\right). 
\end{align*}
}
It is easy to verify that these conics satisfy the three conditions in the previous paragraph. 
We denote the set of double points of $C_1\cup \cdots \cup C_6$ by $D$. 
Using Lemma~\ref{L:relation_monodromy_covering} we can obtain vanishing cycles of $f$ once we can calculate the monodromies of $\pi$ around the image $\pi(D)$.  
Let $S_0=\{[x:y:z]\in S~|~ z\neq 0\}$. 
We define a holomorphic map $\varphi:S_0\to \C^2$ as $\varphi([x:y:z]) = \left(\frac{x}{z},\frac{y}{z}\right)$. 
Since the image $\varphi (S_0)$ is equal to $\{(X,Y)\in \C^2~|~ (X+iY)(X-iY)=-1\}$, the composition $\psi\circ \varphi:S_0\to \C^\times$ is biholomorphic, where $\psi:\C^2\to \C^\times$ is defined as $\psi(X,Y)=X-iY$. 
Furthermore, this map can be extended to a biholomorphic map $S\to \overline{\C}=\C\cup \{\infty\}$ which sends $[1:i:0]$ and $[1:-i:0]$ to $\infty$ and $0$, respectively. 
Using this map, we will identify $S$ with $\overline{\C}$ throughout this subsection. 
The following lemma can be deduced easily by direct calculation.

\begin{lemma}\label{L:configuration_image_intersection}

The image $\pi(D)$ is contained in $\{\xi^n\in \overline{\C}~|~ n=0,\ldots,7\}\cup \{0,\infty\}$, where $\xi =\exp\left(\frac{\pi i}{4}\right)$. 
Furthermore, the intersection $C_5\cap C_6$ is contained in $\pi^{-1}(\{0,\infty\})$. 

\end{lemma}

For any $w \in \C = \overline{\C}\setminus \{\infty\}$, the map $\pi^{-1}(w) \to \C$ defined as $[t:x:y:z]\to \frac{t}{x-iy}$ is biholomorphic. 
Using this, we will identify the fiber $\pi^{-1}(w)$ with $\C$ for any $w\in \C$. 
With this identification in hand, we can define a path $\gamma^{(i)}:J\to \C$ ($i=1,\ldots,6$) for any path $\gamma:J\to \C$ ($J\subset \R$) as follows: 
\[
\gamma^{(i)}(t) = z \in \pi^{-1}(\gamma(t))\cap C_i\subset \C. 
\]
The value of this path is indeed determined uniquely since $C_i$ is a section of $\pi$. 
Let $\alpha$ be an oriented path in $\C$ which intersects $\pi(D)$ only at its terminal point. 
The corresponding paths $\alpha^{(1)},\ldots,\alpha^{(6)}$ are also oriented paths in $\C$ two of which, say $\alpha^{(i)}$ and $\alpha^{(j)}$, intersect at their common terminal point. 
We denote by $\widetilde{\alpha}$ the oriented loop based at the initial point of $\alpha$ obtained by connecting $\alpha$ with a small counterclockwise circle around the terminal point of $\alpha$. 
We can easily verify that the parallel transport along $\widetilde{\alpha}$ is isotopic to a composition of the point pushing self-diffeomorphism of $\C$ along the paths $\alpha^{(1)},\ldots,\alpha^{(6)}$, the $m_{ij}$-th power of the local full-twist around the common terminal point of $\alpha^{(i)}$ and $\alpha^{(j)}$, and the inverse of the point pushing self-diffeomorphism, where $m_{ij}$ is the multiplicity of the intersection between $C_i$ and $C_j$ in the fiber on the terminal point of $\alpha$, which is $1$ if $i,j\leq 4$ or $i,j \geq 5$ and $2$ otherwise. 

Let $\alpha_k$ ($k=1,2,3,4$) and $\beta$ be paths in $\C$ defined as follows: 
{\allowdisplaybreaks
\begin{align*}
\alpha_k:[-1,1] \to \C,\hspace{.5em} & \alpha_k(s) = s\xi^{k-1}, \\
\beta:[-\pi,\pi] \to \C,\hspace{.5em} & \beta(\theta) = \varepsilon\exp \left(i\theta\right), 
\end{align*}
}
where $\varepsilon >0$ is sufficiently small real number. 
In order to determine monodromies of $\pi$, we first calculate the paths $\alpha_k^{(j)}$ and $\beta^{(j)}$. 
Under the identification $S\cong \overline{\C}$, $\alpha_1(s)=s$ corresponds with $[s^2-1:-i(s^2+1):2s]$. 
Thus, $\alpha_1^{(j)}(s) = \frac{t}{x-iy}$ can be calculated as follows: 
{\allowdisplaybreaks
\begin{align*}
\alpha_1^{(1)}(s) = \frac{0}{x-iy} = 0, \hspace{.5em}& \alpha_1^{(2)}(s)= \frac{-x}{x-iy} = \frac{s^2-1}{2}, \\
\alpha_1^{(3)}(s) = \frac{-x/2-y/2}{x-iy} = \frac{s^2-1}{4}-i\frac{s^2+1}{4}, \hspace{.5em}& \alpha_1^{(4)}(s)= \frac{-x/2+y/2}{x-iy} = \frac{s^2-1}{4}+i\frac{s^2+1}{4}, \\
\alpha_1^{(5)}(s) = \frac{-x/2-iz/2}{x-iy} = \frac{s^2-1}{4}+i\frac{s}{2}, \hspace{.5em}& \alpha_1^{(6)}(s)= \frac{-x/2+iz/2}{x-iy} =  \frac{s^2-1}{4}-i\frac{s}{2}.
\end{align*}
}
In the same way above, we can also calculate the other paths as follows: 
{\allowdisplaybreaks
\begin{align*}
\alpha_2^{(1)}(s) = 0, \hspace{.5em}& \alpha_2^{(2)}(s)= \frac{\xi}{2\sqrt{2}}\left((s^2-1)+i(s^2+1)\right), \\
\alpha_2^{(3)}(s) =\frac{\xi}{2\sqrt{2}}\left(s^2-1\right) , \hspace{.5em}& \alpha_2^{(4)}(s)=\frac{\xi^3}{2\sqrt{2}}\left(s^2+1\right), \\
\alpha_2^{(5)}(s) = \frac{\xi}{4\sqrt{2}}\left((s^2-1)+i(s^2+2\sqrt{2}s+1)\right), \hspace{.5em}& \alpha_2^{(6)}(s)=\frac{\xi}{4\sqrt{2}}\left((s^2-1)+i(s^2-2\sqrt{2}s+1)\right),\\
\alpha_3^{(1)}(s) = 0, \hspace{.5em}& \alpha_3^{(2)}(s)= -\frac{1}{2}\left(s^2+1\right), \\
\alpha_3^{(3)}(s) =\frac{1}{4}\left(-(s^2+1)+i(s^2-1)\right) , \hspace{.5em}& \alpha_3^{(4)}(s)=\frac{1}{4}\left(-(s^2+1)-i(s^2-1)\right), \\
\alpha_3^{(5)}(s) = -\frac{1}{4}\left(s^2+1\right)^2, \hspace{.5em}& \alpha_3^{(6)}(s)=-\frac{1}{4}\left(s^2-1\right)^2, \\
\alpha_4^{(1)}(s) = 0, \hspace{.5em}& \alpha_4^{(2)}(s)= -\frac{1}{2}\left(1+is^2\right), \\
\alpha_4^{(3)}(s) =-\frac{\xi}{2\sqrt{2}}\left(1+s^2\right) , \hspace{.5em}& \alpha_4^{(4)}(s)=-\frac{\xi^{-1}}{2\sqrt{2}}\left(1-s^2\right), \\
\alpha_4^{(5)}(s) = -\frac{1}{4}\left((1+\sqrt{2}s)+i(s^2+\sqrt{2}s)\right)^2, \hspace{.5em}& \alpha_4^{(6)}(s)=-\frac{1}{4}\left((1-\sqrt{2}s)+i(s^2-\sqrt{2}s)\right)^2, \\
\beta^{(1)}(\theta) = 0, \hspace{.5em}& \beta^{(2)}(\theta)= \frac{1}{2}\left(\varepsilon^2e^{2i\theta}-1\right), \\
\beta^{(3)}(\theta) =\frac{1}{4}\left((\varepsilon^2e^{2i\theta}-1)-i(\varepsilon^2e^{2i\theta}+1)\right) , \hspace{.5em}& \beta^{(4)}(\theta)=\frac{1}{4}\left((\varepsilon^2e^{2i\theta}-1)+i(\varepsilon^2e^{2i\theta}+1)\right), \\
\beta^{(5)}(\theta) = \left(\frac{\varepsilon e^{i\theta}-i}{2}\right)^2, \hspace{.5em}& \beta^{(6)}(\theta)=\left(\frac{\varepsilon e^{i\theta}+i}{2}\right)^2.
\end{align*}
}
We can draw the paths $\alpha_i^{(j)}$ in the plane $\C$ as shown in Figure~\ref{F:PathsInFiber}. 
\begin{figure}[htbp]
\centering
\subfigure[$\alpha_1^{(j)}(s)$ with $s\geq 0$.]{\includegraphics[height=46mm]{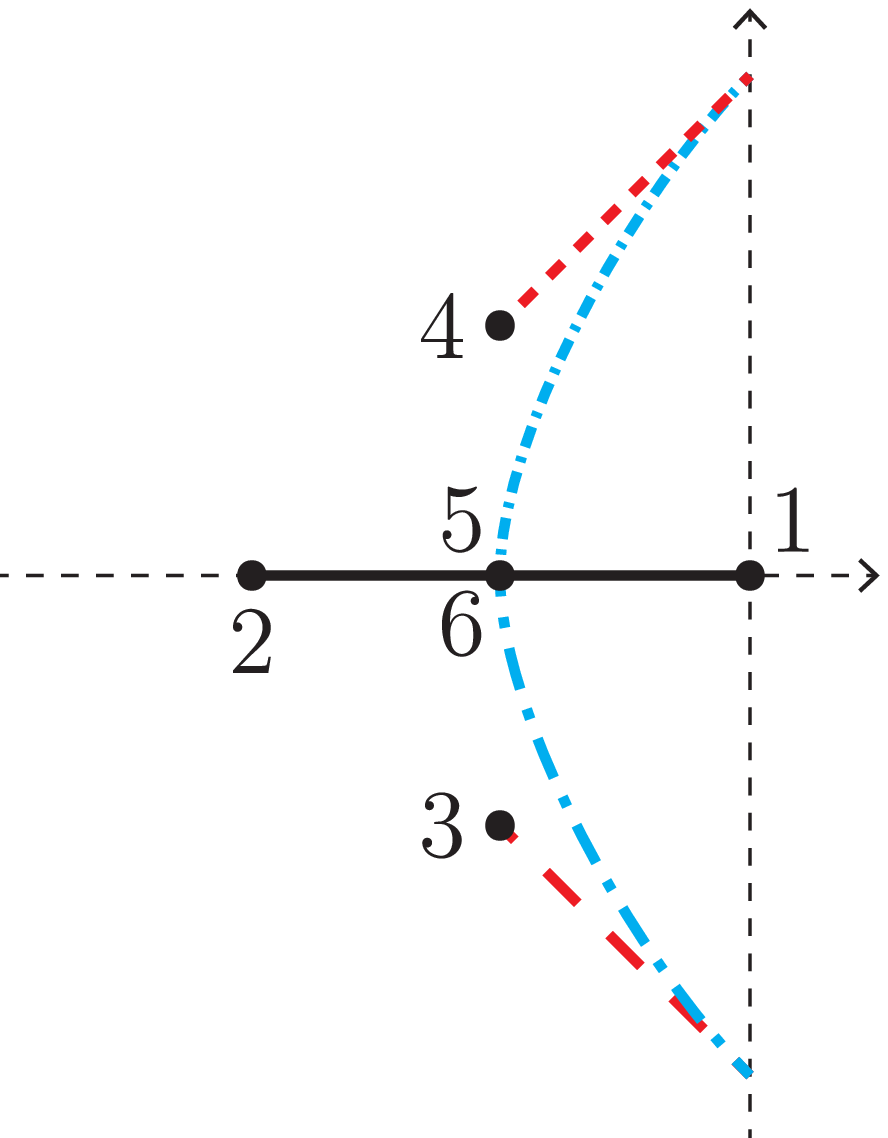}
\label{F:PathsInFiber1}}
\hspace{.2em}
\subfigure[$\alpha_1^{(j)}(s)$ with $s\leq 0$.]{\includegraphics[height=46mm]{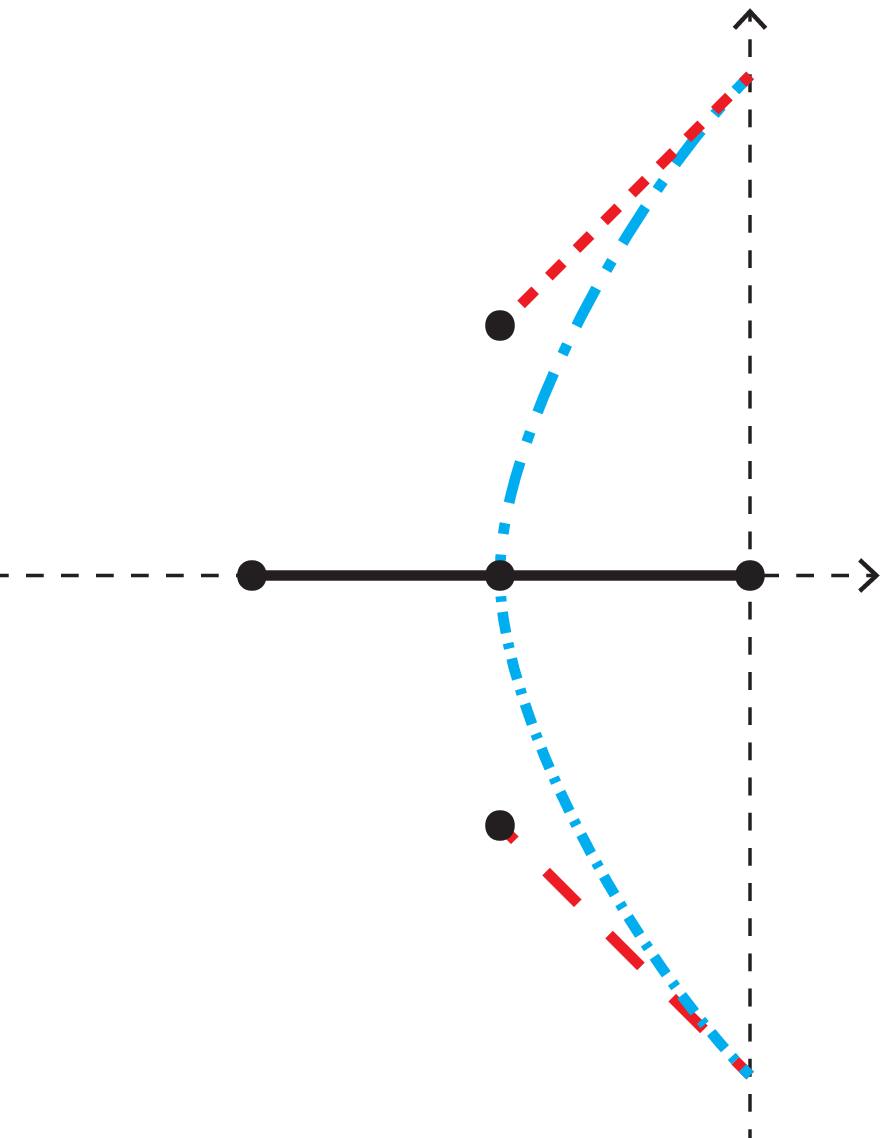}
\label{F:PathsInFiber2}}
\hspace{.2em}
\subfigure[$\alpha_2^{(j)}(s)$ with $s\geq 0$.]{\includegraphics[height=46mm]{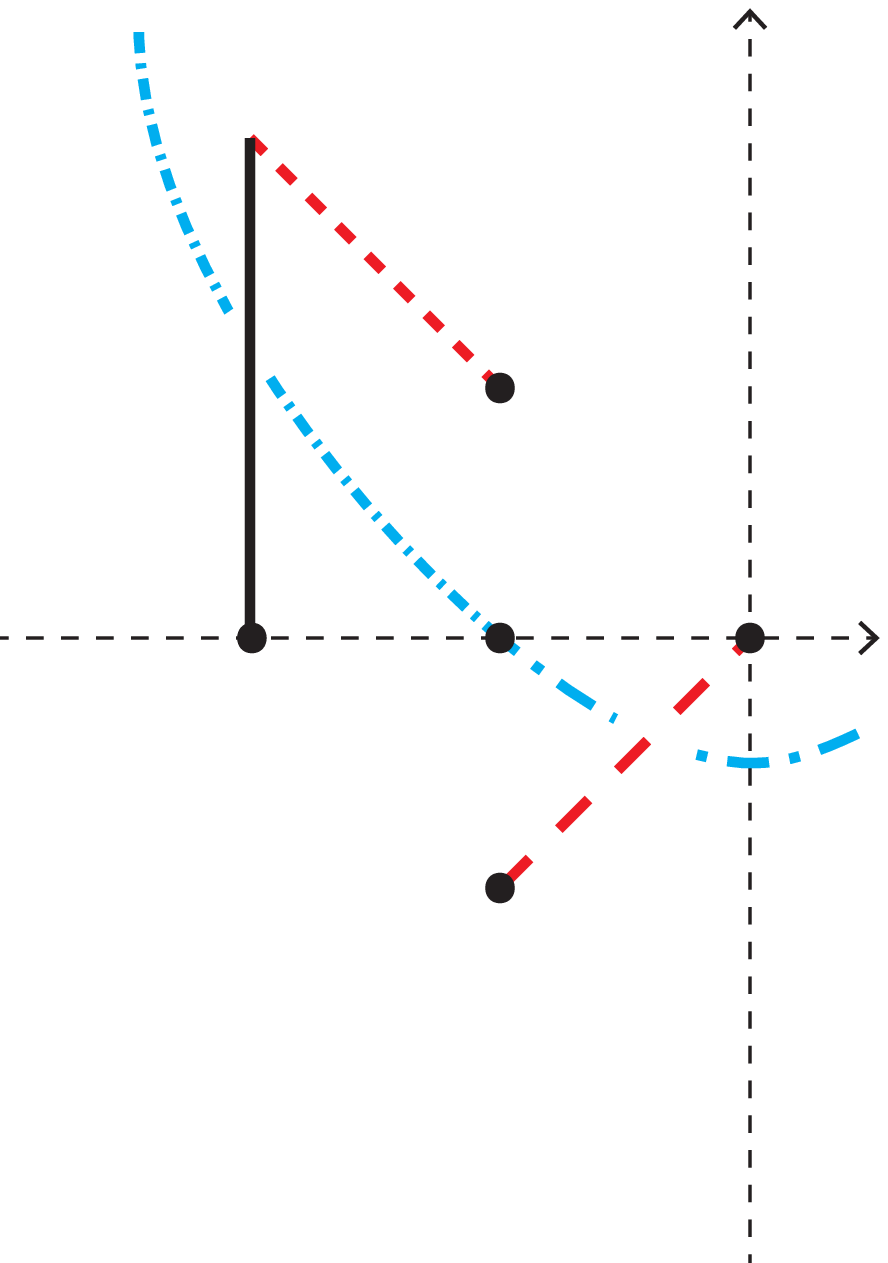}
\label{F:PathsInFiber3}}
\hspace{.2em}
\subfigure[$\alpha_2^{(j)}(s)$ with $s\leq 0$.]{\includegraphics[height=46mm]{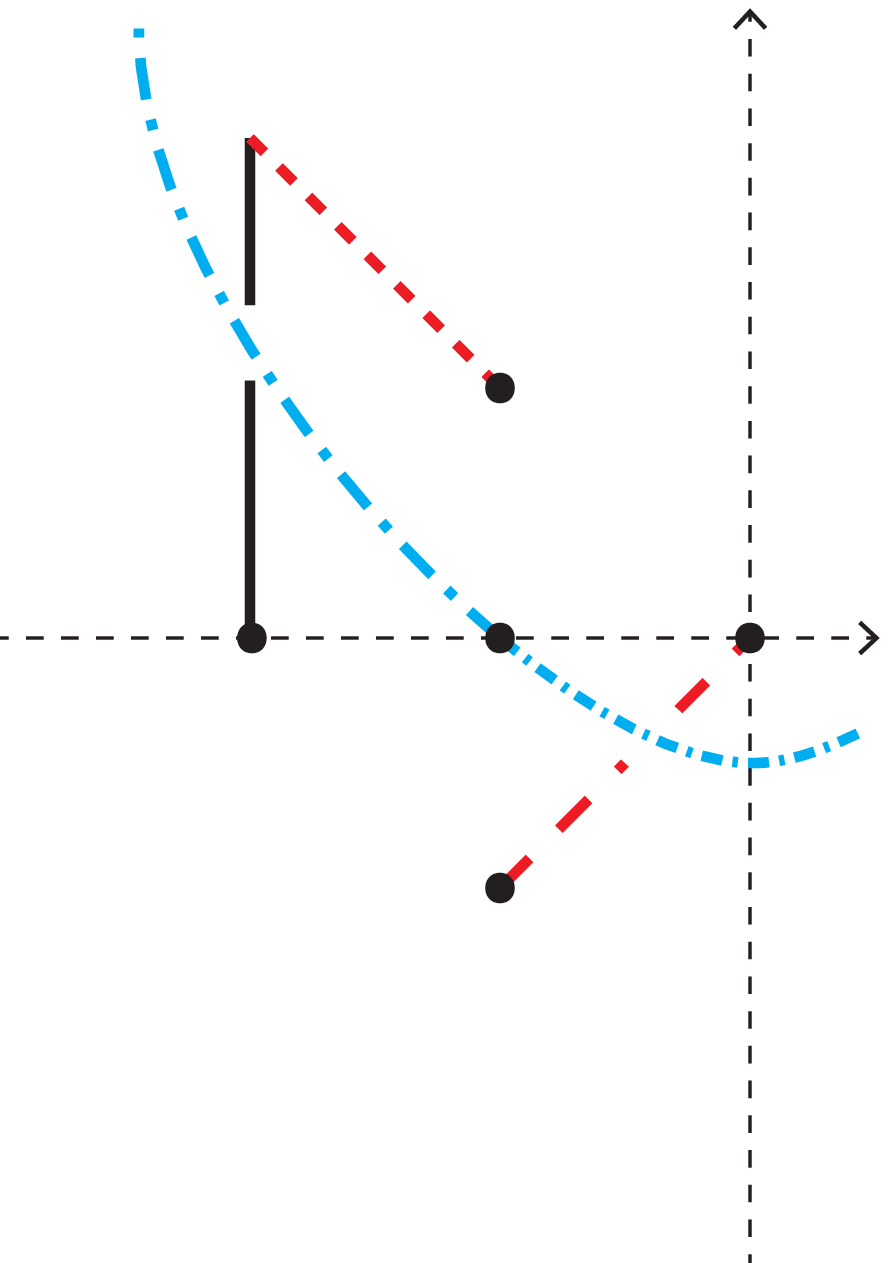}
\label{F:PathsInFiber4}}

\subfigure[$\alpha_3^{(j)}(s)$ with $s\geq 0$.]{\includegraphics[height=46mm]{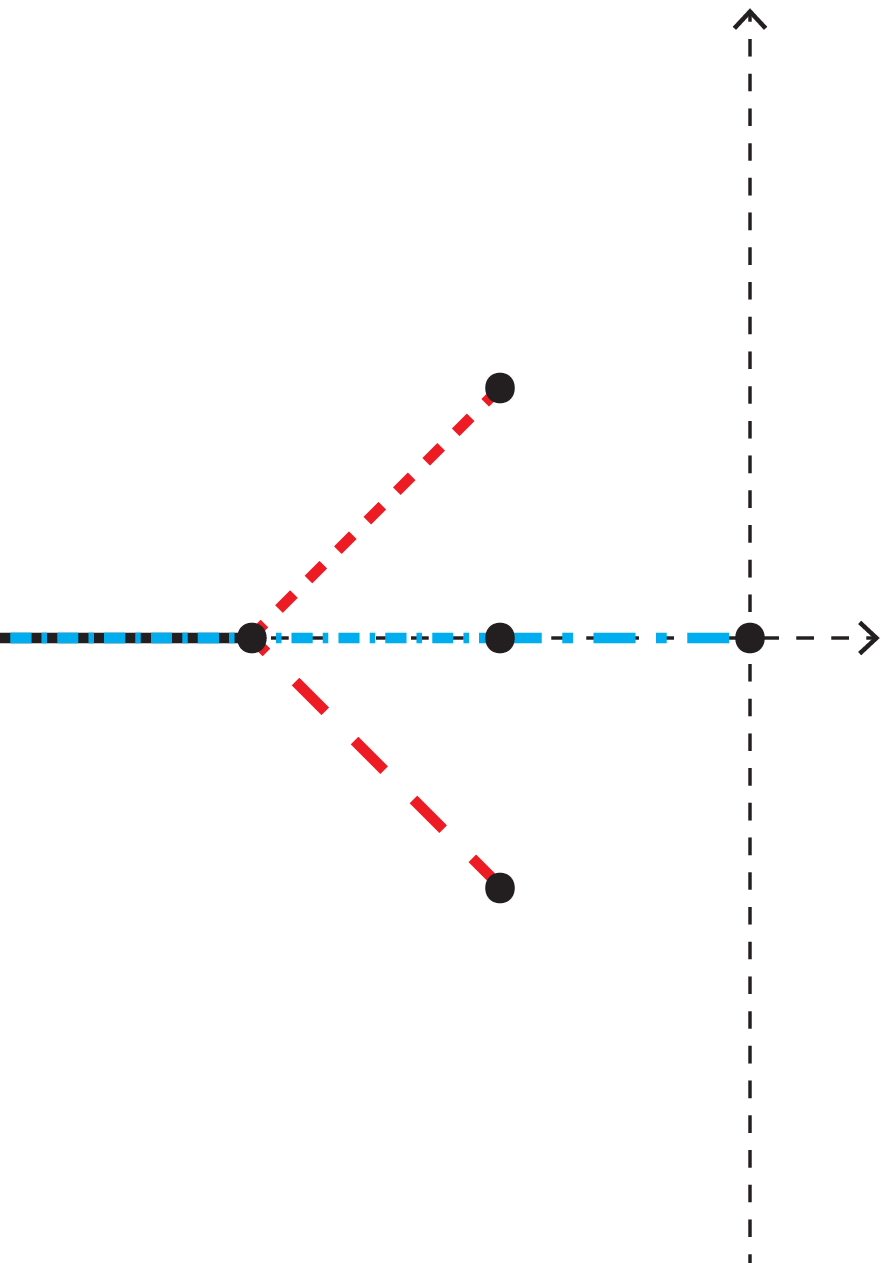}
\label{F:PathsInFiber5}}
\hspace{.2em}
\subfigure[$\alpha_3^{(j)}(s)$ with $s\leq 0$.]{\includegraphics[height=46mm]{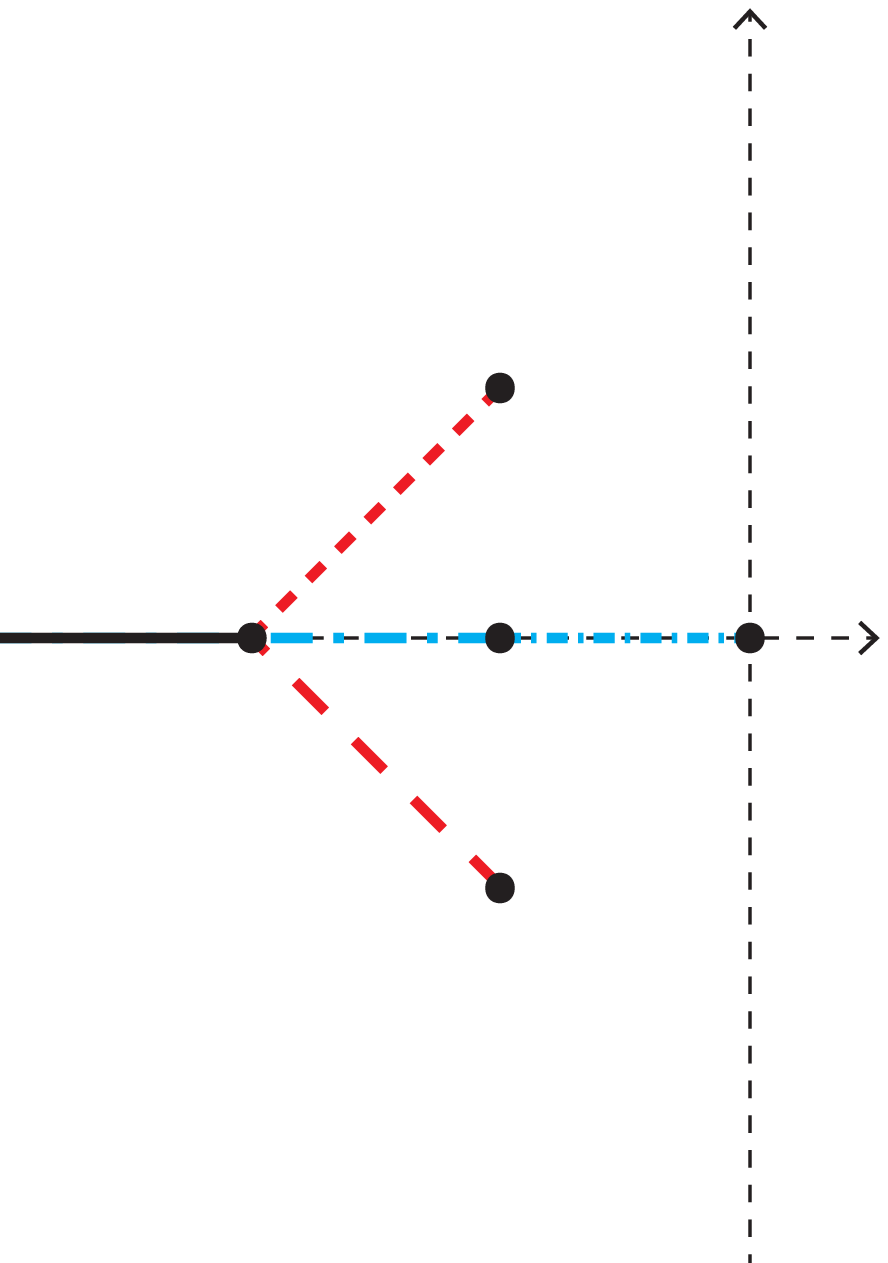}
\label{F:PathsInFiber6}}
\hspace{.2em}
\subfigure[$\alpha_4^{(j)}(s)$ with $s\geq 0$.]{\includegraphics[height=46mm]{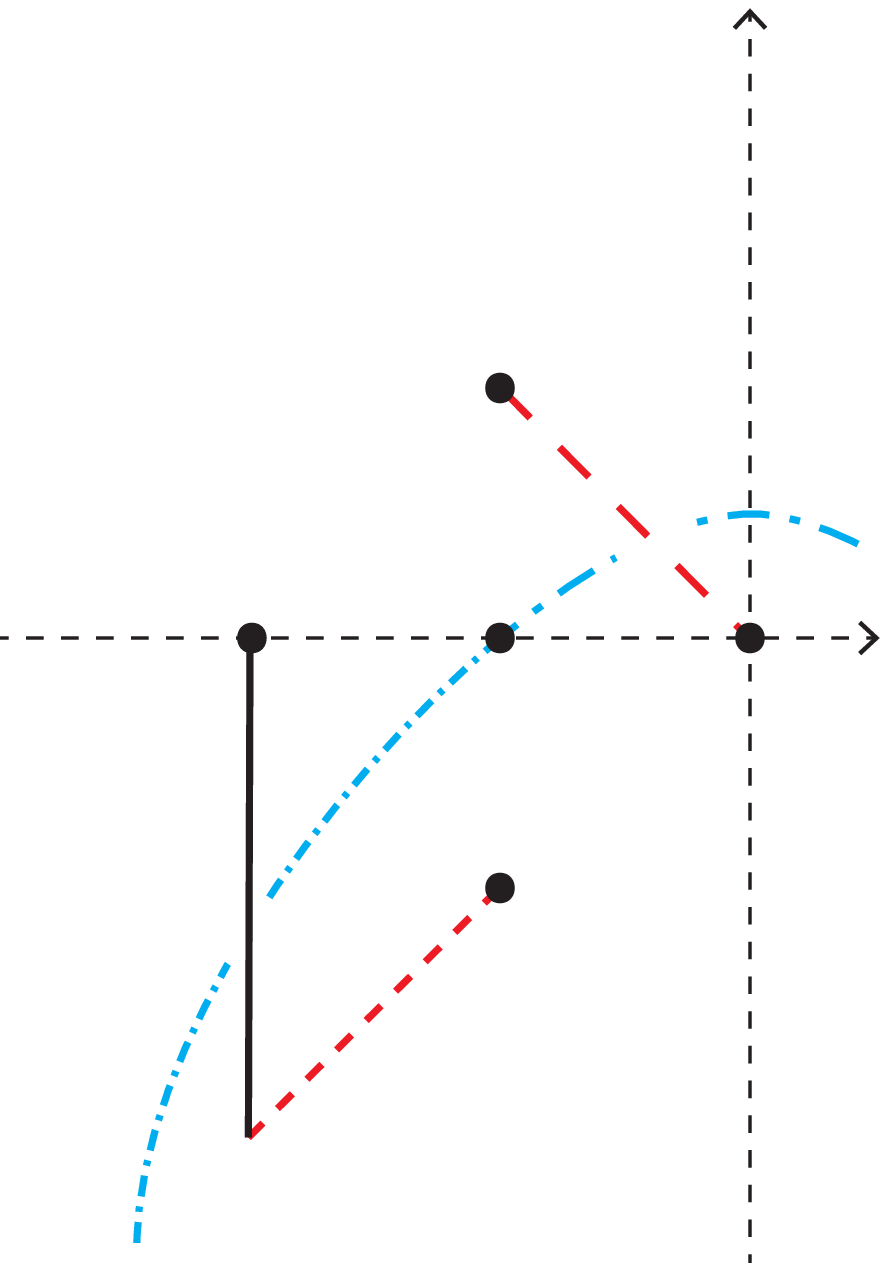}
\label{F:PathsInFiber7}}
\hspace{.2em}
\subfigure[$\alpha_4^{(j)}(s)$ with $s\leq 0$.]{\includegraphics[height=46mm]{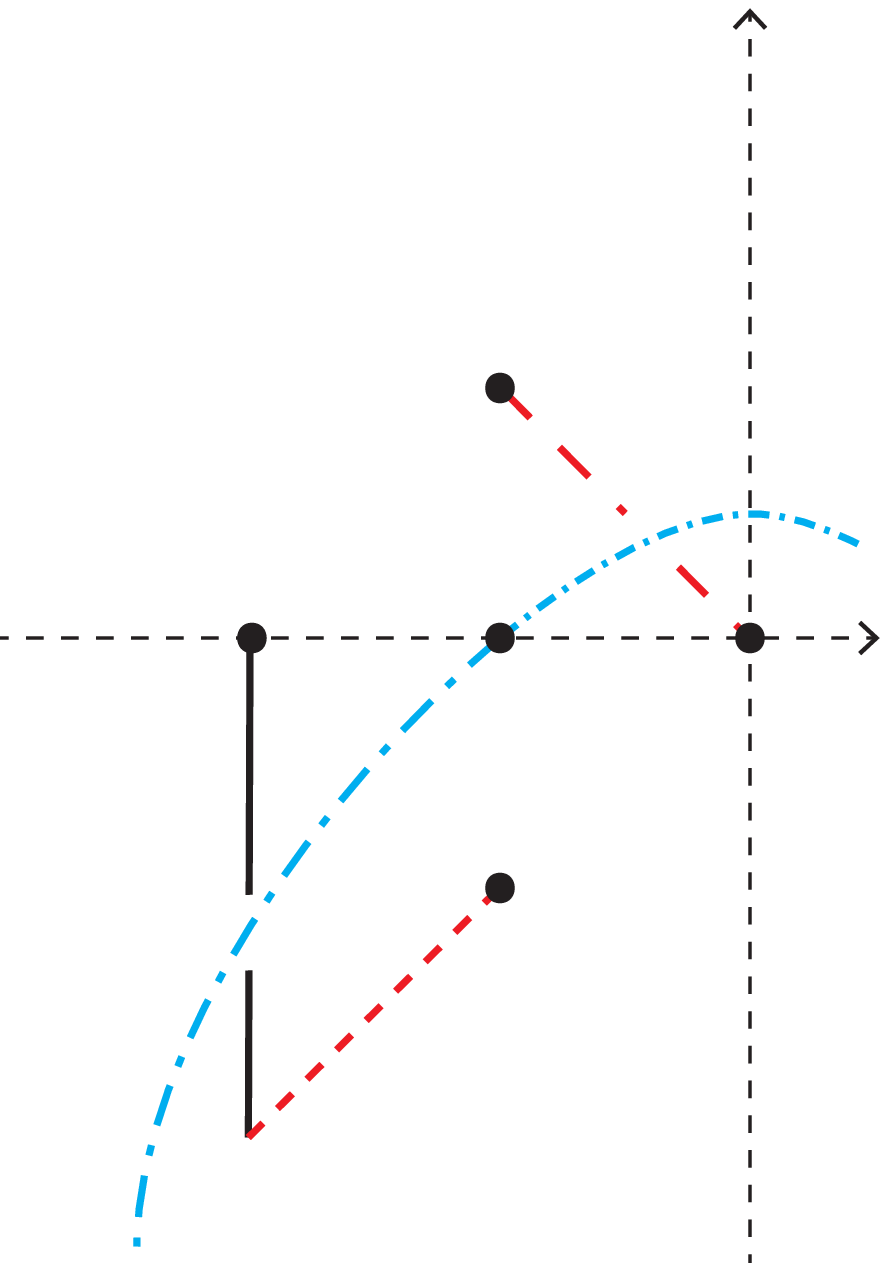}
\label{F:PathsInFiber8}}
\caption{The paths $\alpha_i^{(j)}$ in $\C\cong \R^2$. }
\label{F:PathsInFiber}
\end{figure}
In each of the figures, the five dots are the points $\alpha_k^{(1)}(0), \cdots, \alpha_k^{(6)}(0)$ (note that $\alpha_k^{(5)}(0)=\alpha_k^{(6)}(0)$), $\alpha_k^{(1)}$ is the constant path, the bold line describes the path $\alpha_k^{(2)}$, the dotted lines (which are colored in red) describe the paths $\alpha_k^{(3)}$ and $\alpha_k^{(4)}$ (the denser one is $\alpha_k^{(4)}$, while the other one is $\alpha_k^{(3)}$) and the semi-dotted lines (which are colored in light blue) describe the paths $\alpha_k^{(5)}$ and $\alpha_k^{(6)}$ (the denser one is $\alpha_k^{(5)}$, while the other one is $\alpha_k^{(6)}$). 
Moreover, at each of the transverse crossings except for the terminal points, the path going over the other path goes through the crossing point after the other path comes to the point when the parameter $s$ increases.

We define a path $\gamma_k$ ($k=1,\ldots,8$) in $\C$ as follows: 
\begin{itemize}

\item
for $k\leq 4$, $\gamma_k$ is defined to be the concatenation of $\beta|_{[0,(k-1)\pi i/4]}$ and $\alpha_k|_{[\varepsilon,1]}$, 

\item
for $k\geq 5$, $\gamma_k$ is defined to be the concatenation of $\alpha_{k-4}|_{[-1,-\varepsilon]}$ and $\beta|_{[-(9-k)\pi i/4,0]}$ with the opposite orientation. 

\end{itemize}
\noindent
According to the arguments above, the monodromy of $\pi$ along the path $\widetilde{\gamma_k}$ is the product of the full-twists along the paths shown in Figure~\ref{F:PathsForTwist} and the squares of the full-twists along other paths, which have either $\alpha_k^{(5)}(0)$ or $\alpha_k^{(6)}(0)$ as end points and are not described in Figure~\ref{F:PathsForTwist}. 
(Note that we only need local monodromies derived from the double points in $C_1\cup \cdots \cup C_4$ in order to obtain vanishing cycles of $f:T\to S$.)
\begin{figure}[htbp]
\centering
\subfigure[A path for $\widetilde{\gamma_1}$.]{\includegraphics[height=30mm]{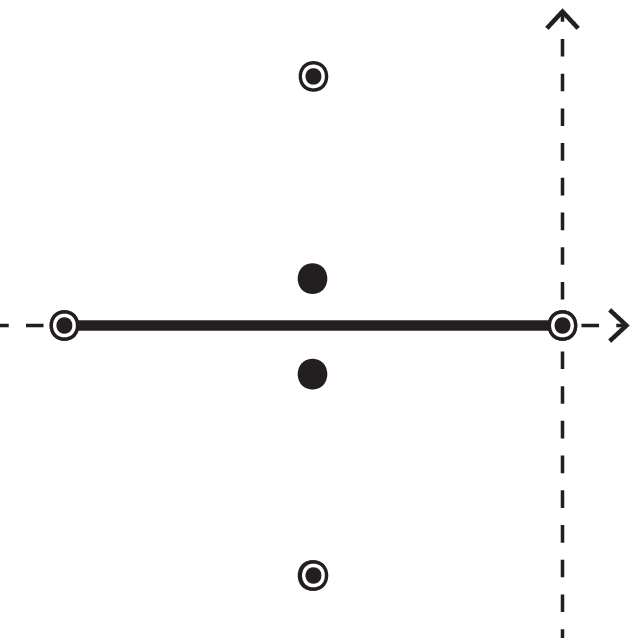}
\label{F:PathsForTwist1}}
\hspace{.2em}
\subfigure[Paths for $\widetilde{\gamma_2}$.]{\includegraphics[height=30mm]{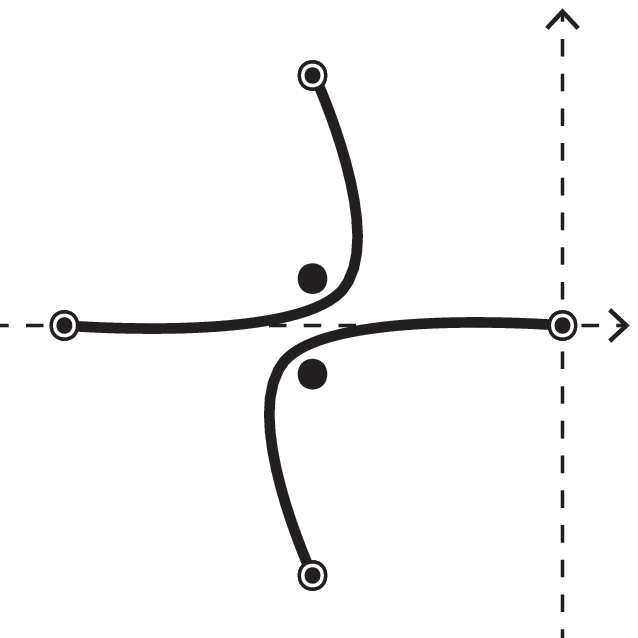}
\label{F:PathsForTwist2}}
\hspace{.2em}
\subfigure[A path for $\widetilde{\gamma_3}$.]{\includegraphics[height=30mm]{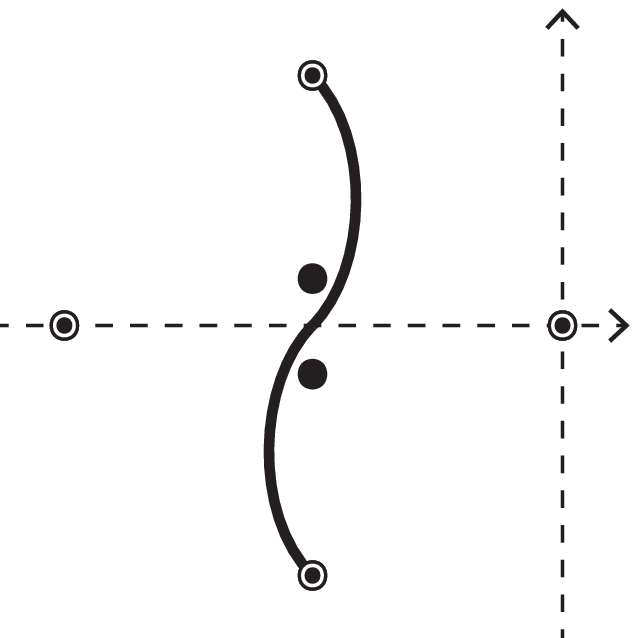}
\label{F:PathsForTwist3}}
\hspace{.2em}
\subfigure[Paths for $\widetilde{\gamma_4}$.]{\includegraphics[height=30mm]{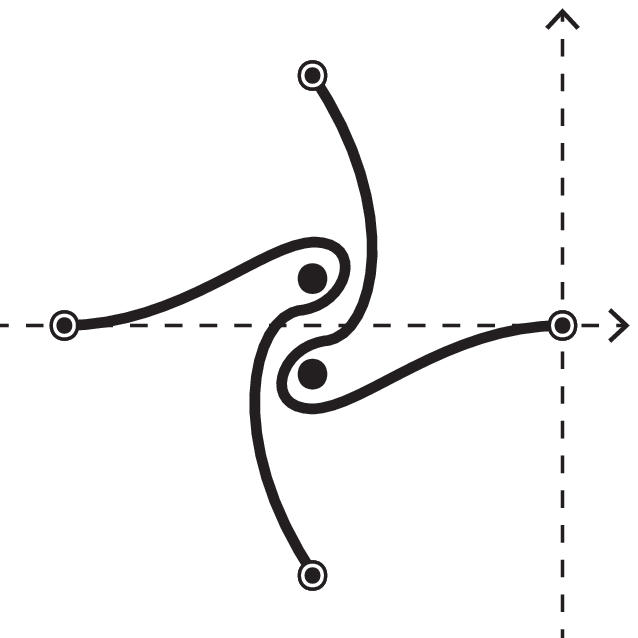}
\label{F:PathsForTwist4}}

\subfigure[A path for $\widetilde{\gamma_5}$.]{\includegraphics[height=30mm]{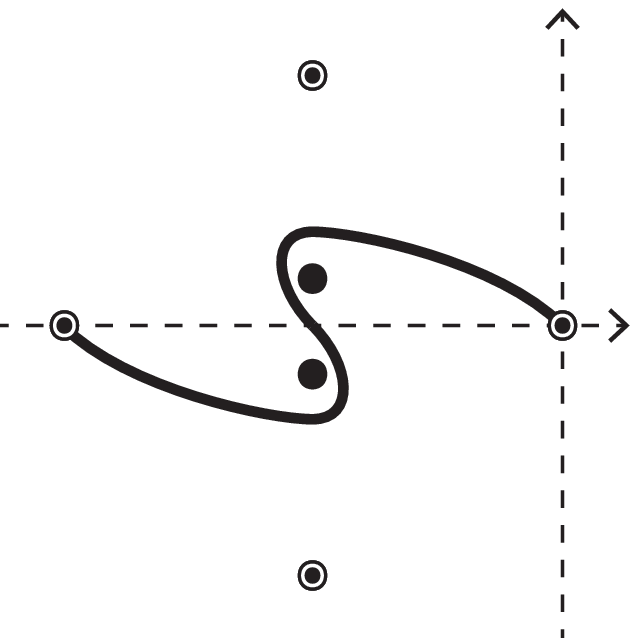}
\label{F:PathsForTwist5}}
\hspace{.2em}
\subfigure[Paths for $\widetilde{\gamma_6}$.]{\includegraphics[height=30mm]{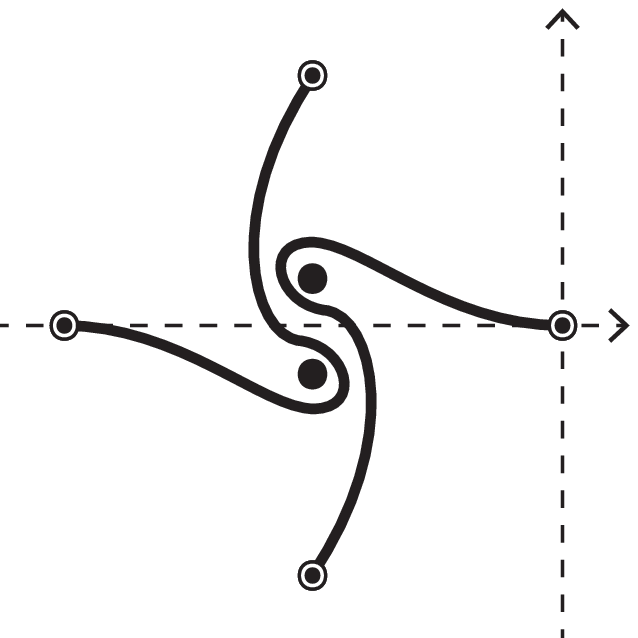}
\label{F:PathsForTwist6}}
\hspace{.2em}
\subfigure[A path for $\widetilde{\gamma_7}$.]{\includegraphics[height=30mm]{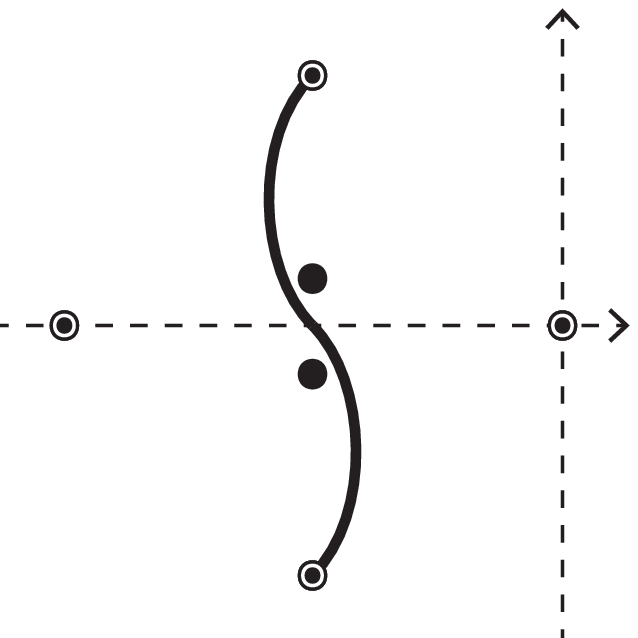}
\label{F:PathsForTwist7}}
\hspace{.2em}
\subfigure[Paths for $\widetilde{\gamma_8}$.]{\includegraphics[height=30mm]{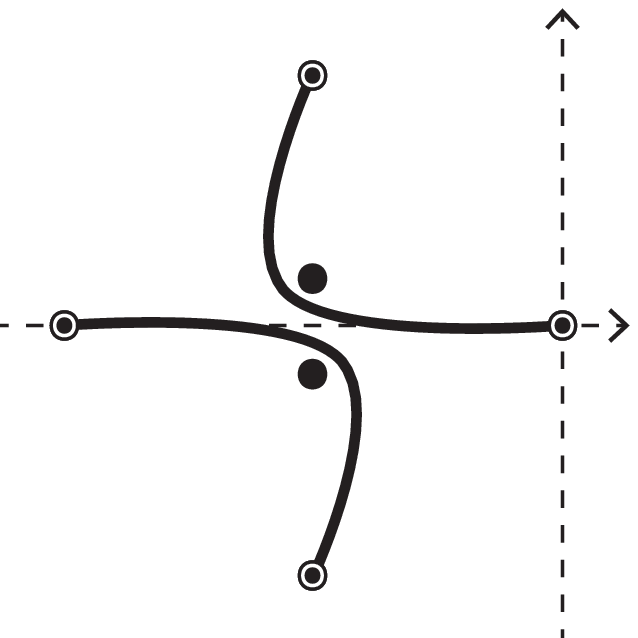}
\label{F:PathsForTwist8}}
\caption{The monodromy along $\widetilde{\gamma_i}$ is the product of the full-twists along the paths above and other twists which are less important.}
\label{F:PathsForTwist}
\end{figure} 

By Lemma~\ref{L:relation_monodromy_covering} we can obtain local monodromies of the genus-$1$ Lefschetz fibration $\pi\circ q_1\circ r:Z_{\sm}\to S$ by taking a lift of the full-twists along the paths in Figure~\ref{F:PathsForTwist} under the double covering of $\overline{\C}\cong \overline{\pi^{-1}(\varepsilon)}$ branched at the set of four points $\overline{\pi^{-1}(\varepsilon)}\cap (C_1\cup\cdots\cup C_4)$ (i.e.~the four nested dots in Figure~\ref{F:PathsForTwist}). 
The resulting local monodromies are the squares of Dehn twists along the curves in Figure~\ref{F:SCCForTwist}. 
Here, two of the four marked points in Figure~\ref{F:SCCForTwist} are the points in the preimage of $\infty \in \overline{\C}\cong \overline{\pi^{-1}(\varepsilon)}$, while the other two marked points are two points in the preimage of the intersection $\pi^{-1}(\varepsilon)\cap (C_5\cup C_6)$. 
All the marked points describe sections of the Lefschetz fibration $\pi\circ q_1\circ r$ with self-intersection $(-2)$. 
\begin{figure}[htbp]
\centering
\subfigure[A vanishing cycle associated with $\gamma_1$.]{\includegraphics[height=17mm]{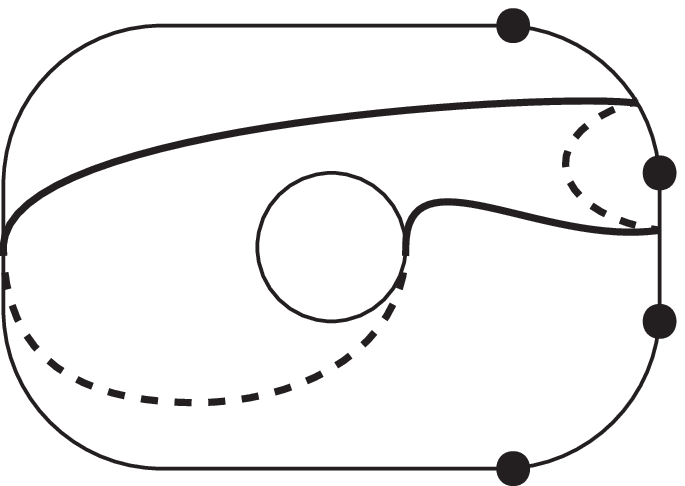}
\label{F:SCCForTwist1}}
\hspace{.2em}
\subfigure[Vanishing cycles associated with $\gamma_2$.]{\includegraphics[height=17mm]{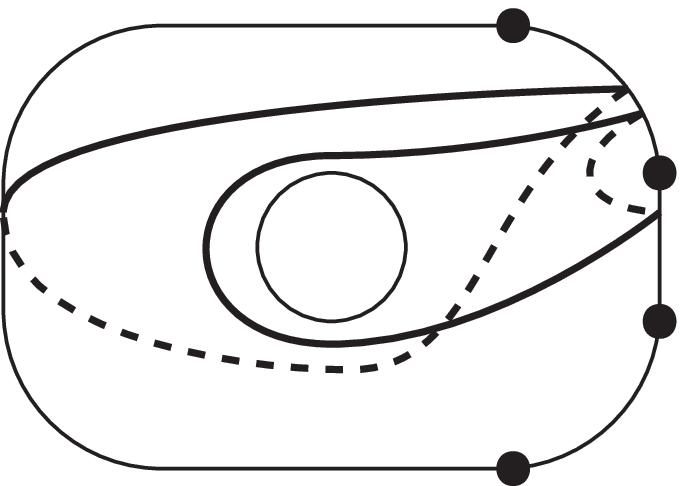}
\label{F:SCCForTwist2}}
\hspace{.2em}
\subfigure[A vanishing cycle associated with $\gamma_3$.]{\includegraphics[height=17mm]{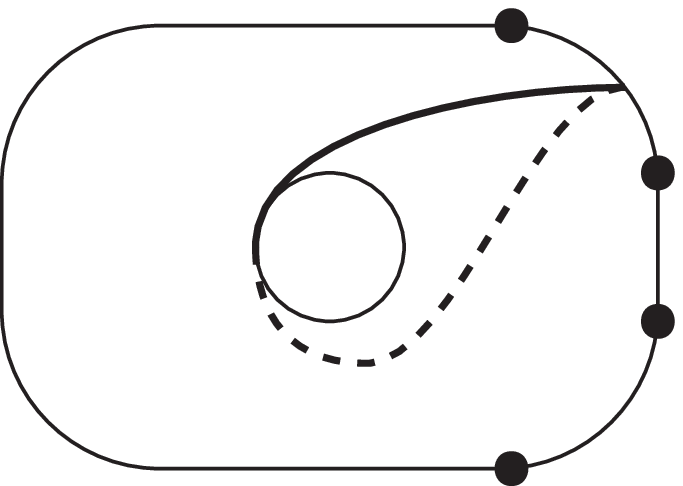}
\label{F:SCCForTwist3}}
\hspace{.2em}
\subfigure[Vanishing cycles associated with $\gamma_4$.]{\includegraphics[height=17mm]{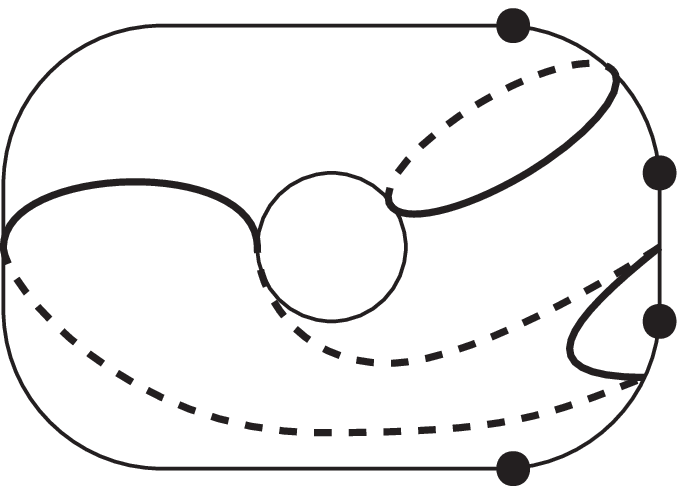}
\label{F:SCCForTwist4}}

\subfigure[A vanishing cycle associated with $\gamma_5$.]{\includegraphics[height=17mm]{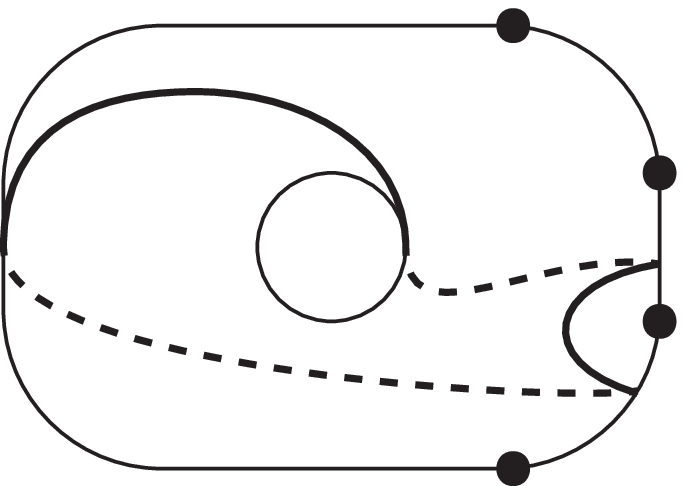}
\label{F:SCCForTwist5}}
\hspace{.2em}
\subfigure[Vanishing cycle associated with $\gamma_6$.]{\includegraphics[height=17mm]{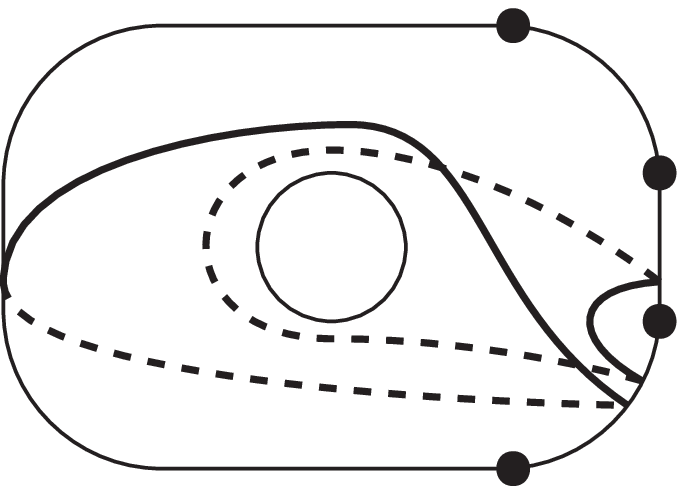}
\label{F:SCCForTwist6}}
\hspace{.2em}
\subfigure[A vanishing cycle associated with $\gamma_7$.]{\includegraphics[height=17mm]{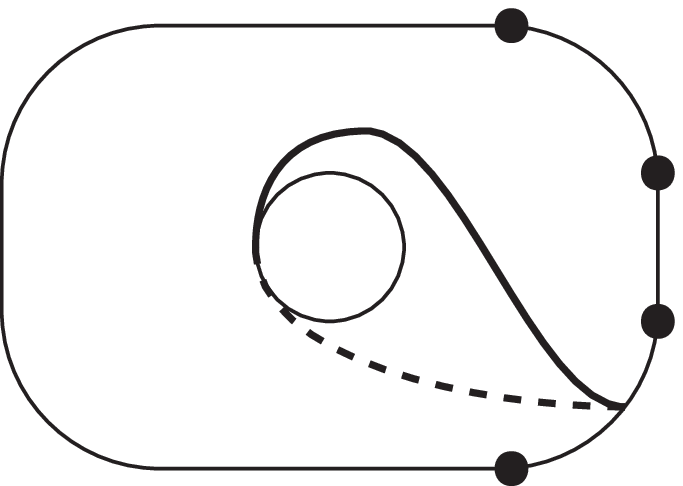}
\label{F:SCCForTwist7}}
\hspace{.2em}
\subfigure[Vanishing cycle associated with $\gamma_8$.]{\includegraphics[height=17mm]{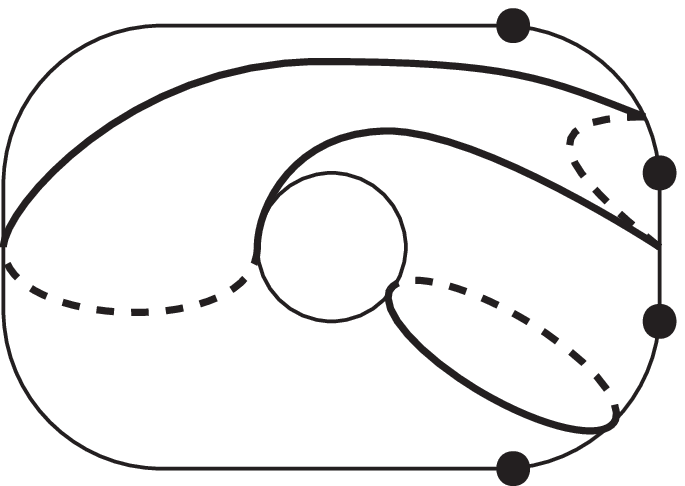}
\label{F:SCCForTwist8}}
\caption{Vanishing cycles of the Lefschetz fibration $\pi\circ q_1\circ r:Z_{\sm}\to S$. }
\label{F:SCCForTwist}
\end{figure}

In order to obtain the pencil $f:T\dashedrightarrow S$, we further take a double covering $q_2:\widetilde{T}\to Z_{\sm}$ branched at $16$ spheres with self-intersection $(-2)$. 
Four of these spheres are the sections of $\pi\circ q_1\circ r$, and the other twelve spheres are contained in singular fibers of $\pi\circ q_1\circ r$, each of which is a irreducible component of a fiber containing two Lefschetz singularities with parallel vanishing cycles. 

\begin{lemma}\label{L:DBCalongSphinFiber_VanCyc}

We denote the fiber $(\pi\circ q_1\circ r)^{-1}(\varepsilon)$ by $F\subset Z_{\sm}$.
The preimage of each vanishing cycle in Figure~\ref{F:SCCForTwist} under the restriction $q_2|_{q_2^{-1}(F)}$ is connected. 

\end{lemma}

\begin{proof}
We first observe that there is a one-to-one correspondence between the set of isomorphism classes of double coverings of a four-manifold $X$ branched at $B\subset X$ and the set of homomorphisms $\varphi:H_1(X\setminus B;\Z/2\Z)\to\Z/2\Z$ sending a meridian of each component of $B$ to $1$. 
Furthermore, for a given double covering $q:\tilde{X}\to X$ branched at $B$, the corresponding homomorphism $\varphi_q:H_1(X\setminus B;\Z/2\Z)\to\Z/2\Z$ can be obtained as follows: for a simple closed curve $c$, the value $\varphi_q([c])$ is $1$ (resp.~$0$) if the preimage $q^{-1}(c)$ is connected (resp.~disconnected). 

Let $S\subset Z_{\sm}$ be one of the twelve spheres in singular fibers of $\pi\circ q_1\circ r$ and $N_1\subset Z_{\sm}$ a tubular neighborhood of $S$. 
The restriction of $\pi\circ q_1\circ r$ on $N_1$ has two Lefschetz singularities and a regular fiber of this restriction is an annulus. 
According to \cite[\S.8.2]{GS} we can draw a handlebody picture of the closure $\overline{N_1}$ which reflects configuration of the two singularities as shown in Figure~\ref{F:NeighborhoodSphere1} (two $(-1)$-framed knots correspond with the two Lefschetz singularities). 
\begin{figure}[htbp]
\centering
\subfigure[The closure $\overline{N_1}$]{\includegraphics[height=30mm]{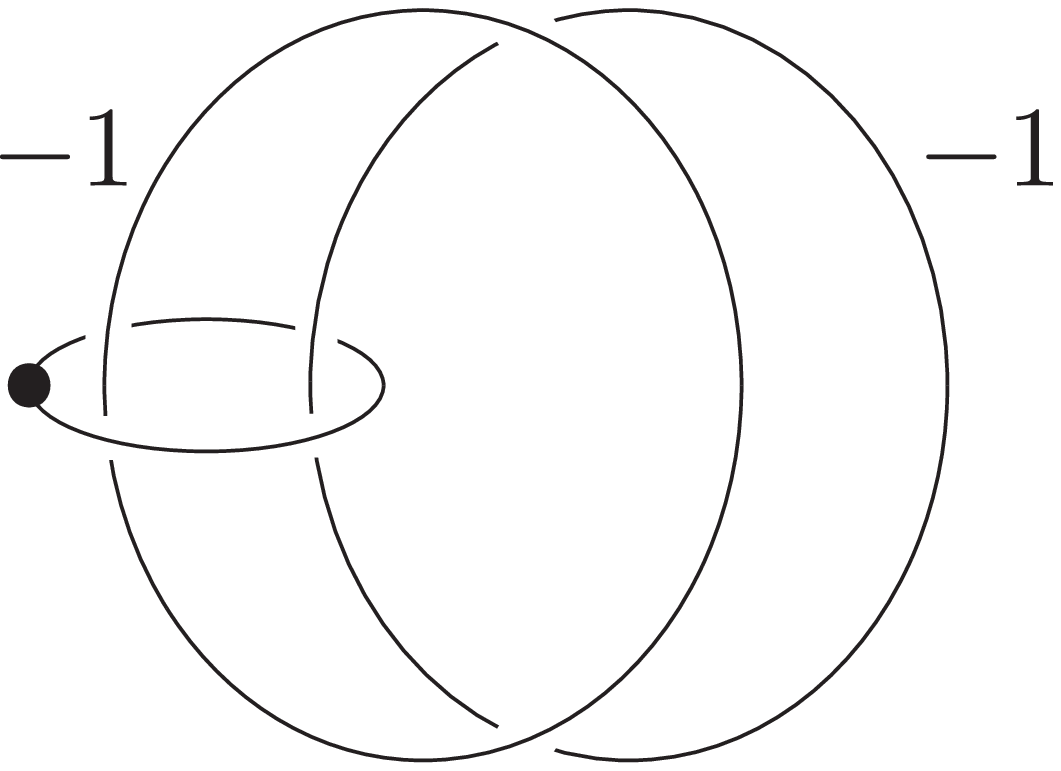}
\label{F:NeighborhoodSphere1}}
\hspace{.2em}
\subfigure[The complement $\overline{N_1}\setminus N_2$.]{\includegraphics[height=30mm]{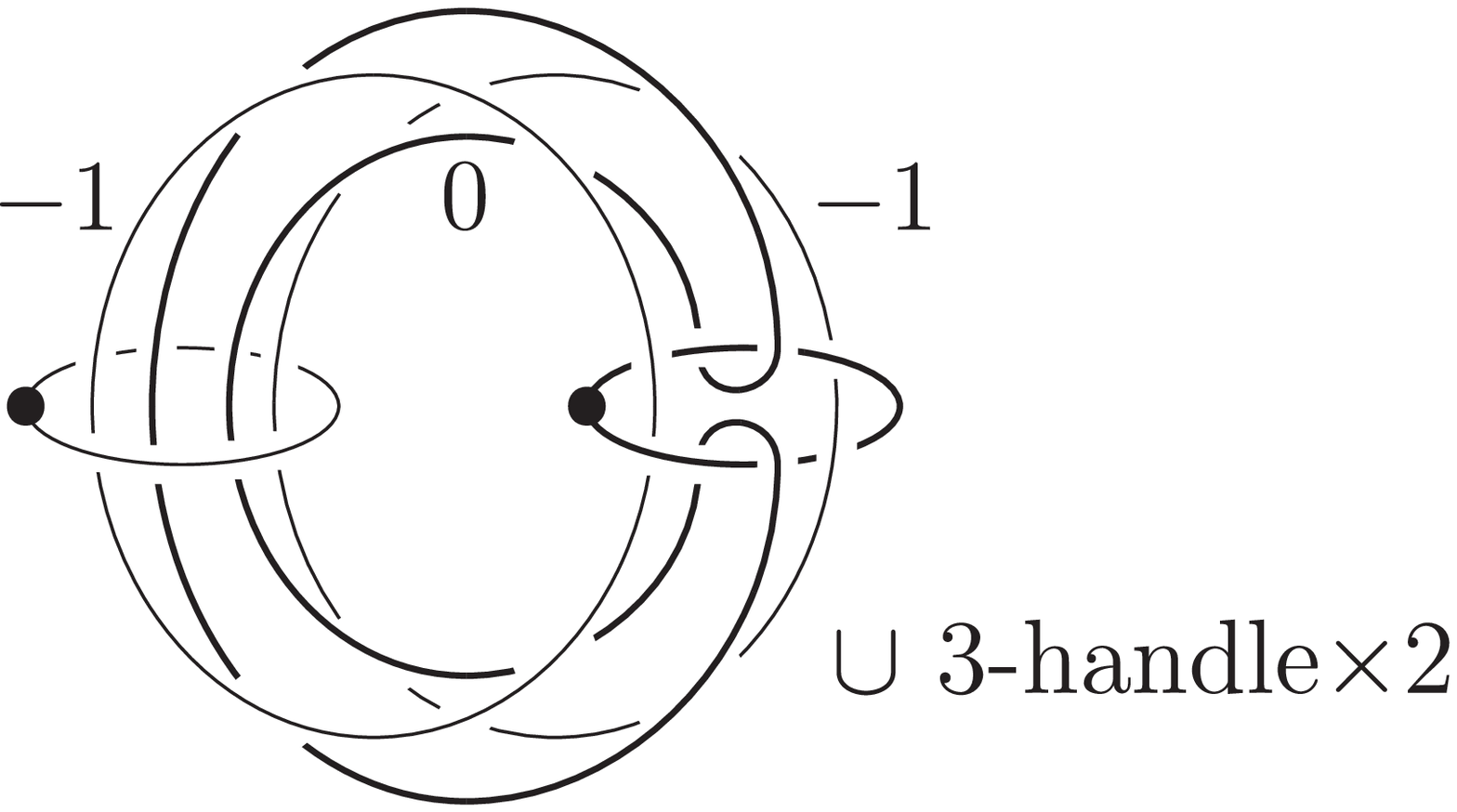}
\label{F:NeighborhoodSphere2}}
\caption{Handlebody pictures of a neighborhood $\overline{N_1}$ of $S$ and the complement of a smaller neighborhood of $S$ in $\overline{N_1}$.}
\label{F:NeighborhoodSphere}
\end{figure}
Moreover, applying the algorithm in \cite[\S.6.2]{GS} to our situation, we can obtain a diagram of the complement $\overline{N_1}\setminus N_2$ of a smaller tubular neighborhood $N_2$ of $S$ as shown in Figure~\ref{F:NeighborhoodSphere2} (the bold handles and the two $3$-handles in the figure correspond with handles of $S$). 

Let $c_1,\ldots,c_{12}\subset \Sigma_{1,4}$ be the vanishing cycles in the fourth-punctured torus $\Sigma_{1,4}$ described in Figure~\ref{F:SCCForTwist}. 
It is easy to verify (by drawing a handlebody picture using the observation above) that the first homology of the complement of the sixteen spheres in $Z_{\sm}$ is isomorphic to the following group:
\[
\left(H_1(\Sigma_1^4;\Z/2\Z) \oplus_{i=1}^{12}(\Z/2\Z e_i)\right) / \left<\left\{c_i+e_i~|~i=1,\ldots,12\right\}\right>,
\]
where $e_i$'s coincide with the meridians of the spheres in singular fibers.  
As we observed above, the homomorphism $\varphi_{q_2}$ associated with the branched covering $q_2$ must send each $e_i$ to $1$. 
Since $c_i$ is equal to $e_i$ in the group above, the preimage $q_2^{-1}(c_i)$ is connected. 
\end{proof}

\noindent
Since the vanishing cycles $c_1,\ldots,c_{12}$ span the homology group $H_1(T^2;\Z/2\Z)$, the argument in the proof of Lemma~\ref{L:DBCalongSphinFiber_VanCyc} also shows that a double covering of $T^2$ branched at the marked points by which each loop $c_i$ cannot be lifted is unique up to isomorphism. 
In particular, we can obtain vanishing cycles of the pencil $f:T\dashedrightarrow S$ once we can find such a branched covering, which is obtained by dividing $\Sigma_3$ by the involution $\eta$ shown in Figure~\ref{F:involution_genus3}. 
\begin{figure}[htbp]
\centering
\includegraphics[height=23mm]{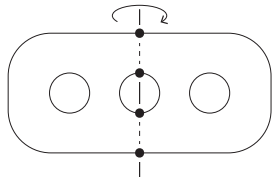}
\caption{The involution $\eta$, which is the $\pi$-degree rotation around the axis.}
\label{F:involution_genus3}
\end{figure}
Taking the preimage of the vanishing cycles in Figure~\ref{F:SCCForTwist} by the branched covering induced by $\eta$, we can eventually obtain vanishing cycles $\tilde{c}_1,\ldots,\tilde{c}_{12}$ of $f:T\dashedrightarrow S$ as shown in Figure~\ref{F:VC_Smith}, and thus the monodromy factorization associated with $f$: 
\begin{equation}
t_{\tilde{c}_{12}}\cdots t_{\tilde{c}_1}=t_{\delta_1}\cdots t_{\delta_4}. \label{eq:Smith'sLP}
\end{equation}
%
\begin{figure}[htbp]
\centering
\subfigure[The curve $\tilde{c}_{12}$.]{\includegraphics[height=15mm]{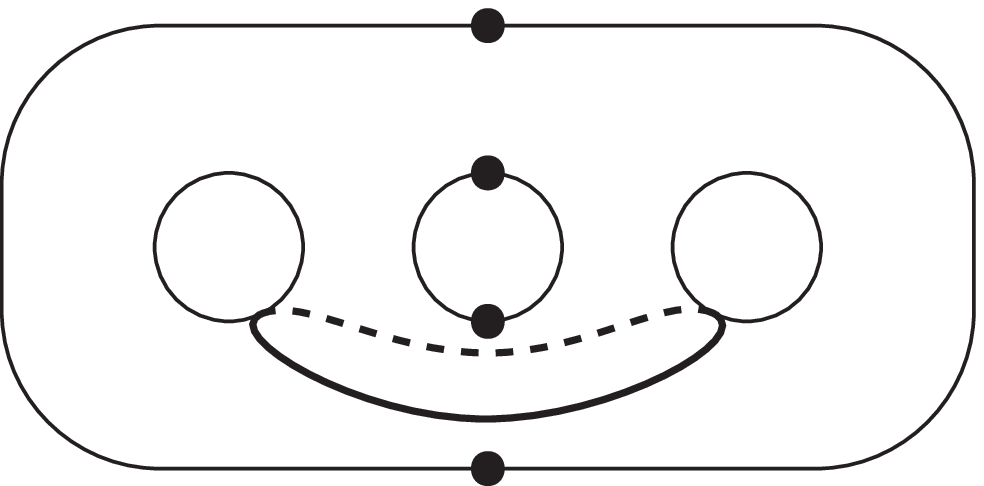}
\label{F:VC_Smith12}}
\subfigure[The curve $\tilde{c}_{11}$.]{\includegraphics[height=15mm]{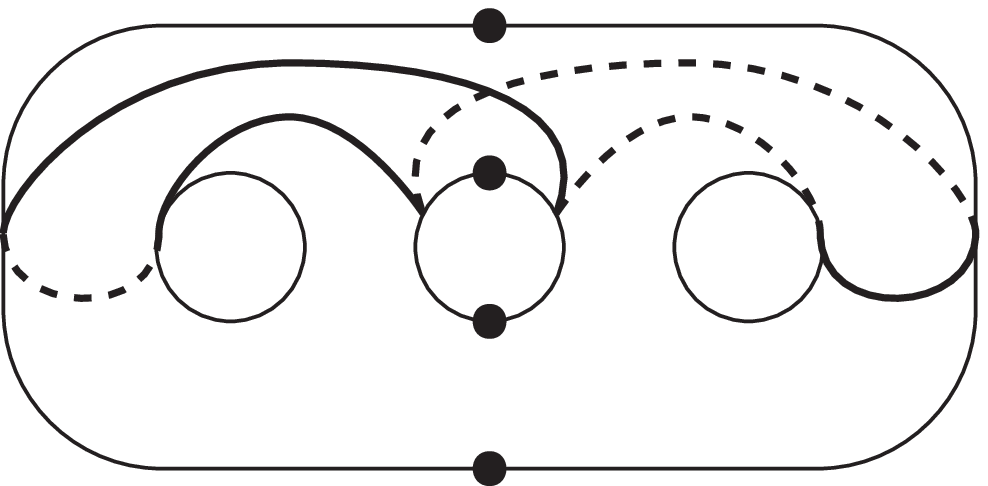}
\label{F:VC_Smith11}}
\subfigure[The curve $\tilde{c}_{10}$.]{\includegraphics[height=15mm]{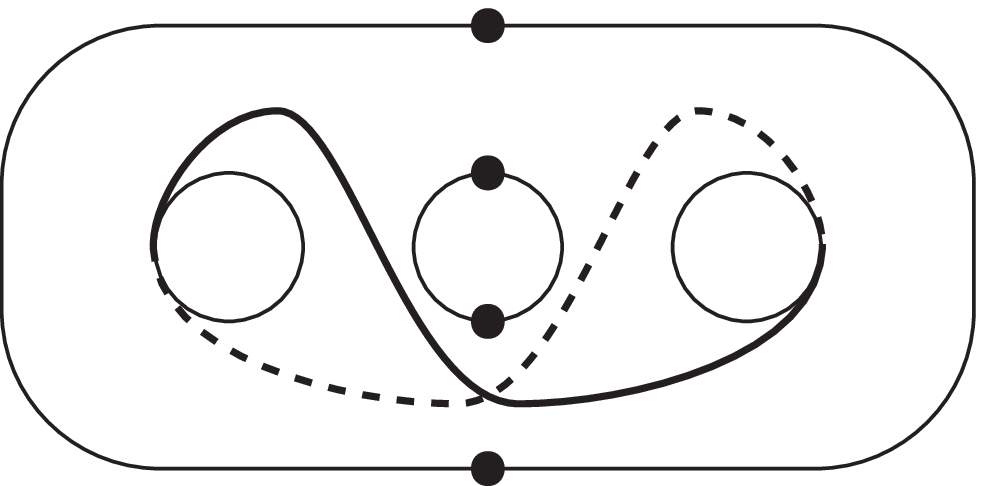}
\label{F:VC_Smith10}}
\subfigure[The curve $\tilde{c}_{9}$.]{\includegraphics[height=15mm]{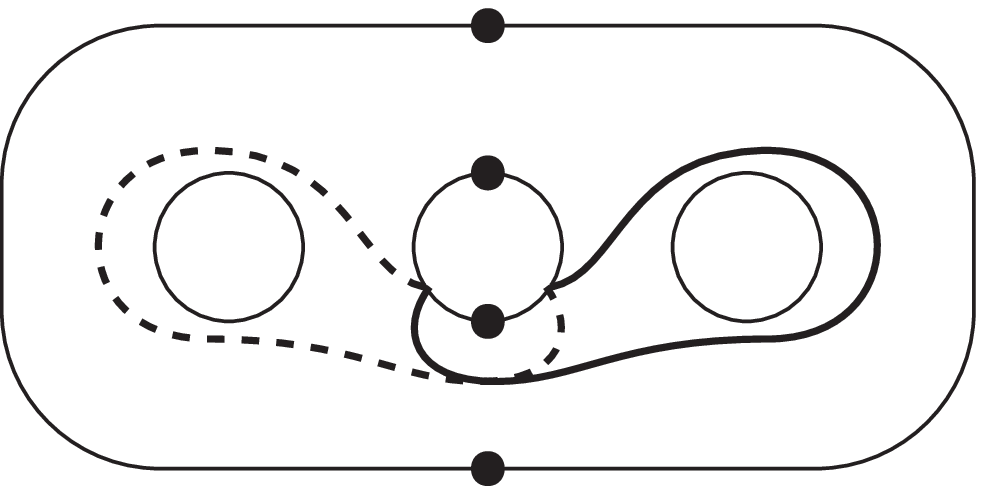}
\label{F:VC_Smith9}}
\subfigure[The curve $\tilde{c}_{8}$.]{\includegraphics[height=15mm]{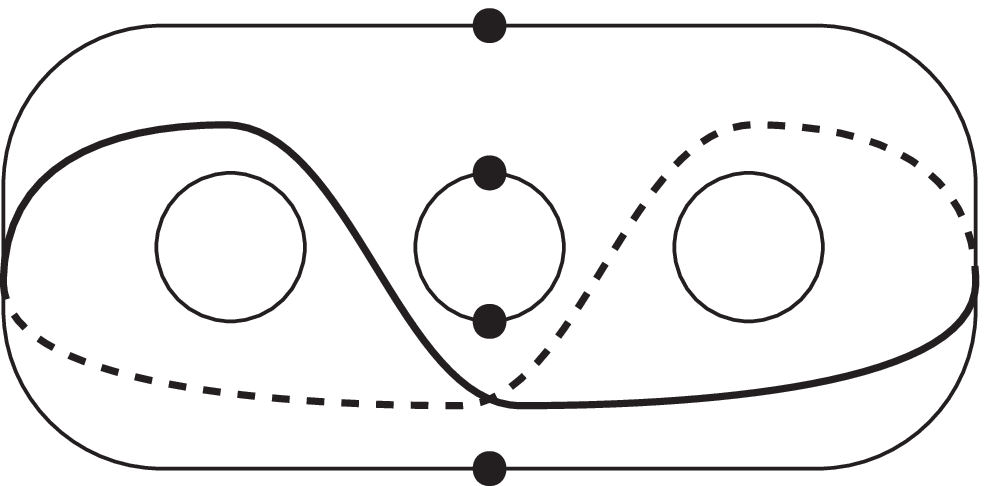}
\label{F:VC_Smith8}}
\subfigure[The curve $\tilde{c}_{7}$.]{\includegraphics[height=15mm]{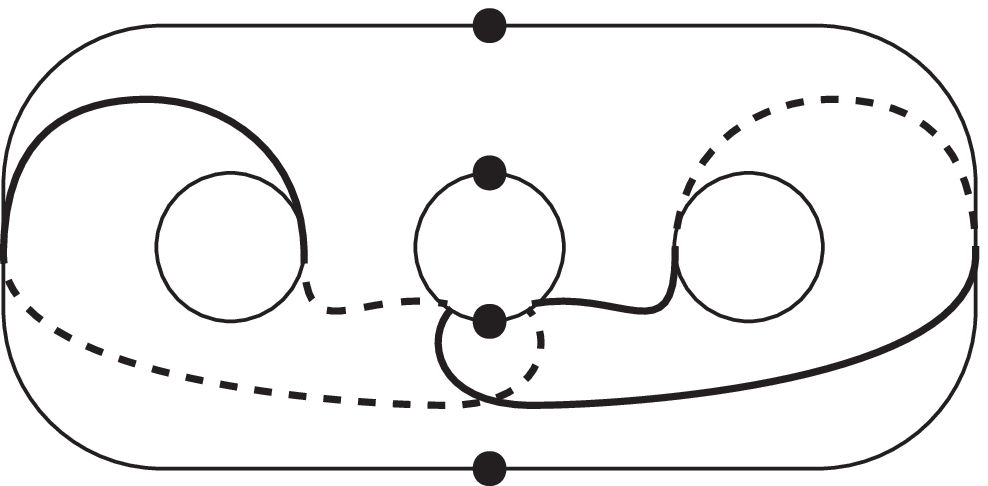}
\label{F:VC_Smith7}}
\subfigure[The curve $\tilde{c}_{6}$.]{\includegraphics[height=15mm]{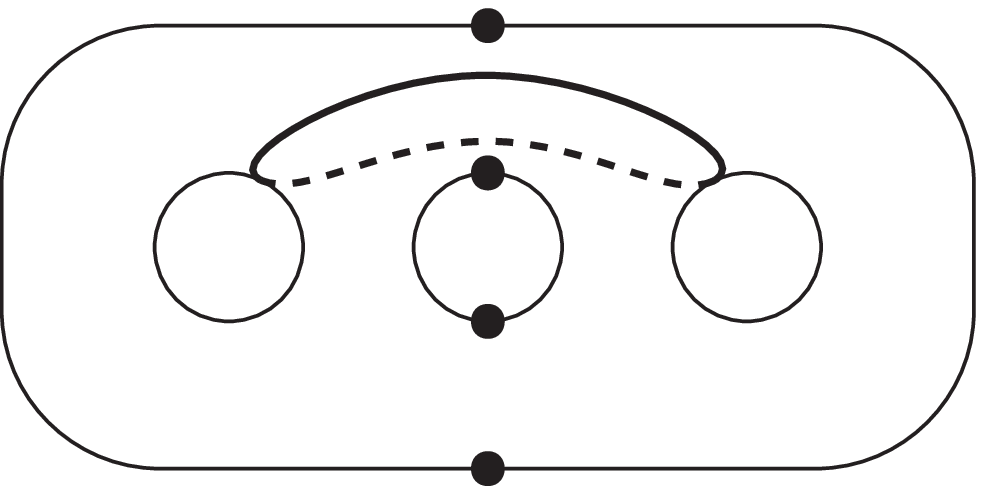}
\label{F:VC_Smith6}}
\subfigure[The curve $\tilde{c}_{5}$.]{\includegraphics[height=15mm]{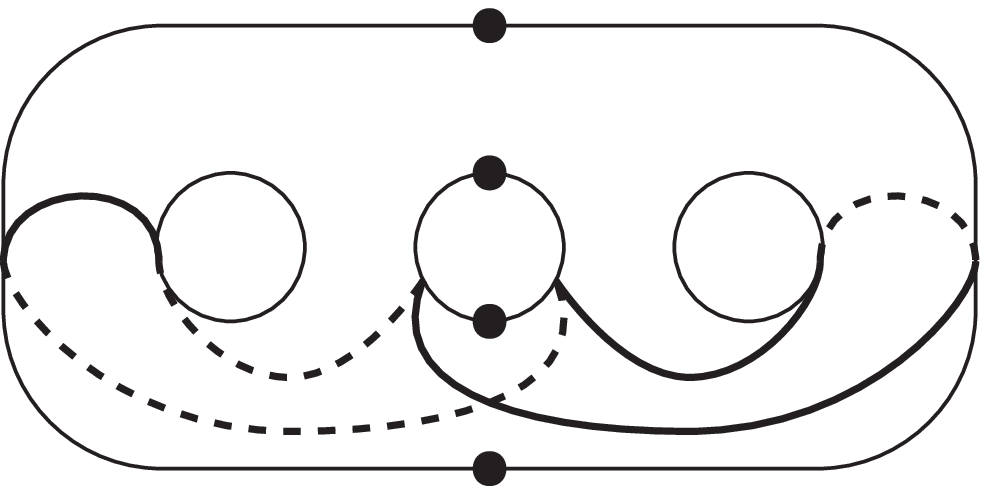}
\label{F:VC_Smith5}}
\subfigure[The curve $\tilde{c}_{4}$.]{\includegraphics[height=15mm]{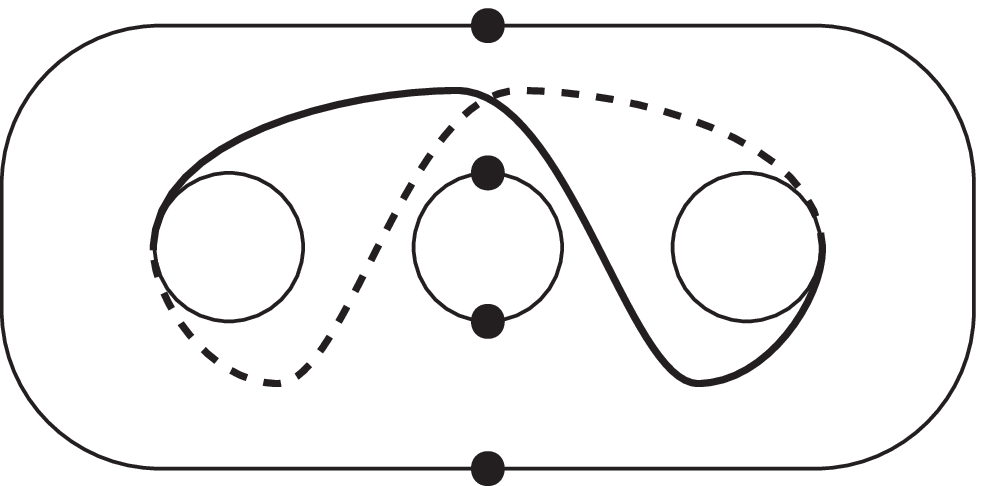}
\label{F:VC_Smith4}}
\subfigure[The curve $\tilde{c}_{3}$.]{\includegraphics[height=15mm]{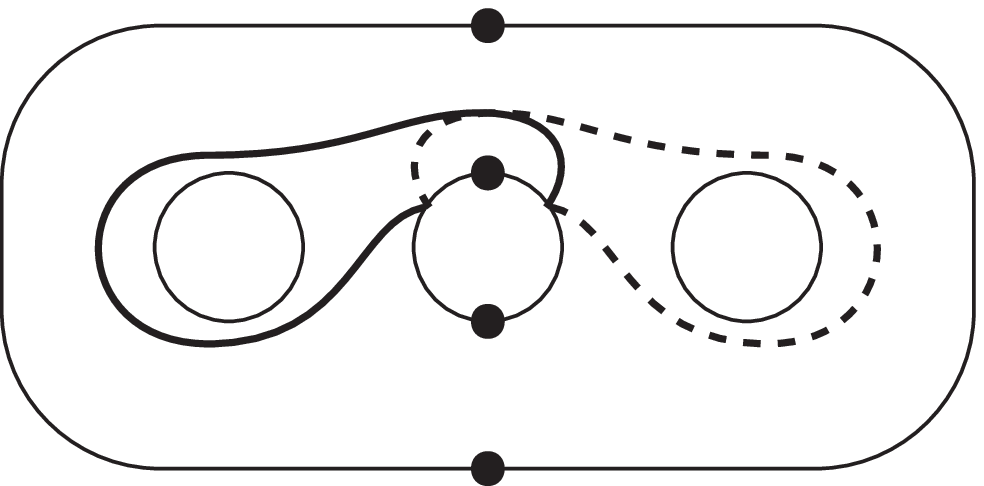}
\label{F:VC_Smith3}}
\subfigure[The curve $\tilde{c}_{2}$.]{\includegraphics[height=15mm]{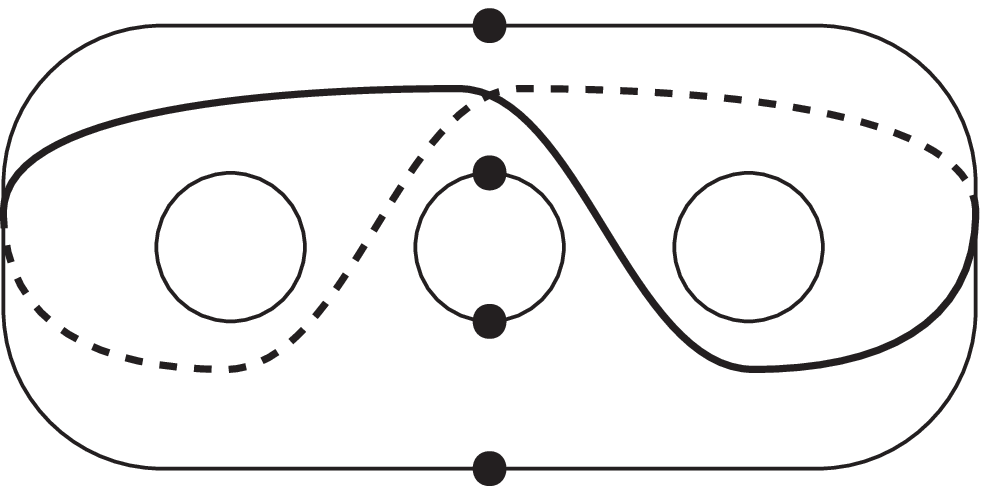}
\label{F:VC_Smith2}}
\subfigure[The curve $\tilde{c}_{1}$.]{\includegraphics[height=15mm]{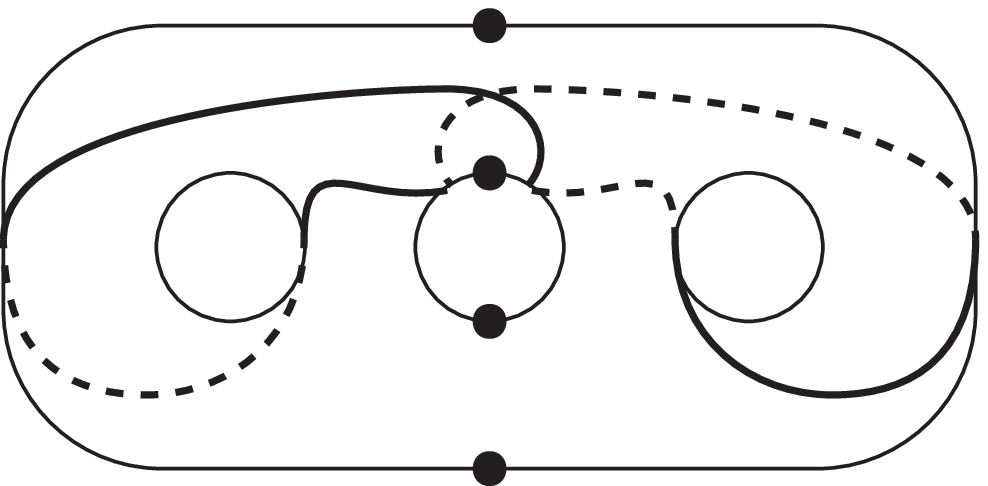}
\label{F:VC_Smith1}}
\caption{Vanishing cycles of a holomorphic genus-$3$ pencil on $T^4$ due to Smith. }
\label{F:VC_Smith}
\end{figure}

\begin{remark}

To be precise, the pencil we have constructed here is not a Lefschetz pencil yet since it does not satisfy the condition (1) in p.\pageref{P:definition_LP}. 
However, we can obtain a holomorphic genus-$3$ Lefschetz pencil on $T^4$ by perturbing the conics $C_1,\ldots,C_6$ so that the restriction of $\pi$ on the set of double points of $C_1\cup \cdots \cup C_4$ becomes injective. 
We can further check (using Mathematica) that the monodoromy factorization of the Lefschetz pencil obtained in this way is Hurwitz equivalent to that of our pencil. 

\end{remark}

\begin{remark}

Recently, Baykur~\cite{Baykur_preprint_genus3LP} has also constructed a genus-$3$ symplectic Calabi-Yau Lefschetz pencil whose total space is \emph{homeomorphic} to the standard four-torus $T^4$, but the diffeomorphism type was unknown.
In addition, the geometric structure of the pencil is not clear since his construction is based on a purely combinatorial method in terms of relations among Dehn twists.
In Section~\ref{Section:CombinatorialApproach}, we will see that his pencil is in fact isomorphic to the pencil corresponding to~(\ref{eq:Smith'sLP}) (see Remark~\ref{R:HurwitzEquivalenceBaykur}) after observing some arguments on combinatorial structures of the factorization~(\ref{eq:Smith'sLP}).
Thus, we now understand the detail of geometric structure of Baykur's pencil, in particular, his pencil is not only homeomorphic but also \emph{diffeomorphic} to the standard $T^4$, and the pencil may be considered \emph{holomorphic}.
\end{remark}

\subsection{Holomorphic Lefschetz pencils with higher genera}

According to Theorem~\ref{T:main_uniquenessLP} and Lemma~\ref{L:connectedness_lines}, for any integers $d_1,d_2>0$ with $d_1|d_2$ and $d_1d_2\geq 5$ there exists a genus-$(d_1d_2+1)$ holomorphic Lefschetz pencil on $T^4$ with divisibility $d_1$ and such a Lefschetz pencil is unique up to isomorphism. 
In this subsection we will explain how to obtain monodromy factorizations of some of these Lefschetz pencils. 

Let $c\in H^2(T^4;\Z)$ be a $(d_1,d_2)$-polarization of $T^4$. 
The cohomology class $c$ is equal to $d_1 \alpha_1\cup \beta_1 + d_2\alpha_2\cup \beta_2$ for some generating system $\alpha_1,\beta_1,\alpha_2,\beta_2$ of $H^1(T^4;\Z)$. 
Let $a_i,b_j\in H_1(T^4;\Z)$ be respectively duals of $\alpha_i,\beta_j$ with respect to the Kronecker product. 
We take an unbranched covering map $q:\tilde{T} \to T^4$ corresponding to the subgroup of $H_1(T^4;\Z)$ generated by $n_1a_1,b_1,n_2a_2,b_2$ for some $n_1,n_2\in\Z$.
It is easy to see that $\tilde{T}$ is again a $4$-torus and $\left\{\frac{q^\ast(\alpha_1)}{n_1}, q^\ast(\beta_1),\frac{q^\ast(\alpha_2)}{n_2}, q^\ast(\beta_2)\right\}$ generates $H^1(\tilde{T};\Z)$. 
In particular the pull-back $q^\ast(c)\in H^2(\tilde{T};\Z)$ is a $(\tilde{d}_1,\tilde{d}_2)$-polarization of $\tilde{T}$, where $\tilde{d}_1 = \gcd(n_1d_1,n_2d_2)$ and $\tilde{d}_2=\frac{n_1n_2d_1d_2}{\tilde{d}_1}$. 
Since $d_2$ is divisible by $d_1$, $\tilde{d}_1$ is also divisible by $d_1$. 
Furthermore, $\tilde{d}_2$ is divisible by $d_2$ since $\frac{n_1d_1}{\tilde{d}_1}$ must be an integer. 
Conversely, for any positive integers $l_1,l_2$, the pull-back $\overline{q}^{\ast}(c)$ by an unbranched covering map $\overline{q}$ corresponding to the subgroup $\left<l_1a_1,b_1,l_2a_2,b_2\right>\subset H_1(T^4;\Z)$ is an $(l_1d_1,l_2d_2)$-polarization of $\tilde{T}$.  
We can thus obtain the following: 

\begin{lemma}\label{L:relation_covering_polarization}

Let $c\in H^2(T^4;\Z)$ be a $(d_1,d_2)$-polarization of $T^4$. 
For any integers $\tilde{d}_1,\tilde{d}_2$ with $\tilde{d}_1|\tilde{d}_2$ and $d_i|\tilde{d}_i$ ($i=1,2$), there exists an unbranched covering map $q:\tilde{T}\to T^4$ such that the pull-back $q^\ast(c)$ is a $(\tilde{d}_1,\tilde{d}_2)$-polarization of $\tilde{T}$. 
If $q$ is an $n$-fold unbranched covering, $\tilde{d}_1\tilde{d}_2$ is equal to $nd_1d_2$. 

\end{lemma}

Let $f:T^4 \dashedrightarrow \CP^1$ be a holomorphic pencil associated with a $(d_1,d_2)$-polarization $c \in H^2(T^4;\Z)$ and $q:\tilde{T}\to T^4$ a finite unbranched covering map. 
It is easily verify that the composition $f\circ q:\tilde{T}\dashedrightarrow \CP^1$ is also a holomorphic pencil associated with the polarization $q^\ast (c)$. 
By Lemma~\ref{L:relation_covering_polarization} for any $d_1,d_2$ with $d_1d_2$ even, we can obtain a holomorphic pencil on $T^4$ associated with a $(d_1,d_2)$-polarization by composing a finite unbranched covering map with the genus-$3$ pencil in the preceding subsection. 
We can perturb this pencil so that it becomes a Lefschetz pencil (see Lemma~\ref{L:connectedness_lines}). 
We can further obtain the vanishing cycles of the pencil $f\circ q$ using Lemma~\ref{L:relation_monodromy_covering} once we can find out how the deck transformations of $q$ act on a reference fiber of $f\circ q$. 
In what follows we will apply the above procedure to obtain two holomorphic pencils on $T^4$ with the same genus but distinct divisibilities. 

\begin{example}

Let $f:T^4\dashedrightarrow \CP^1$ be the holomorphic pencil obtained in the preceding subsection. 
According to Lemma~\ref{L:relation_covering_polarization}, for a double unbranched covering $q:\tilde{T}\to T^4$ the type of a polarization associated with the composition $f\circ q$ is either $(1,4)$ or $(2,2)$. 
We will give two double unbranched coverings which yield both of the types of polarizations below. 

It is easy to see that $H_1(T^4;\Z)$ is generated by the elements represented by the curves $a_1,b_1,a_2,b_2$ shown in Figure~\ref{F:generator_homology} (which are contained in a reference fiber of $f$). 
\begin{figure}[htbp]
\centering
\includegraphics[width=35mm]{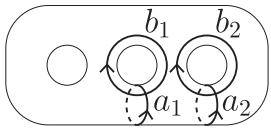}
\caption{The curves generating $H_1(T^4;\Z)$. }
\label{F:generator_homology}
\end{figure}
Let $q_1:\tilde{T}_1\to T^4$ be a double unbranched covering corresponding to the subgroup $\left<a_1,2b_1,a_2,b_2\right>\subset H_1(T^4;\Z)$.
The restriction of $q_1$ on the preimage of a reference fiber of $f$ is the quotient map by the involution $\eta_1$ shown in Figure~\ref{F:involution_eta1}, in particular the restriction of the deck transformation of $q_1$ is equal to $\eta_1$. 
\begin{figure}[htbp]
\centering
\includegraphics[width=80mm]{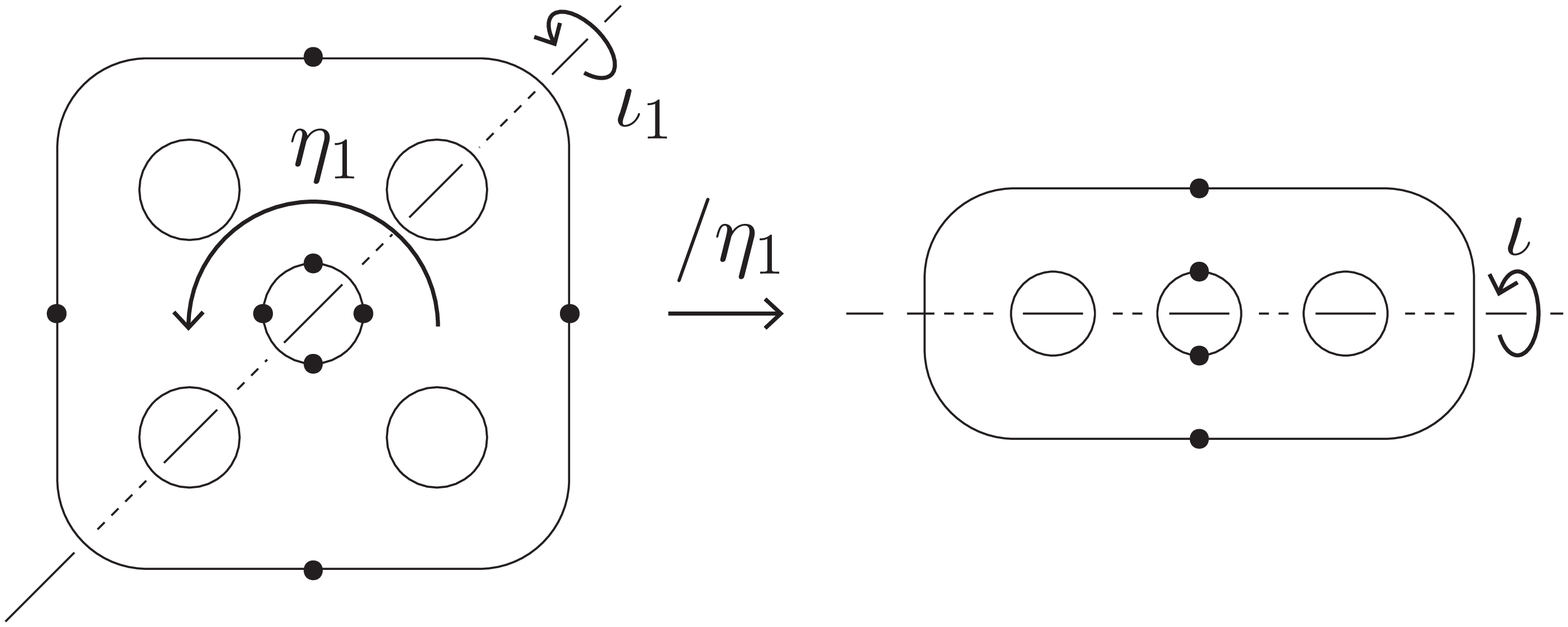}
\caption{The involutions $\eta_1$ and $\iota_1$.
Both of them are $\pi$-degree rotations.}
\label{F:involution_eta1}
\end{figure}
Let $\iota_1$ be the involution of a fiber of $f\circ q$ shown in Figure~\ref{F:involution_eta1}, which is a lift of the hyperelliptic involution $\iota$ of the genus-$3$ fiber. 
For any $i =1,\ldots,6$ we take a lift $d_i$ of $\tilde{c}_i$ in Figure~\ref{F:VC_Smith} under the unbranched covering map $q_1$ as shown in Figure~\ref{F:liftVCSmith1}.
\begin{figure}[htbp]
\centering
\subfigure[The curves $d_1$ and $\eta_1(d_1)$. ]{\includegraphics[height=32mm]{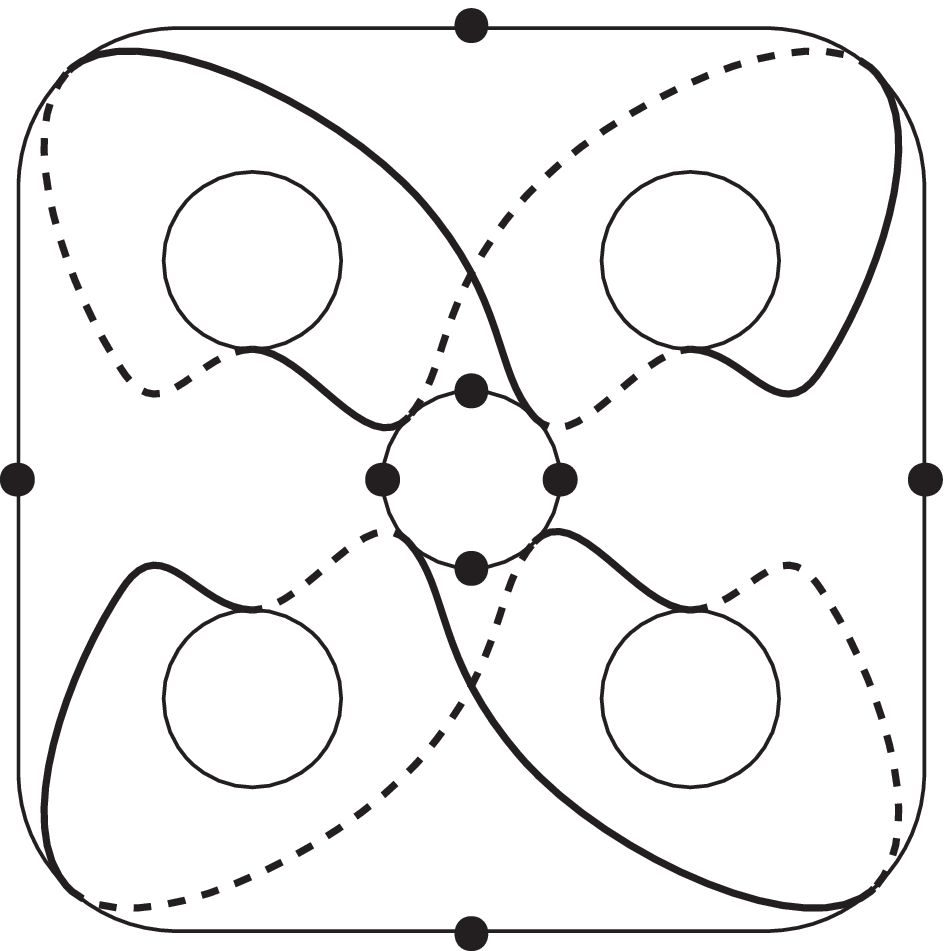}
\label{F:lift_VCSmith1_1}}
\hspace{.8em}
\subfigure[The curves $d_2$ and $\eta_1(d_2)$.]{\includegraphics[height=32mm]{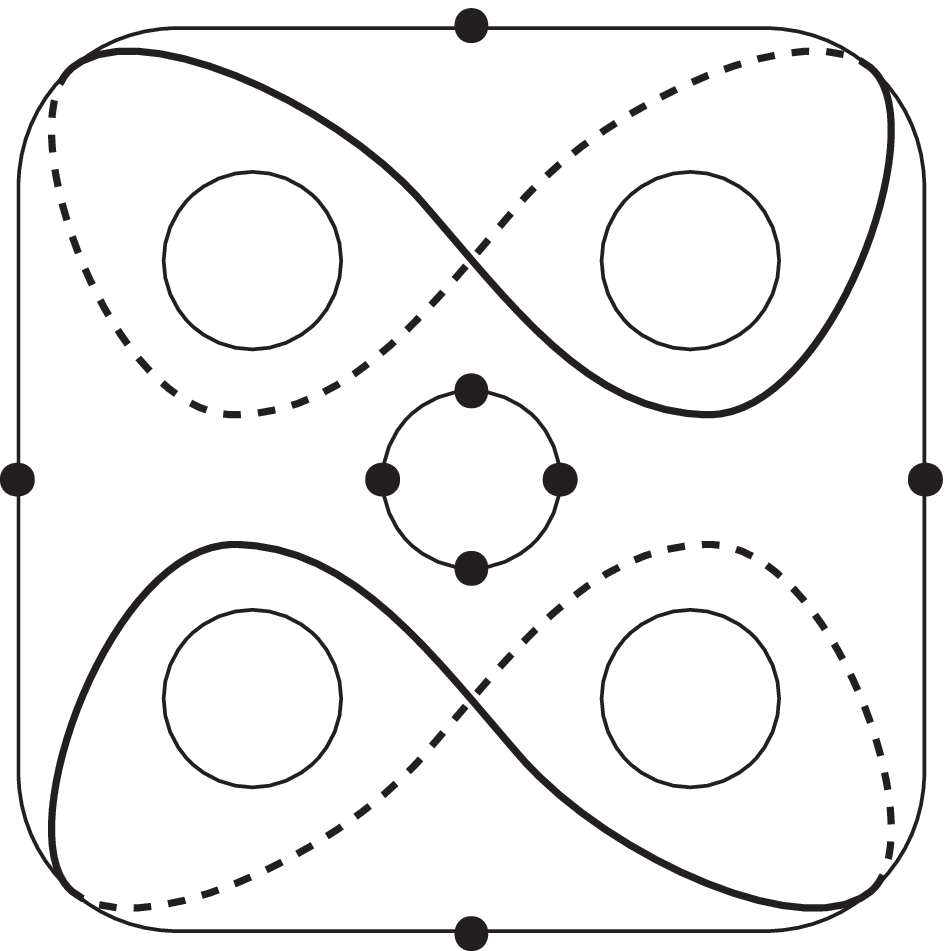}
\label{F:lift_VCSmith1_2}}
\hspace{.8em}
\subfigure[The curves $d_3$ and $\eta_1(d_3)$.]{\includegraphics[height=32mm]{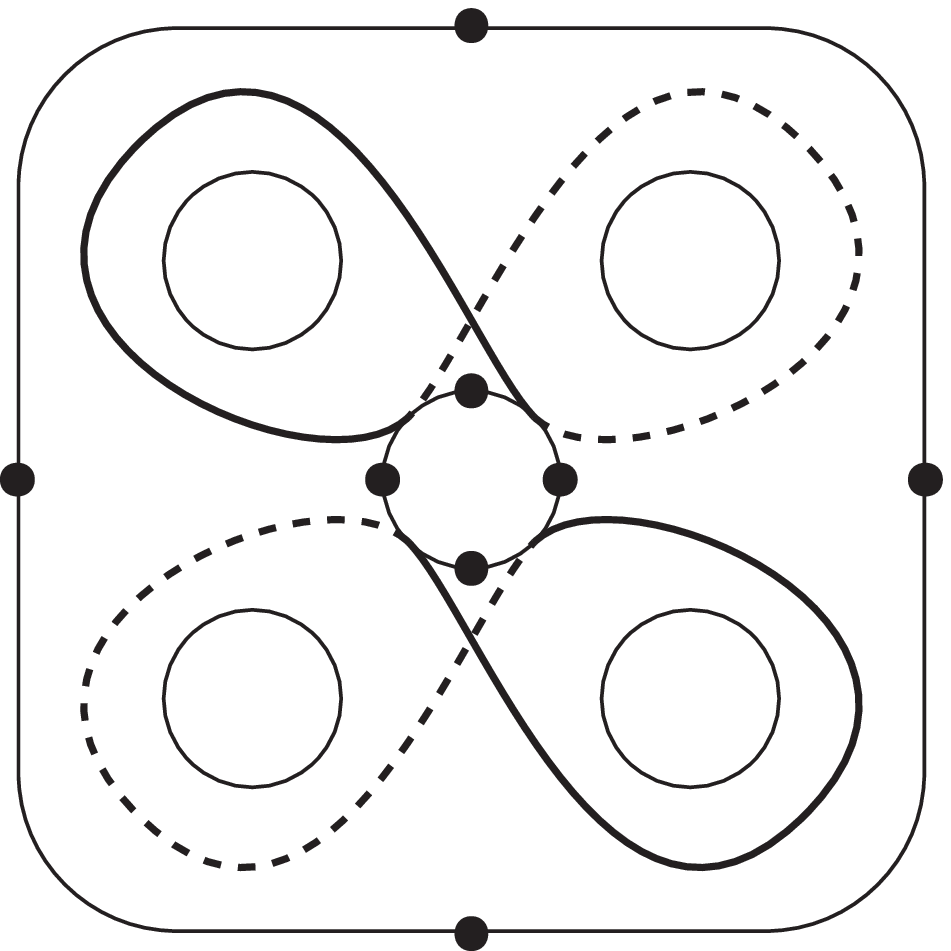}
\label{F:lift_VCSmith1_3}}

\subfigure[The curves $d_4$ and $\eta_1(d_4)$.]{\includegraphics[height=32mm]{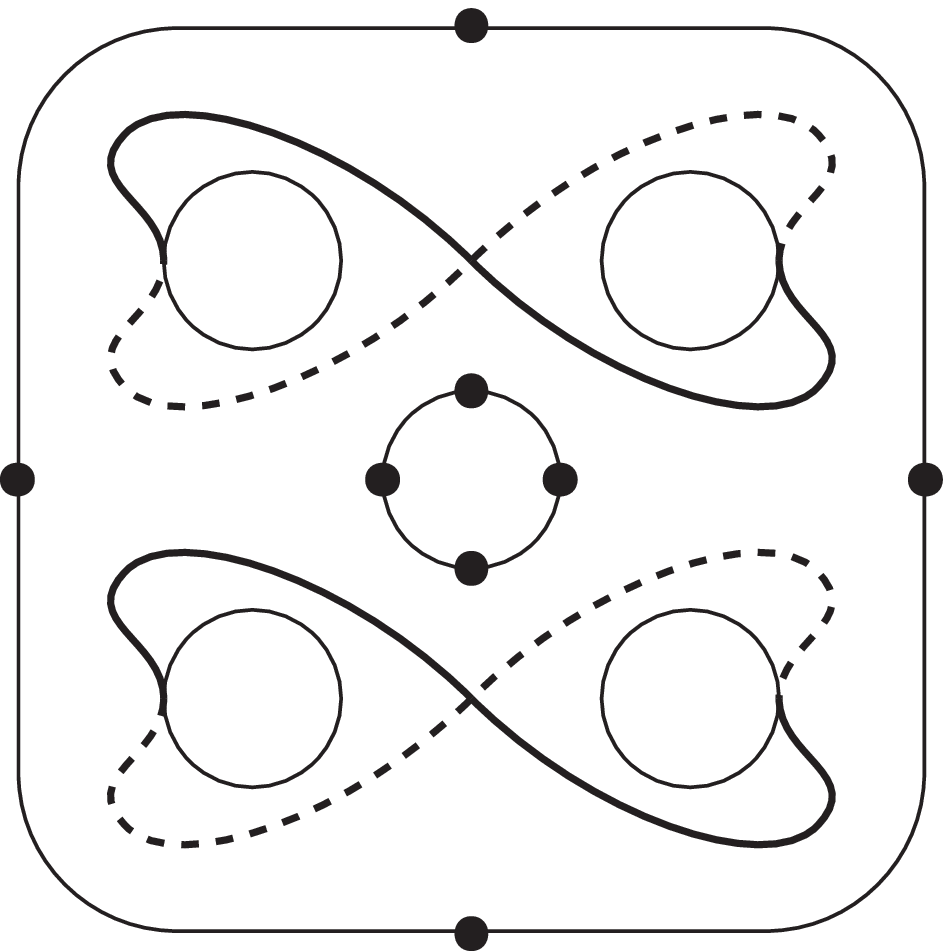}
\label{F:lift_VCSmith1_4}}
\hspace{.8em}
\subfigure[The curves $d_5$ and $\eta_1(d_5)$.]{\includegraphics[height=32mm]{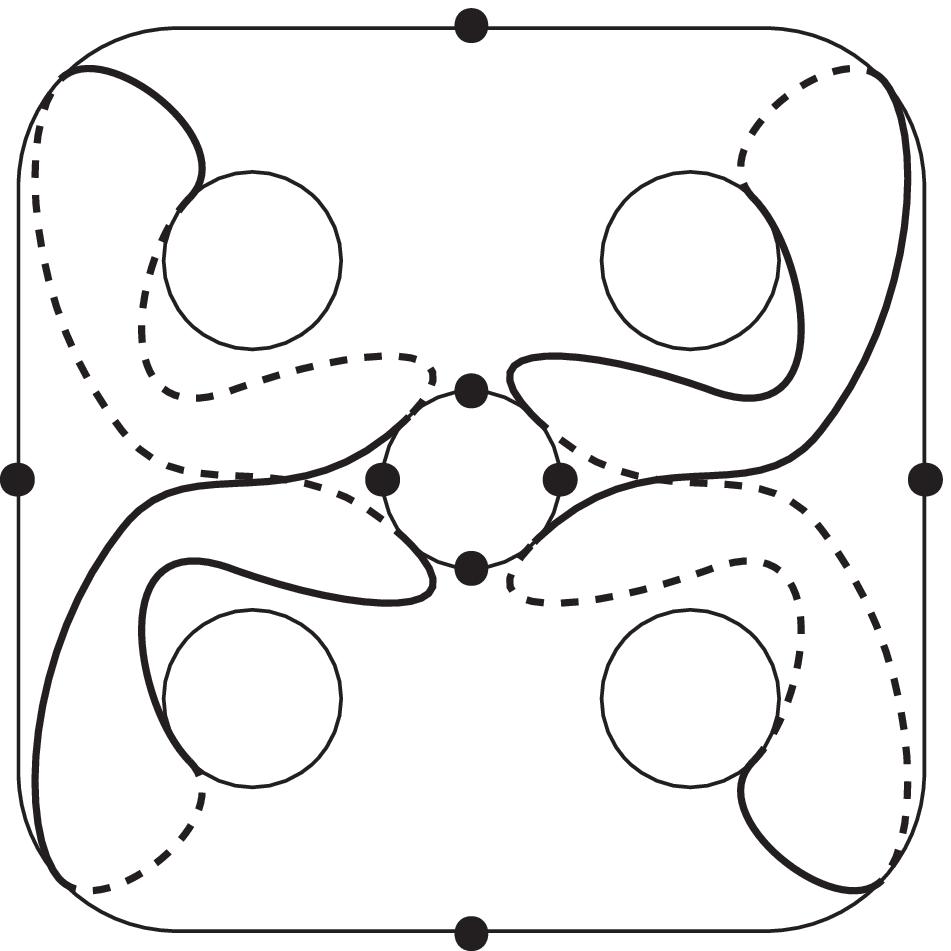}
\label{F:lift_VCSmith1_5}}
\hspace{.8em}
\subfigure[The curves $d_6$ and $\eta_1(d_6)$.]{\includegraphics[height=32mm]{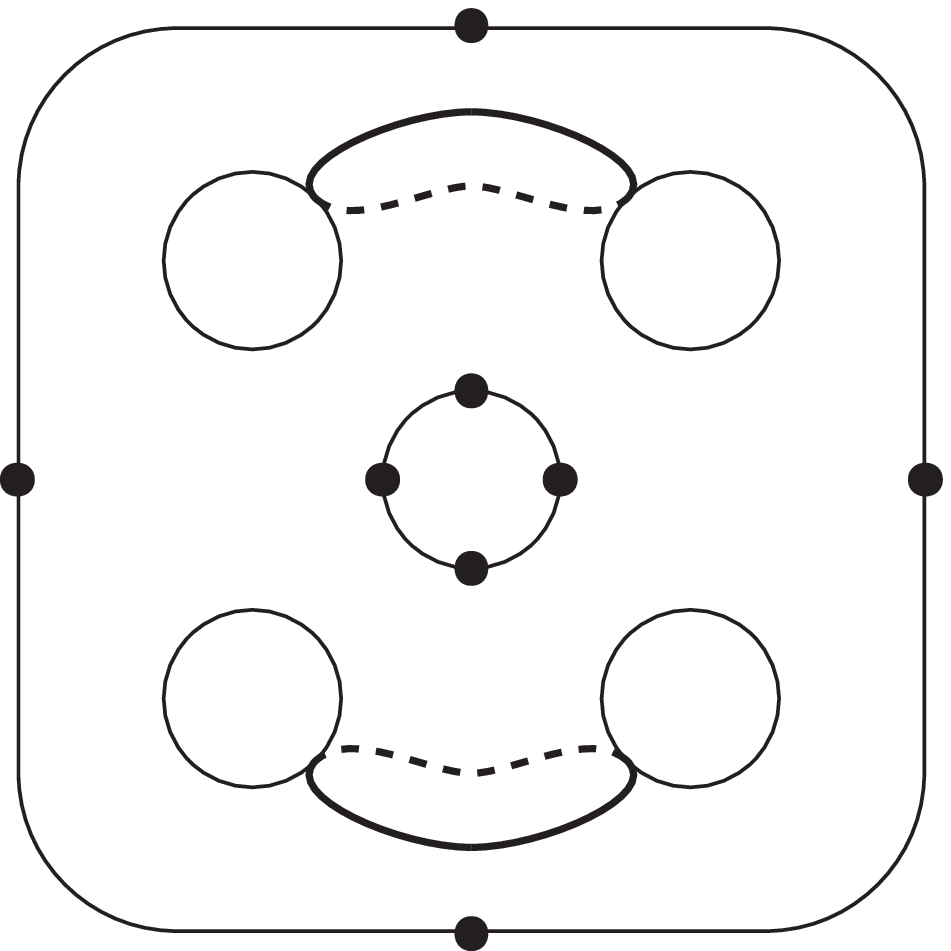}
\label{F:lift_VCSmith1_6}}
\caption{Vanishing cycles of the Lefschetz pencil $f\circ q_1$. }
\label{F:liftVCSmith1}
\end{figure}
The other lift of $\tilde{c}_i$ is $\eta_1(d_i)$, which is also given in Figure~\ref{F:liftVCSmith1}. 
Since $\tilde{c}_{i+6}$ is equal to $\iota(\tilde{c}_i)$, the lifts of the curve $\tilde{c}_{i+6}$ is $\iota_1(d_i)$ and $\eta_1(\iota_1(d_i))$. 
Thus a monodromy factorization of the pencil $f\circ q_1$ is as follows: 
\[
t_{\eta_1(\iota_1(d_6))}t_{\iota_1(d_6)}\cdots t_{\eta_1(\iota_1(d_1))}t_{\iota_1(d_1)}\cdot t_{\eta_1(d_6)}t_{d_6}\cdots t_{\eta_1(d_1)}t_{d_1} = t_{\delta_1}\cdots t_{\delta_8}. 
\]
Applying the algorithm given in Appendix~\ref{A:2ndhomology}, we can calculate the divisibility of $f\circ q_1$ (using Mathematica), which is equal to $1$. 
Thus the type of a polarization associated with $f\circ q_1$ is $(1,4)$. 

Let $q_2:\tilde{T}_2\to T^4$ be a double unbranched covering corresponding to the subgroup $\left<a_1,b_1,a_2,2b_2\right>\subset H_1(T^4;\Z)$. 
We take involutions $\eta_2$ and $\iota_2$ of a genus-$5$ surface as shown in Figure~\ref{F:involution_eta2}. 
\begin{figure}[htbp]
\centering
\subfigure[The punctured dots are in the opposite side of the surface. ]{\includegraphics[width=45mm]{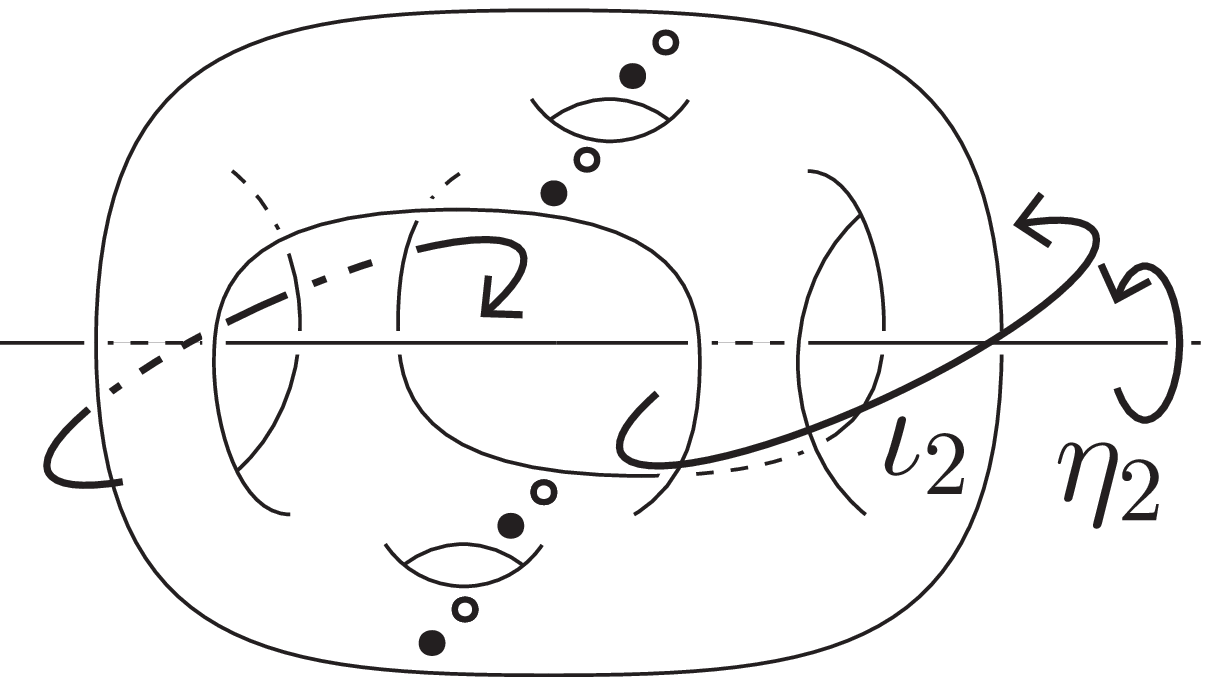}
\label{F:involution_eta2_1}}
\hspace{.8em}\raisebox{4em}{\huge$\cong$}\hspace{.8em}
\subfigure[Another description of the surface. 
The involution $\iota_2$ becomes the $\pi$-degree rotation along the dotted axis.]{\includegraphics[width=40mm]{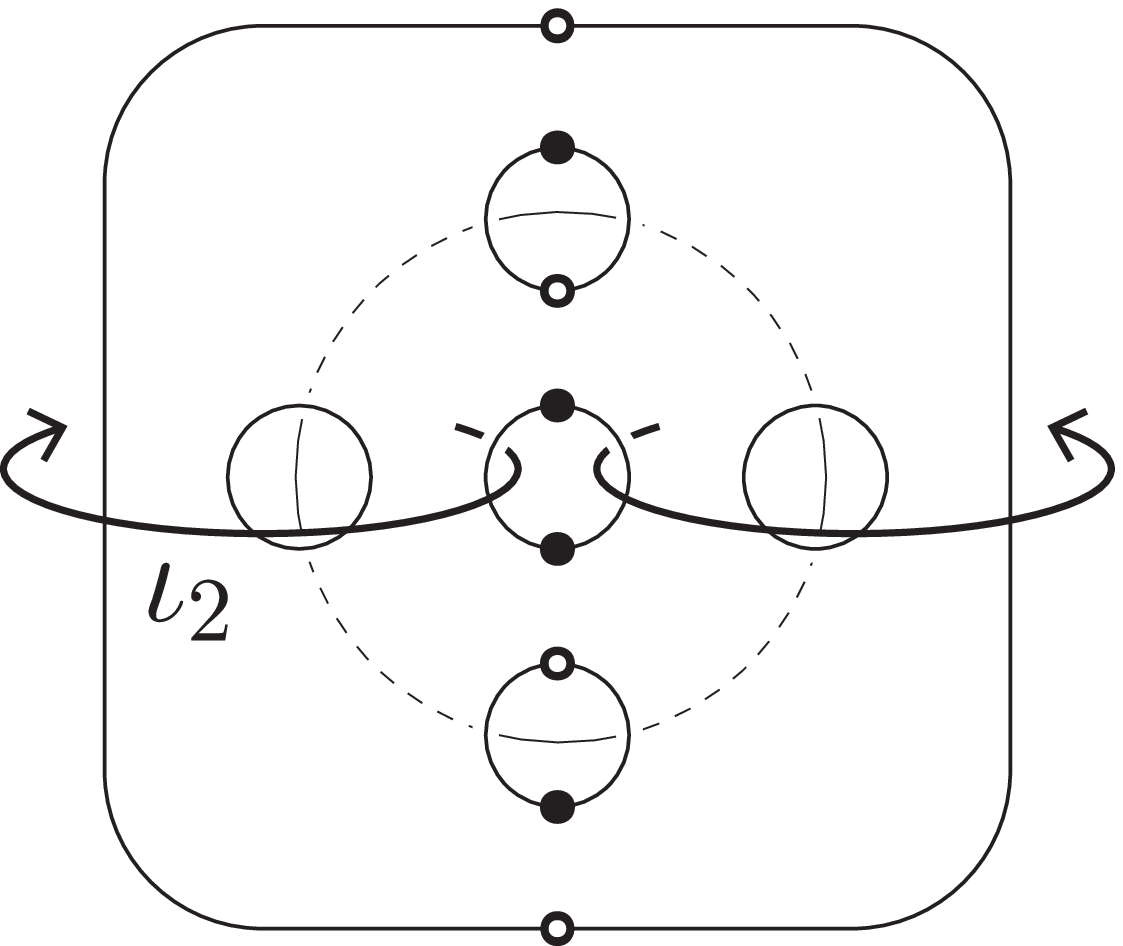}
\label{F:involution_eta2_2}}
\caption{The involutions $\eta_2$ and $\iota_2$.}
\label{F:involution_eta2}
\end{figure}
It is easily verified that the restriction of $q_2$ on the preimage of a reference fiber of $f$ is the quotient map by $\eta_2$, and $\iota_2$ is a lift of $\iota$ under this map. 
By Lemma~\ref{L:relation_monodromy_covering}, we can obtain vanishing cycles of $f\circ q_2$ by taking lifts of $\tilde{c}_i$'s, which are denoted by $e_i$'s and given in Figure~\ref{F:liftVCSmith2}. 
\begin{figure}[htbp]
\centering
\subfigure[The curve $e_1$. ]{\includegraphics[height=32mm]{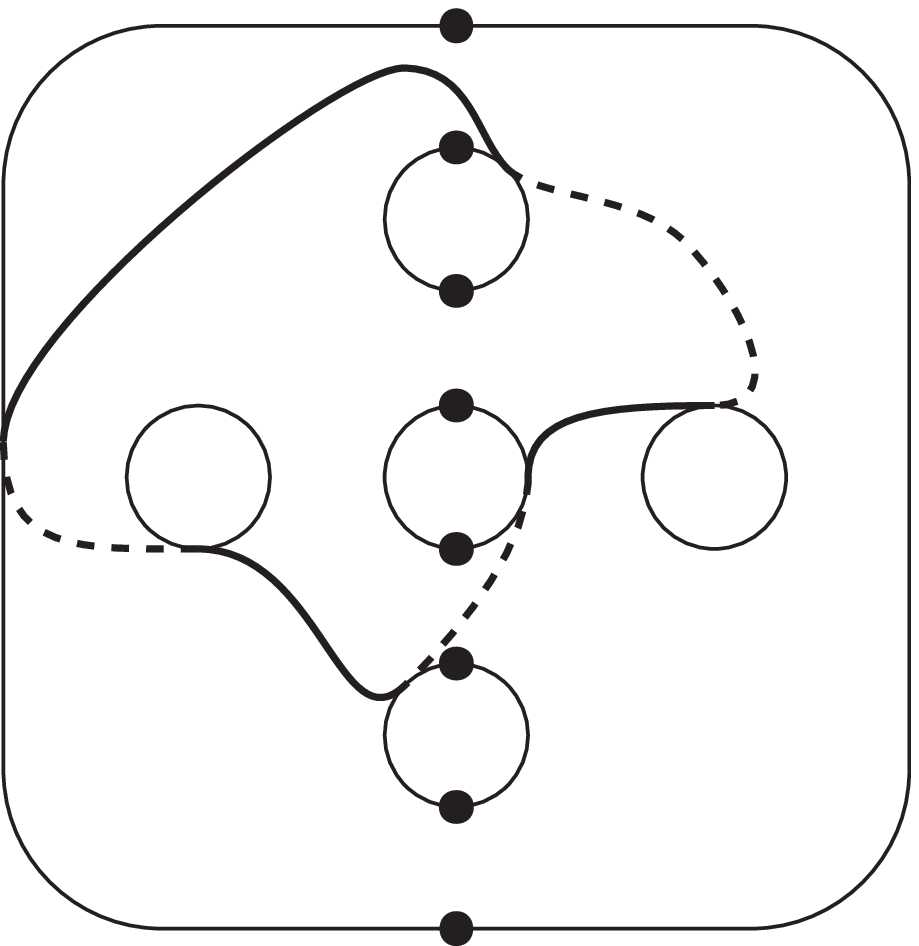}
\label{F:liftVCSmith2_1}}
\hspace{.8em}
\subfigure[The curve $e_2$.]{\includegraphics[height=32mm]{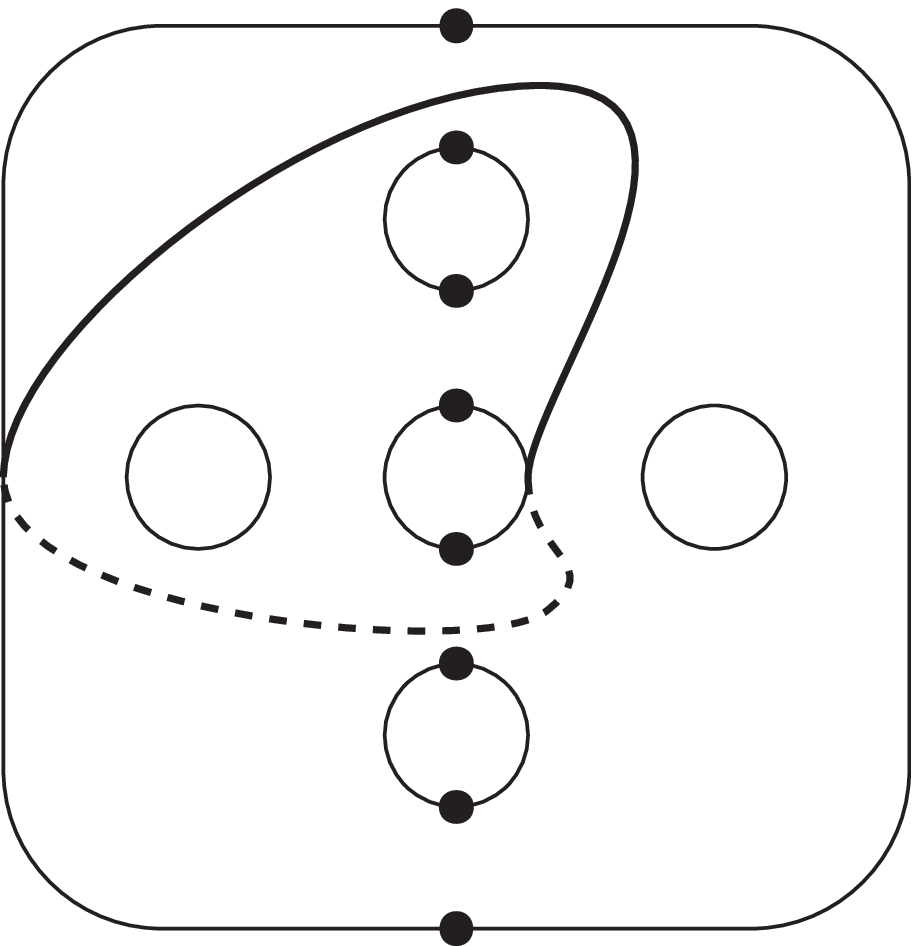}
\label{F:liftVCSmith2_2}}
\hspace{.8em}
\subfigure[The curve $e_3$.]{\includegraphics[height=32mm]{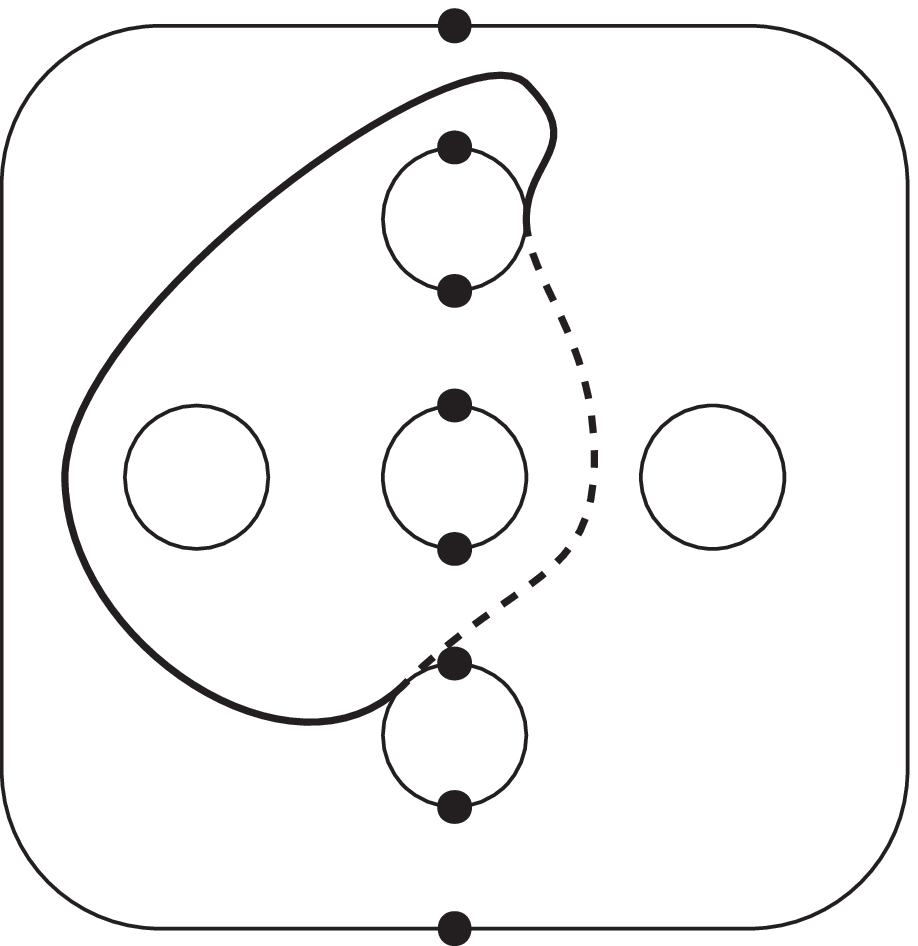}
\label{F:liftVCSmith2_3}}

\subfigure[The curve $e_4$.]{\includegraphics[height=32mm]{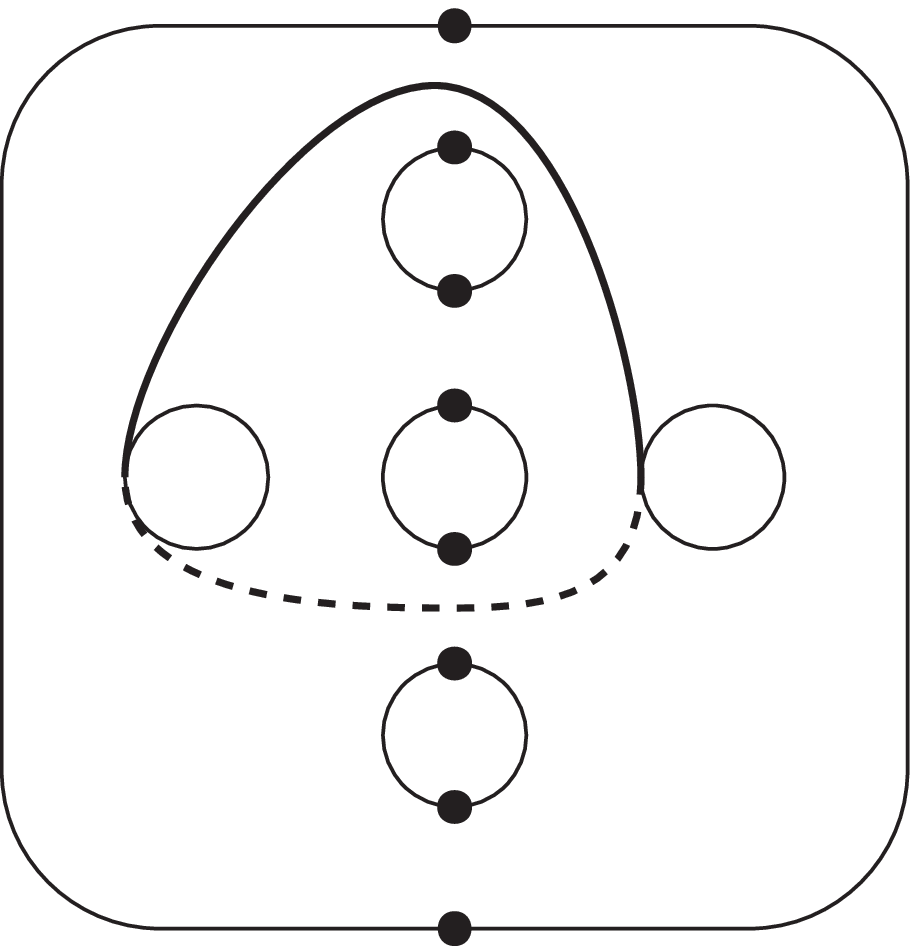}
\label{F:liftVCSmith2_4}}
\hspace{.8em}
\subfigure[The curve $e_5$.]{\includegraphics[height=32mm]{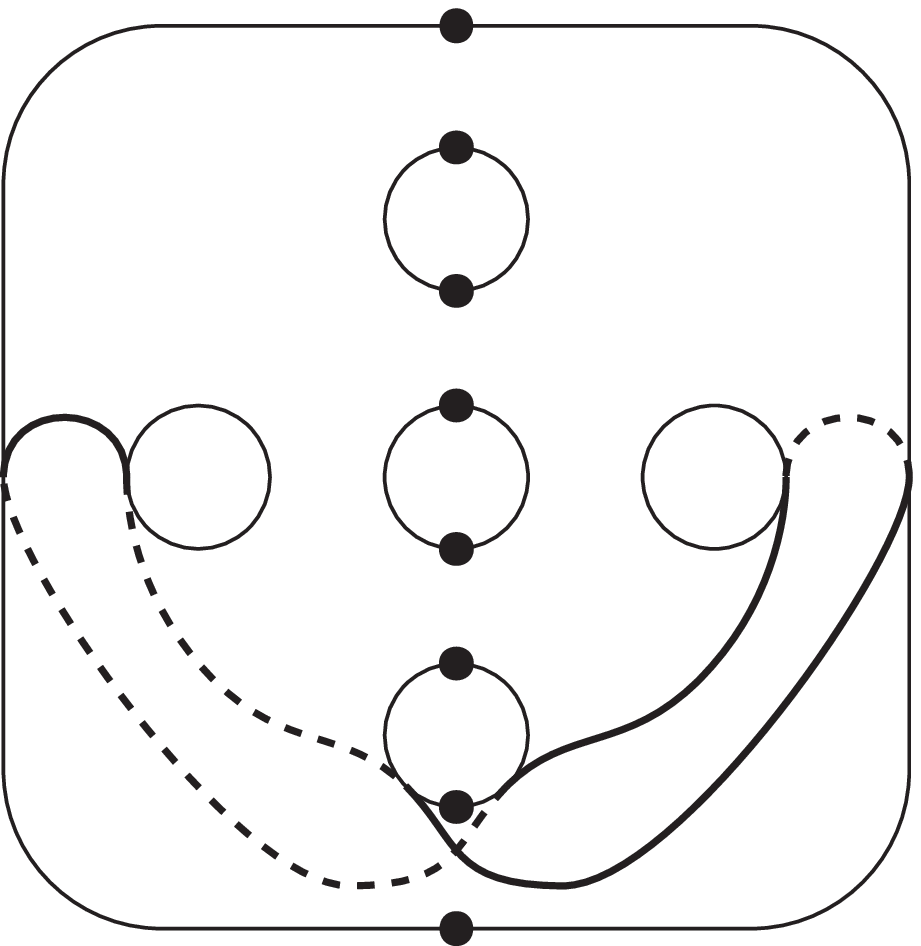}
\label{F:liftVCSmith2_5}}
\hspace{.8em}
\subfigure[The curve $e_6$.]{\includegraphics[height=32mm]{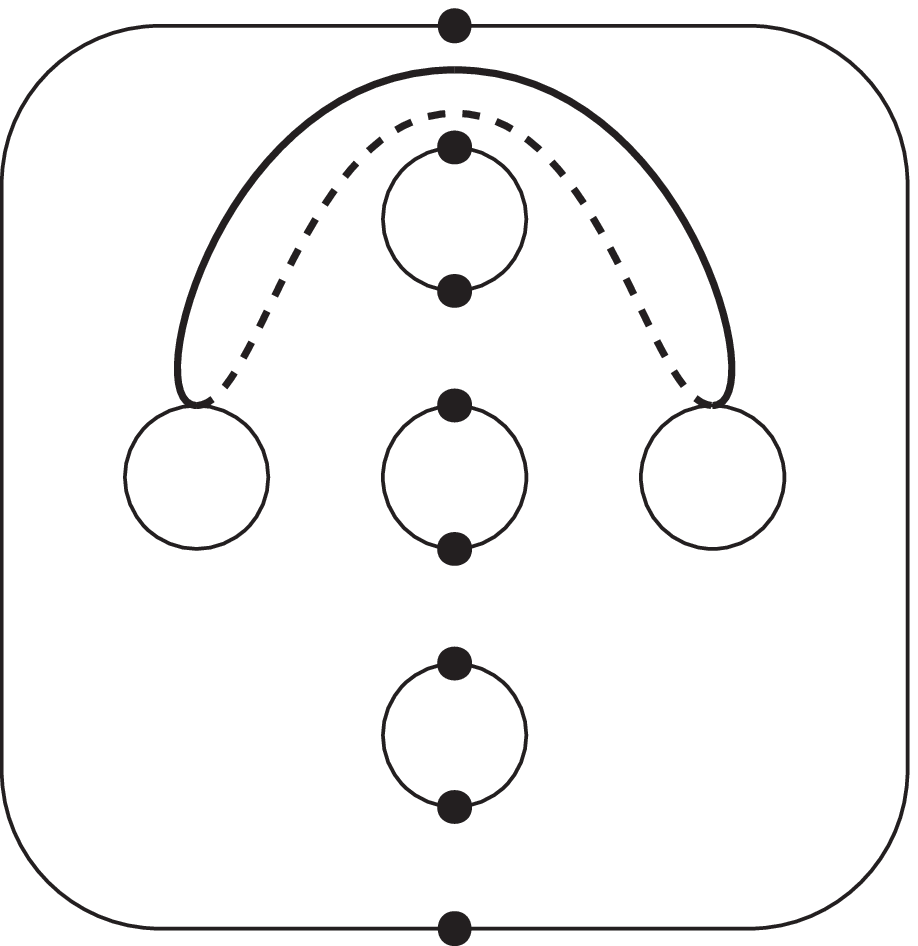}
\label{F:liftVCSmith2_6}}
\hspace{.8em}
\caption{Vanishing cycles of $f\circ q_2$. }
\label{F:liftVCSmith2}
\end{figure}
We eventually obtain a monodromy factorization of $f\circ q_2$ as follows: 
\[
t_{\eta_2(\iota_2(e_6))}t_{\iota_2(e_6)}\cdots t_{\eta_2(\iota_2(e_1))}t_{\iota_2(e_1)}\cdot  t_{\eta_2(e_6)}t_{e_6}\cdots t_{\eta_2(e_1)}t_{e_1}=t_{\delta_1}\cdots t_{\delta_8}. 
\]
Applying the algorithm in Appendix~\ref{A:2ndhomology} (using Mathematica), we can verify that the divisibility of $f\circ q_2$ is equal to $2$. 
Thus, the type of a polarization associated with $f\circ q_2$ is $(2,2)$. 

\end{example}

\begin{remark}

By Corollary~\ref{C:genericness_injectivity} we can obtain two holomorphic Lefschetz pencils by perturbing $f\circ q_1$ and $f\circ q_2$.
These Lefschetz pencils are not isomorphic since they have distinct divisibilities. 
As far as the authors know, this pair is the first example of a pair of non-isomorphic \emph{holomorphic} Lefschetz pencils on the same four-manifold with the same genus, the same number of base points and explicit monodromy factorizations. 

\end{remark}

\section{Combinatorial approach and its applications} \label{Section:CombinatorialApproach}

In this section we will observe a combinatorial aspect of our pencils.
We can reconstuct the factorization (\ref{eq:Smith'sLP}) in a combinatorial way by utilizing a lift to $\Mod(\Sigma_2^4;U)$ of Matsumoto's factorization in $\Mod(\Sigma_2)$ which was given in \cite{Hamada_preprint_sectionMCKLF}.
In \cite{Baykur_preprint_genus3LP}, Baykur independently gave a very similar construction to obtain a genus-$3$ Lefschetz pencil whose total space is homeomorphic to $T^4$. 
In fact, it turns out that his factorization is Hurwitz equivalent to the factorization (\ref{eq:Smith'sLP}) (Remark~\ref{R:HurwitzEquivalenceBaykur}).
Although the combinatorial construction has been already presented in \cite{Baykur_preprint_genus3LP}, we repeat it here in a slightly different way (more symmetrical way) for completeness of the rest of the paper.
Our combinatorial construction of Smith's pencil is pretty useful so that we can obtain two new families of symplectic Calabi-Yau Lefschetz pencils: one is a generalization of Smith's pencil to higher genera, the other consists of pencils on four-manifolds homeomorphic to the total spaces of torus-bundles over the torus admitting a section.

In this section we will freely use elementary transformations, especially commutativity relations, and permutations in the calculations.
For a Dehn twist factorization $W = t_{a_1} \cdots t_{a_k}$ (which is not necessarily eqaul to the identity or the boundary twist) and a mapping class $\phi$ we denote the simultaneous conjugation $\phi W \phi^{-1} = t_{\phi(a_1)} \cdots t_{\phi(a_k)}$ by $_{\phi}(W)$ throughout the section.

\subsection{Smith's pencil and its generalization} \label{Subsection:Combinatorial_Smith's_Pencil}

As we mentioned, in order to combinatorially construct the factorization (\ref{eq:Smith'sLP}) we make use of a lift of Matsumoto's factorization. 
Matsumoto's factorization has been well-known as a factorization of a genus-$2$ Lefschetz fibration on $T^2 \times S^2 \# 4\overline{\mathbb{C}P^2}$~\cite{Matsumoto_1996}.
In \cite{Hamada_preprint_sectionMCKLF}, the first author found several lifts of the factorization to $\Mod(\Sigma_2^4)$, each of which gives four $(-1)$-sections of Matsumoto's Lefschetz fibration. One of them is the following:
\begin{equation}
	t_{{B}_{0,1}}t_{{B}_{1,1}}t_{{B}_{2,1}}t_{{C}_{1}}t_{{B}_{0,2}}t_{{B}_{1,2}}t_{{B}_{2,2}}t_{{C}_{2}} = t_{\delta_1}\cdots t_{\delta_4}, \label{eq:Matsumoto'sLP}
\end{equation}
where the curves are as shown in Figure~\ref{F:MatsumotoSectionT1}.
\begin{figure}[htbp]
	\centering
	\subfigure[The curve $B_{0,1}$. ]{\includegraphics[height=20mm]{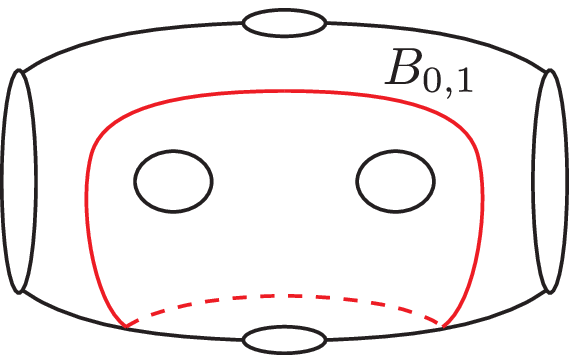}
		\label{F:MatsumotoSectionT1_01}}
	\hspace{.8em}	
	\subfigure[The curve $B_{1,1}$. ]{\includegraphics[height=20mm]{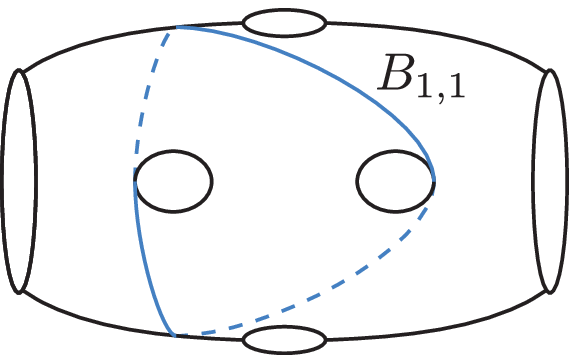}
		\label{F:MatsumotoSectionT1_11}}
	\hspace{.8em}	
	\subfigure[The curve $B_{2,1}$. ]{\includegraphics[height=20mm]{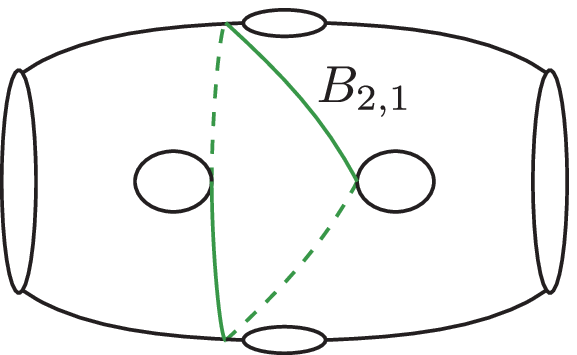}
		\label{F:MatsumotoSectionT1_21}}
	\hspace{.8em}	
	\subfigure[The curve $C_1$ and a $3$-chain configuration. ]{\includegraphics[height=20mm]{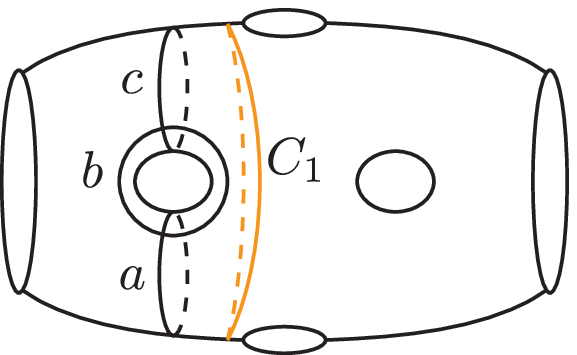}
		\label{F:MatsumotoSectionT1_C1}}
	
	\subfigure[The curve $B_{0,2}$. ]{\includegraphics[height=20mm]{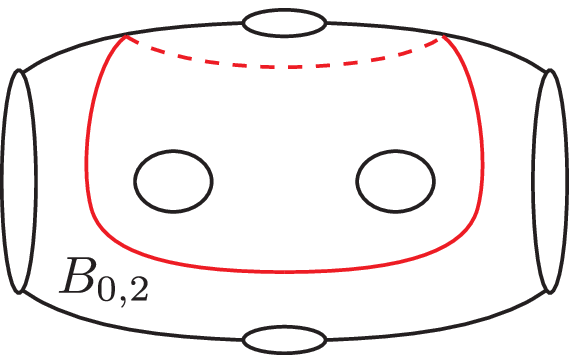}
		\label{F:MatsumotoSectionT1_02}}
	\hspace{.8em}	
	\subfigure[The curve $B_{1,2}$. ]{\includegraphics[height=20mm]{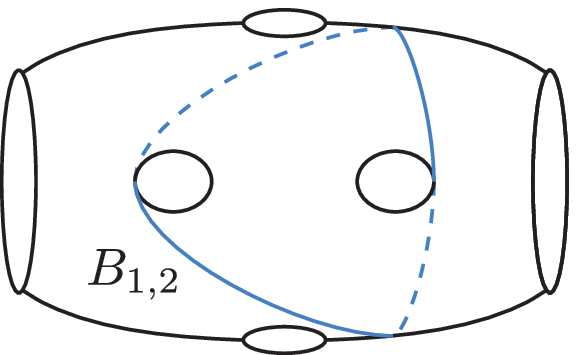}
		\label{F:MatsumotoSectionT1_12}}
	\hspace{.8em}	
	\subfigure[The curve $B_{2,2}$. ]{\includegraphics[height=20mm]{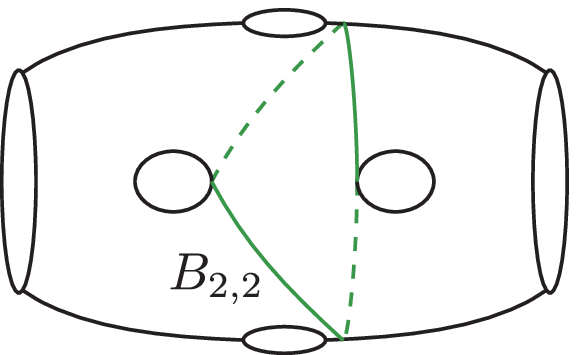}
		\label{F:MatsumotoSectionT1_22}}
	\hspace{.8em}	
	\subfigure[The curve $C_2$ and the boundary curves. ]{\includegraphics[height=20mm]{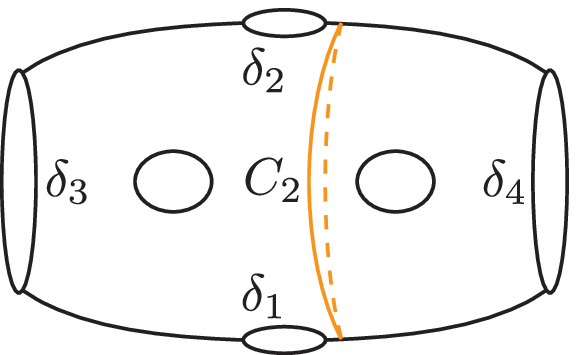}
		\label{F:MatsumotoSectionT1_C2}}	
	\caption{A lift of Matsumoto's factorization. } \label{F:MatsumotoSectionT1}
\end{figure}
We first modify the relation to make it match our scheme.
Consider a $3$-chain relation $(t_ct_at_b)^4 = t_{\delta_3}t_{C_1}$ as in Figure~\ref{F:MatsumotoSectionT1_C1}.
By substituting it to $t_{C_1}$ in \eqref{eq:Matsumoto'sLP} we have
{\allowdisplaybreaks
\begin{align*}
	t_{\delta_1}\cdots t_{\delta_4} &= t_{{B}_{0,1}}t_{{B}_{1,1}}t_{{B}_{2,1}} \cdot (t_ct_at_b)^4 t_{\delta_3}^{-1} \cdot t_{{B}_{0,2}}t_{{B}_{1,2}}t_{{B}_{2,2}}t_{{C}_{2}} \\ 
	&= t_{{B}_{0,1}}t_{{B}_{1,1}}t_{{B}_{2,1}} \cdot (t_ct_at_b)^2 \underline{(t_ct_at_b)^2  \cdot t_{{B}_{0,2}}t_{{B}_{1,2}}t_{{B}_{2,2}}} t_{\delta_3}^{-1}t_{{C}_{2}} \\ 
	&= \underline{t_{{B}_{0,1}}t_{{B}_{1,1}}t_{{B}_{2,1}} \cdot (t_ct_at_b)^2}  \cdot t_{{B}_{0,2}^{\prime}}t_{{B}_{1,2}^{\prime}}t_{{B}_{2,2}^{\prime}} (t_ct_at_b)^2 t_{\delta_3}^{-1}t_{{C}_{2}} \\ 
	&=  (t_ct_at_b)^2 t_{{B}_{0,1}^{\prime}}t_{{B}_{1,1}^{\prime}}t_{{B}_{2,1}^{\prime}} \cdot t_{{B}_{0,2}^{\prime}}t_{{B}_{1,2}^{\prime}}t_{{B}_{2,2}^{\prime}} (t_ct_at_b)^2 t_{\delta_3}^{-1}t_{{C}_{2}} \\ 
	&=   t_{{B}_{0,1}^{\prime}}t_{{B}_{1,1}^{\prime}}t_{{B}_{2,1}^{\prime}} t_{{B}_{0,2}^{\prime}}t_{{B}_{1,2}^{\prime}}t_{{B}_{2,2}^{\prime}} (t_ct_at_b)^2 t_{\delta_3}^{-1}t_{{C}_{2}} (t_ct_at_b)^2 \\ 
	&=   t_{{B}_{0,1}^{\prime}}t_{{B}_{1,1}^{\prime}}t_{{B}_{2,1}^{\prime}}  t_{{B}_{0,2}^{\prime}}t_{{B}_{1,2}^{\prime}}t_{{B}_{2,2}^{\prime}} \cdot (t_ct_at_b)^4 t_{\delta_3}^{-1} \cdot t_{{C}_{2}},
\end{align*}
}
where ${B}_{0,1}^{\prime} = (t_ct_at_b)^{-2}({B}_{0,1})$, ${B}_{1,1}^{\prime} = (t_ct_at_b)^{-2}({B}_{1,1})$, ${B}_{2,1}^{\prime} = (t_ct_at_b)^{-2}({B}_{2,1})$, ${B}_{0,2}^{\prime} = (t_ct_at_b)^{2}({B}_{0,2})$, ${B}_{1,2}^{\prime} = (t_ct_at_b)^{2}({B}_{1,2})$, and ${B}_{2,2}^{\prime} = (t_ct_at_b)^{2}({B}_{2,2})$, which are as depicted in Figure~\ref{F:ModifiedMatsumotoSection}.
(Note that the geometric action of the mapping class $(t_ct_at_b)^{\pm2}t_{\delta_3}^{\mp1}$ to the surface rotates the subsurface between $\delta_3$ and  $C_1$ by $\pm 180$ degrees with respect to the horizontal axis while holding $\delta_3$ and  $C_1$.)  
\begin{figure}[htbp]
	\centering
	\subfigure[The curves $B_{0,1}^{\prime}$, $B_{1,1}^{\prime}$, $B_{2,1}^{\prime}$. ]{\includegraphics[height=30mm]{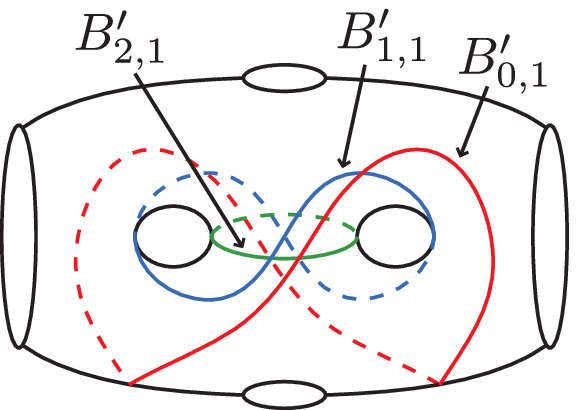}
		\label{F:MatsumotoSectionT1Sym_1}}
	\hspace{.8em}
	\subfigure[The curves $B_{0,2}^{\prime}$, $B_{1,2}^{\prime}$, $B_{2,2}^{\prime}$.]{\includegraphics[height=30mm]{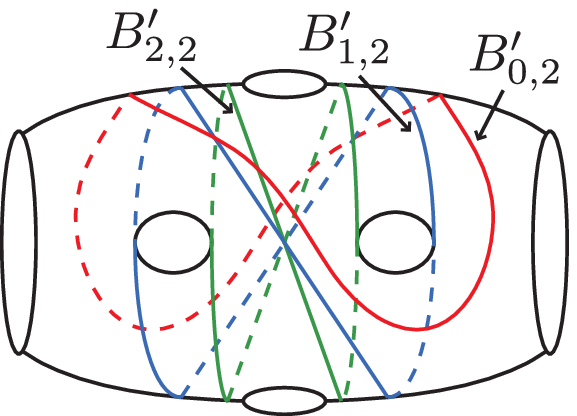}
		\label{F:MatsumotoSectionT1Sym_2}}
	\hspace{.8em}
	\subfigure[The curves $C_1$, $C_2$.]{\includegraphics[height=30mm]{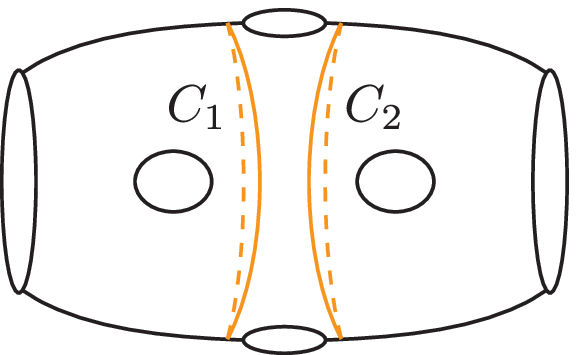}
		\label{F:MatsumotoSectionT1Sym_3}}
	\caption{Modified lift of Matsumoto's factorization. } \label{F:ModifiedMatsumotoSection}
\end{figure}
By resubstituting the $3$-chain relation in the reverse way, we obtain
\begin{equation}
	t_{{B}_{0,1}^{\prime}}t_{{B}_{1,1}^{\prime}}t_{{B}_{2,1}^{\prime}}  t_{{B}_{0,2}^{\prime}}t_{{B}_{1,2}^{\prime}}t_{{B}_{2,2}^{\prime}} t_{{C}_1} t_{{C}_{2}} = t_{\delta_1}\cdots t_{\delta_4}. \label{eq:SymMatsumoto'sLP}
\end{equation}
This expression has a nice symmetry, namely, each $B_{i,j}^{\prime}$ is preserved by the $180$ degree rotation with respect to the vertical axis while $C_1$ and $C_2$ switch.

To construct Smith's pencil, we now consider a $4$-holed genus-$3$ surface and two configurations for the relation~\eqref{eq:SymMatsumoto'sLP} as in Figure~\ref{F:SmithsLPConfiguration} that give
\begin{align*}
	&t_{d_1}t_{d_2}t_{d_3}t_{d_4}t_{d_5}t_{d_6} \cdot t_{s_1} t_{s_2} = t_{\delta_1} t_{\delta_2} \cdot t_{s_3} t_{s_4}, \\
	&t_{d_7}t_{d_8}t_{d_9}t_{d_{10}}t_{d_{11}}t_{d_{12}} \cdot t_{s_3} t_{s_{4}} = t_{\delta_3} t_{\delta_4} \cdot t_{s_1} t_{s_2}. 
\end{align*}
\begin{figure}[htbp]
	\centering
	\subfigure[The curves $d_1$, $d_2$, $d_3$. ]{\includegraphics[height=25mm]{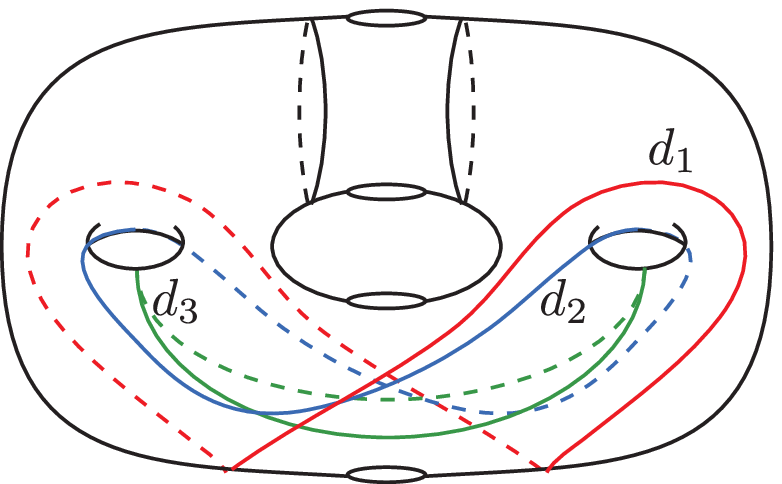}
		\label{F:SmithsLPConfiguration_1_1}}
	\hspace{.8em}
	\subfigure[The curves $d_4$, $d_5$, $d_6$. ]{\includegraphics[height=25mm]{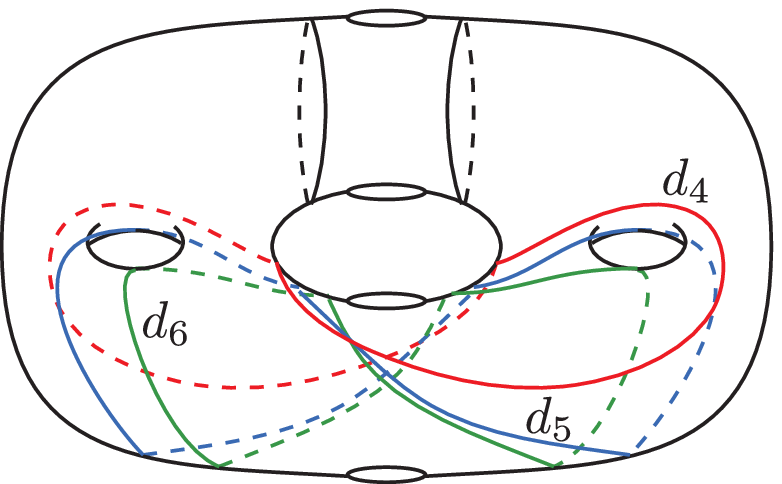}
		\label{F:SmithsLPConfiguration_1_2}}
	\hspace{.8em}
	\subfigure[The curves $s_1$, $s_2$, $s_3$, $s_4$, $\delta_1$, $\delta_2$ and $a$, $b$, $c$. ]{\includegraphics[height=25mm]{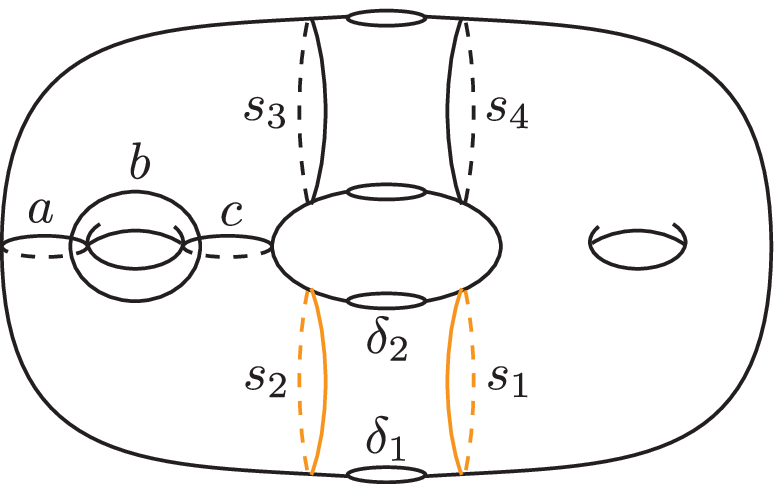}
		\label{F:SmithsLPConfiguration_1_3}}
	
	\subfigure[The curves $d_7$, $d_8$, $d_9$. ]{\includegraphics[height=25mm]{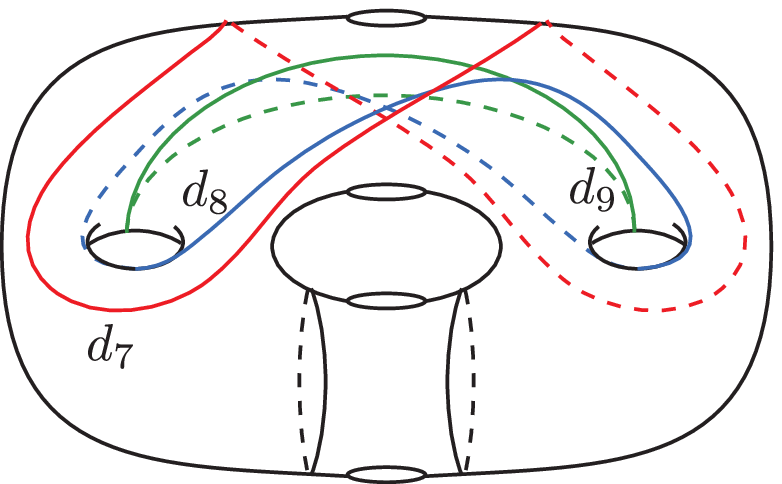}
		\label{F:SmithsLPConfiguration_2_1}}
	\hspace{.8em}
	\subfigure[The curves $d_{10}$, $d_{11}$, $d_{12}$. ]{\includegraphics[height=25mm]{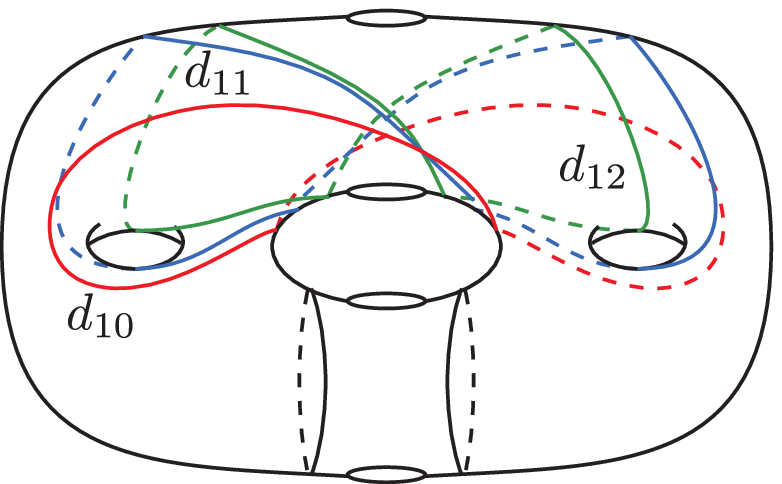}
		\label{F:SmithsLPConfiguration_2_2}}
	\hspace{.8em}
	\subfigure[The curves $s_3$, $s_4$, $s_1$, $s_2$, $\delta_3$, $\delta_4$. ]{\includegraphics[height=25mm]{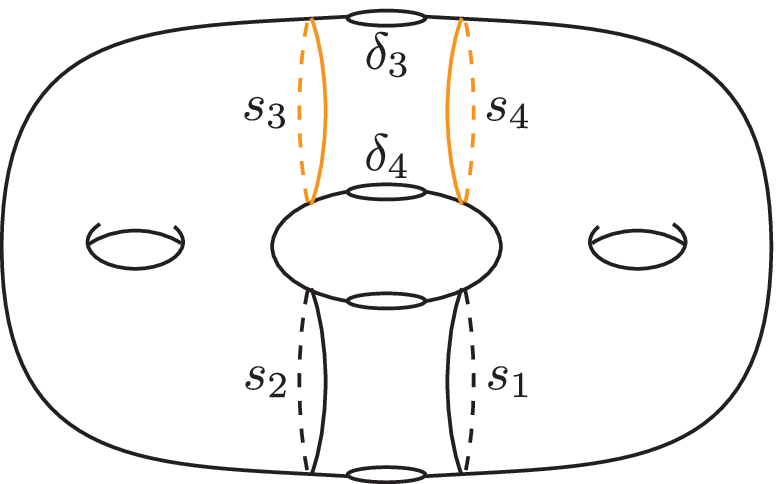}
		\label{F:SmithsLPConfiguration_2_3}}
	\caption{Combinatorial construction of Smith's Lefchetz pencil. } \label{F:SmithsLPConfiguration}
\end{figure}

\noindent
We rewrite them as follows:
{\allowdisplaybreaks
\begin{align*}
	&t_{d_1}t_{d_2}t_{d_3}t_{d_4}t_{d_5}t_{d_6}  = t_{\delta_1} t_{\delta_2} \cdot t_{s_3} t_{s_4} t_{s_1}^{-1} t_{s_2}^{-1}, \\
	&t_{d_7}t_{d_8}t_{d_9}t_{d_{10}}t_{d_{11}}t_{d_{12}}  = t_{\delta_3} t_{\delta_4} \cdot t_{s_1} t_{s_2}  t_{s_3}^{-1} t_{s_{4}}^{-1}. 
\end{align*}
}
Combining them and canceling out $t_{s_3} t_{s_4} t_{s_1}^{-1} t_{s_2}^{-1}$ and $t_{s_1} t_{s_2}  t_{s_3}^{-1} t_{s_{4}}^{-1}$, we then obtain
\begin{equation}
	t_{d_1}t_{d_2}t_{d_3}t_{d_4}t_{d_5}t_{d_6}t_{d_7}t_{d_8}t_{d_9}t_{d_{10}}t_{d_{11}}t_{d_{12}} = t_{\delta_1}t_{\delta_2}t_{\delta_3}t_{\delta_4}. \label{eq:CombinatorialSmith'sLP}
\end{equation}

\begin{lemma}\label{L:HurwitzEquivalence}
	The factorization \eqref{eq:Smith'sLP} is Hurwitz equivalent to the factorization \eqref{eq:CombinatorialSmith'sLP}.
\end{lemma}

\begin{proof}
	Noticing that we already have $\tilde{c}_1=d_{11}$, $\tilde{c}_3=d_{10}$, $\tilde{c}_5=d_{6}$, $\tilde{c}_6=d_{9}$, $\tilde{c}_7=d_{5}$, $\tilde{c}_9=d_{4}$, $\tilde{c}_{11}=d_{12}$ and $\tilde{c}_{12}=d_{3}$, 
	{\allowdisplaybreaks
	\begin{align*}
		t_{\delta_1}\cdots t_{\delta_4} &= t_{\tilde{c}_{12}} t_{\tilde{c}_{11}} t_{\tilde{c}_{10}} t_{\tilde{c}_9} t_{\tilde{c}_8} t_{\tilde{c}_7} t_{\tilde{c}_6} t_{\tilde{c}_5} t_{\tilde{c}_4} t_{\tilde{c}_3} t_{\tilde{c}_2} t_{\tilde{c}_1} \\
		&= \underline{t_{d_3} t_{d_{12}}} t_{\tilde{c}_{10}} t_{d_4} t_{\tilde{c}_8} t_{d_5} \underline{t_{d_9} t_{d_6}} t_{\tilde{c}_4} t_{d_{10}} t_{\tilde{c}_2} t_{d_{11}} \\ 
		&\sim t_{d_3} t_{\tilde{c}_{10}} \underline{t_{d_4} t_{\tilde{c}_8}} t_{d_5} t_{d_6} \cdot t_{d_9} t_{\tilde{c}_4} \underline{t_{d_{10}} t_{\tilde{c}_2}} t_{d_{11}} t_{d_{12}}\\ 
		&\sim t_{d_3} \underline{t_{\tilde{c}_{10}} t_{\tilde{c}_8}} t_{d_4} t_{d_5} t_{d_6} \cdot t_{d_9} \underline{t_{\tilde{c}_4} t_{\tilde{c}_2}} t_{d_{10}} t_{d_{11}} t_{d_{12}}\\ 
		&\sim \underline{t_{d_3}  t_{t_{\tilde{c}_{10}}(\tilde{c}_8)} t_{\tilde{c}_{10}}} t_{d_4} t_{d_5} t_{d_6} \cdot \underline{t_{d_9}  t_{t_{\tilde{c}_4}(\tilde{c}_2)} t_{\tilde{c}_4}} t_{d_{10}} t_{d_{11}} t_{d_{12}}\\ 
		&\sim   t_{t_{d_3}t_{\tilde{c}_{10}}(\tilde{c}_8)} t_{t_{d_3}(\tilde{c}_{10})} t_{d_3} t_{d_4} t_{d_5} t_{d_6} \cdot t_{t_{d_9}t_{\tilde{c}_4}(\tilde{c}_2)} t_{t_{d_9}(\tilde{c}_4)} t_{d_9} t_{d_{10}} t_{d_{11}} t_{d_{12}} \\
		&= t_{d_1} t_{d_2} t_{d_3} t_{d_4} t_{d_5} t_{d_6} t_{d_7} t_{d_8} t_{d_9} t_{d_{10}} t_{d_{11}} t_{d_{12}},
	\end{align*}
	}
	where the last equality follows from an easy observation that $t_{d_3}t_{\tilde{c}_{10}}(\tilde{c}_8)= d_1$, $t_{d_3}(\tilde{c}_{10}) = d_2$, $t_{d_9}t_{\tilde{c}_4}(\tilde{c}_2) = d_7$ and $t_{d_9}(\tilde{c}_4)=d_8$.
\end{proof}

\begin{remark} \label{R:HurwitzEquivalenceBaykur}
	As we repeated, the Lefschetz pencil constructed by Baykur \cite{Baykur_preprint_genus3LP} whose total space is homeomorphic to $T^4$ is isomorphic to the pencil corresponding \eqref{eq:CombinatorialSmith'sLP}, hence \eqref{eq:Smith'sLP}.
	To see this we first take the simultaneous conjugation of \eqref{eq:CombinatorialSmith'sLP} by $(t_ct_at_b)^2$ where the curves $a$, $b$, $c$ are as shown in Figure~\ref{F:SmithsLPConfiguration_1_3}.
	Letting $d_i^{\prime}$ be the resulting curve of $d_i$ mapped by $(t_ct_at_b)^2$, we see
	\begin{align*}
		t_{\delta_1}\cdots t_{\delta_4} &=t_{d_1}t_{d_2}t_{d_3}t_{d_4}t_{d_5}t_{d_6}t_{d_7}t_{d_8}t_{d_9}t_{d_{10}}t_{d_{11}}t_{d_{12}} \\
		&\sim t_{d_1^{\prime}}t_{d_2^{\prime}}t_{d_3^{\prime}}t_{d_4^{\prime}}t_{d_5^{\prime}}t_{d_6^{\prime}} \underline{t_{d_7^{\prime}}t_{d_8^{\prime}}t_{d_9^{\prime}}}t_{d_{10}^{\prime}}t_{d_{11}^{\prime}}t_{d_{12}^{\prime}} \\
		&\sim t_{d_1^{\prime}}t_{d_2^{\prime}}t_{d_3^{\prime}}t_{d_4^{\prime}}t_{d_5^{\prime}}t_{d_6^{\prime}} t_{d_8^{\prime}}t_{d_9^{\prime}} \underline{t_{t_{d_9^{\prime}}^{-1}t_{d_8^{\prime}}^{-1}(d_7^{\prime})}t_{d_{10}^{\prime}}} t_{d_{11}^{\prime}}t_{d_{12}^{\prime}} \\
		&\sim t_{d_1^{\prime}}t_{d_2^{\prime}}t_{d_3^{\prime}}t_{d_4^{\prime}}t_{d_5^{\prime}}t_{d_6^{\prime}} \underline{t_{d_8^{\prime}}t_{d_9^{\prime}} t_{d_{10}^{\prime}}} t_{t_{d_9^{\prime}}^{-1}t_{d_8^{\prime}}^{-1}(d_7^{\prime})} t_{d_{11}^{\prime}}t_{d_{12}^{\prime}} \\
		&\sim t_{d_1^{\prime}}t_{d_2^{\prime}}t_{d_3^{\prime}}t_{d_4^{\prime}}t_{d_5^{\prime}}t_{d_6^{\prime}} t_{t_{d_8^{\prime}}t_{d_9^{\prime}}(d_{10}^{\prime})} t_{d_8^{\prime}}t_{d_9^{\prime}}  t_{t_{d_9^{\prime}}^{-1}t_{d_8^{\prime}}^{-1}(d_7^{\prime})} t_{d_{11}^{\prime}}t_{d_{12}^{\prime}}.
	\end{align*}
	The last factorization is exactly the same as Baykur's factorization.
\end{remark}

The construction of the factorization \eqref{eq:CombinatorialSmith'sLP} can be generalized to higher genera, which provides a new family of symplectic Calabi-Yau Lefschetz pencils.
We consider the surface $\Sigma_g^{2g-2}$ of genus $g \geq 3$ with $2g-2$ boundary components in a circular position and the $(2\pi/(g-1))$-rotation $\phi$ around the center as shown in Figure~\ref{F:SmithsRelation_HigherGenera_Configuration_a}.
Take the configuration for the relation~\eqref{eq:SymMatsumoto'sLP} as in Figure~\ref{F:SmithsRelation_HigherGenera_Configuration_b},~\ref{F:SmithsRelation_HigherGenera_Configuration_c},~\ref{F:SmithsRelation_HigherGenera_Configuration_d}, then, as before, we have
\[
t_{d_{1,1}} t_{d_{2,1}} t_{d_{3,1}} t_{d_{4,1}} t_{d_{5,1}} t_{d_{6,1}} = t_{\delta_1}t_{\delta_2} \cdot t_{s_0} t_{s_3} t_{s_1}^{-1} t_{s_2}^{-1}.
\]
We put $W_1 = t_{d_{1,1}} t_{d_{2,1}} t_{d_{3,1}} t_{d_{4,1}} t_{d_{5,1}} t_{d_{6,1}} $ (as a factorization) and take the simultaneous conjugation $W_i := _{\phi^{i-1}}(W_1)$ of $W_1$ by the rotation map $\phi^{i-1}$ for $i= 1, \cdots, g-1$:  
\[
W_i = t_{d_{1,i}} t_{d_{2,i}} t_{d_{3,i}} t_{d_{4,i}} t_{d_{5,i}} t_{d_{6,i}} = t_{\delta_{2i-1}}t_{\delta_{2i}} \cdot t_{s_{2i-2}} t_{s_{2i+1}} t_{s_{2i-1}}^{-1} t_{s_{2i}}^{-1},
\]
where $d_{j,i}= \phi^{i-1}(d_{j,1})$, $s_{2i-2}=\phi^{i-1}(s_0) $, $s_{2i-1}=\phi^{i-1}(s_1) $, $s_{2i}=\phi^{i-1}(s_2) $, $s_{2i+1}=\phi^{i-1}(s_3) $, $\delta_{2i-1}= \phi^{i-1}(\delta_1)$ and $\delta_{2i}= \phi^{i-1}(\delta_2)$.
Note that $s_{2g-2}=s_0$ and $s_{2g-1}=s_1$.
When we combine $W_1, \cdots, W_{g-1}$, a similar cancelling process as before works well so that we obtain
\begin{equation}
	W_1 \cdots W_{g-1} = \prod_{i=1}^{g-1} t_{d_{1,i}} t_{d_{2,i}} t_{d_{3,i}} t_{d_{4,i}} t_{d_{5,i}} t_{d_{6,i}} = t_{\delta_1} \cdots t_{\delta_{2g-2}}. \label{eq:GeneralizedSmith'sLP}
\end{equation}
\begin{figure}[htbp]
	\centering
	\subfigure[The surface $\Sigma_g^{2g-2}$ and the rotation $\phi$. ]{\includegraphics[height=55mm]{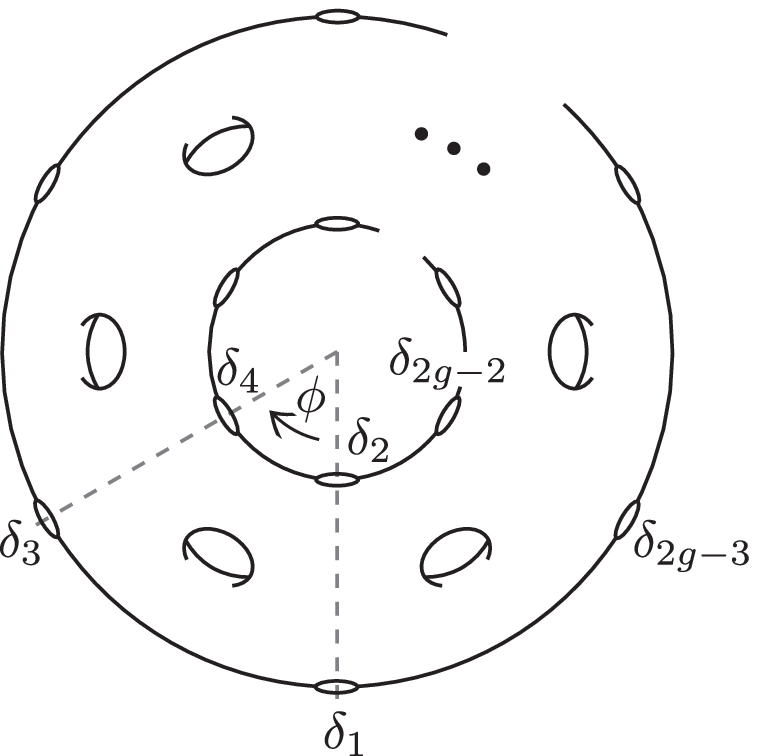}
		\label{F:SmithsRelation_HigherGenera_Configuration_a}}
	\hspace{.8em}
	\subfigure[The curves $s_1 = s_{2g-1}$, $s_2$, $s_3$, $s_0 = s_{2g-2}$, $\delta_1$ and $\delta_2$. ]{\includegraphics[height=55mm]{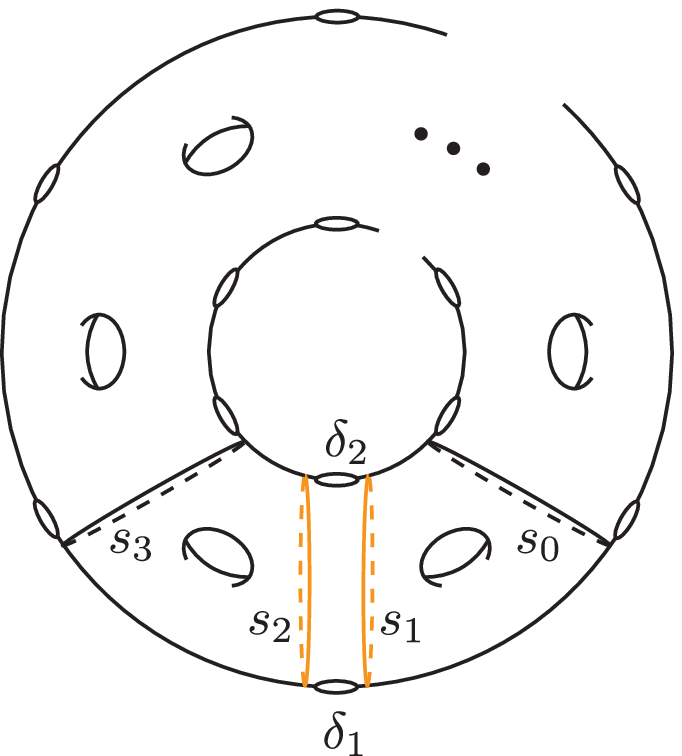}
		\label{F:SmithsRelation_HigherGenera_Configuration_b}}
	
	\hspace{.2em}
	\subfigure[The curves $d_{1,1}$, $d_{2,1}$, $d_{3,1}$. ]{\includegraphics[height=50mm]{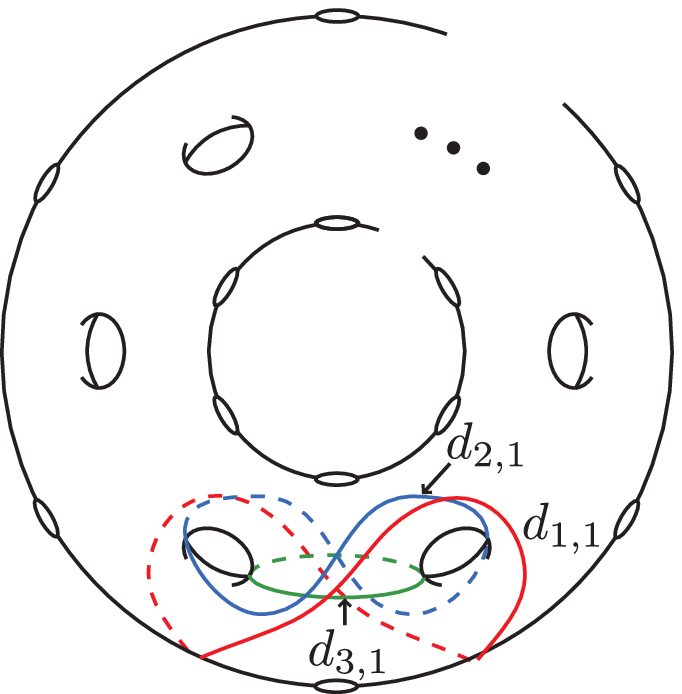}
		\label{F:SmithsRelation_HigherGenera_Configuration_c}}
	\hspace{2.2em}
	\subfigure[The curves $d_{4,1}$, $d_{5,1}$, $d_{6,1}$. ]{\includegraphics[height=50mm]{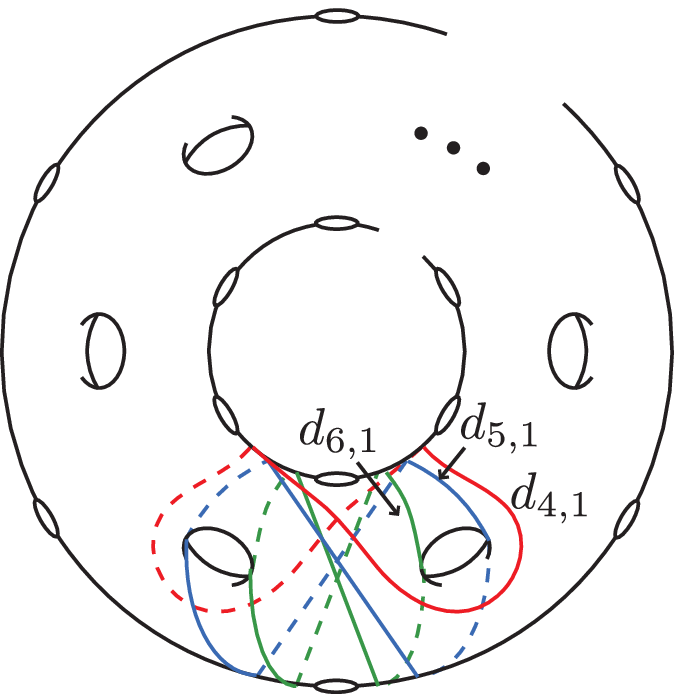}
		\label{F:SmithsRelation_HigherGenera_Configuration_d}}
	
	\caption{Construction of a generalization of Smith's pencil. } \label{F:SmithsRelation_HigherGenera_Configuration}
\end{figure}
This factorization gives a genus-$g$ Lefschetz pencil with $2g-2$ base points and $6g-6$ irreducible critical points for $g \geq 3$.
We denote this Lefschetz pencil by $f_g : X_g \setminus B_g \to \CP^1$ ($g\geq 3$). 
Note that $f_3$ is nothing but the Lefschetz pencil corresponding to the factorization~\eqref{eq:CombinatorialSmith'sLP}, i.e., Smith's pencil.
It is straightforward to see that the Euler characteristic of $X_g$ is $0$.
The signature of the relation~\eqref{eq:Matsumoto'sLP} is $0$~\cite{Hamada_preprint_sectionMCKLF} and 
the signature of a braid relation is also $0$ in the sense of Endo-Nagami~\cite{EN_2005}. 
Since the relation~\eqref{eq:GeneralizedSmith'sLP} is constructed by a combination of copies of the relation~\eqref{eq:Matsumoto'sLP} and braid relations, 
the signature of~\eqref{eq:GeneralizedSmith'sLP} is $0$, hence the signature of $X_g$ is $0$.
It is also easy to verify that the fundamental group of the total space $X_g$ of $f_g$ is $\Z^4$.
By a theorem by Baykur-Hayano~\cite[Theorem 4.1]{BaykurHayano_multisection} the Lefschetz pencil $f_g$ is symplectic Calabi-Yau.
Since a symplectic Calabi-Yau manifold whose fundamental group is isomorphic to $\pi_1(T^4)=\Z^4$ is indeed homeomorphic to $T^4$~\cite[Corollary 3.3]{FriedlVidussi}, so is the total space $X_g$. 
In fact, for odd $g$ we can even show that $X_g$ is diffeomorphic to the $T^4$ and $f_g$ is holomorphic.

\begin{lemma}\label{L:generalization_coveringSmith}

 For odd $g$, the Lefschetz pencil $f_g$ can be obtained by perturbing $f_3\circ q$, where $q:T^4\to T^4$ is a $((g-1)/2)$-fold unbranched covering. 

\end{lemma}

In order to prove Lemma~\ref{L:generalization_coveringSmith} we first observe that, in general, the order among $W_1, \cdots W_{g-1}$ in \eqref{eq:GeneralizedSmith'sLP} does not matter.

\begin{lemma}\label{L:OrderOfW_i}
	For any permutation $\sigma$ of $\{1, \cdots, g-1 \}$, the factorization $W_{\sigma(1)} W_{\sigma(2)} \cdots W_{\sigma(g-1)} = t_{\delta_1} \cdots t_{\delta_{2g-2}}$ is Hurwitz equivalent to the factorization $W_1 W_2 \cdots W_{g-1} = t_{\delta_1} \cdots t_{\delta_{2g-2}}$.
\end{lemma}

\begin{proof}
	We will show that $W_i W_j$ in the factorization can switch to $W_j W_i$ for $i \neq j$.
	We only need to consider the cases $|i-j|=1$ and $\{i,j\}= \{1,g-1\}$, otherwise $W_i W_j$ obviously switches since the supporting subsurfaces for $W_i$ and $W_j$ are disjoint.
	Recalling that $W_i = t_{\delta_{2i-1}}t_{\delta_{2i}} t_{s_{2i-2}} t_{s_{2i+1}} t_{s_{2i-1}}^{-1} t_{s_{2i}}^{-1}$ as a mapping class, and noticing that the only curve among them that intersect with the curves of $W_{i+1} = t_{d_{1,i+1}} t_{d_{2,i+1}} t_{d_{3,i+1}} t_{d_{4,i+1}} t_{d_{5,i+1}} t_{d_{6,i+1}}$ is $s_{2i+1}$, in addition, $s_{2i+1}$ is away from any other curve in $W_j$ for $k \neq i+1$,
	\begin{align*}
		t_{\delta_1} \cdots t_{\delta_{2g-2}}
		&= W_{\sigma(1)} \cdots W_i W_{i+1} \cdots  W_{\sigma(g-1)}\\
		&\sim W_{\sigma(1)} \cdots {}_{W_i}(W_{i+1}) W_{i}  \cdots W_{\sigma(g-1)}\\
		&= W_{\sigma(1)} \cdots {}_{t_{s_{2i+1}}}(W_{i+1}) W_{i}  \cdots W_{\sigma(g-1)}\\
		&\sim W_{\sigma(1)} \cdots W_{i+1} W_i \cdots  W_{\sigma(g-1)},
	\end{align*}
	where the last equivalence is achieved by taking a simultaneous conjugation by $t_{s_{2i+1}}^{-1}$.
	The same argument works when $W_i$ is replaced by $W_1$ and $W_{i+1}$ by $W_g$.
\end{proof}

\begin{proof}[Proof of Lemma~\ref{L:generalization_coveringSmith}]
By Lemma~\ref{L:OrderOfW_i} the factorization~\eqref{eq:GeneralizedSmith'sLP} is Hurwitz equivalent to the following factorization for odd $g$: 
\[
\prod_{i=1}^{(g-1)/2} t_{d_{1,2i-1}} t_{d_{2,2i-1}} \cdots t_{d_{6,2i-1}} \cdot \prod_{i=1}^{(g-1)/2} t_{d_{1,2i}} t_{d_{2,2i}} \cdots t_{d_{6,2i}} = t_{\delta_1} \cdots t_{\delta_{2g-2}}.
\]
Then, by only using commutativity relations, it can be reformed as
\begin{multline*}
	\prod_{i=1}^{(g-1)/2} t_{d_{1,2i-1}} \prod_{i=1}^{(g-1)/2}t_{d_{2,2i-1}} \cdots \prod_{i=1}^{(g-1)/2}t_{d_{6,2i-1}}  \cdot \prod_{i=1}^{(g-1)/2} t_{d_{1,2i}} \prod_{i=1}^{(g-1)/2}t_{d_{2,2i}} \cdots \prod_{i=1}^{(g-1)/2}t_{d_{6,2i}} \\
	= \prod_{i=1}^{(g-1)/2} t_{\delta_{4i-3}} \prod_{i=1}^{(g-1)/2} t_{\delta_{4i-2}} \prod_{i=1}^{(g-1)/2} t_{\delta_{4i-1}} \prod_{i=1}^{(g-1)/2} t_{\delta_{4i}}.
\end{multline*}
Each of the subfactorizations $\prod t_{d_{j,2i-1}}$, $\prod t_{d_{j,2i}}$, $\prod t_{\delta_{4i-3}}$, $\prod t_{\delta_{4i-2}}$, $\prod t_{\delta_{4i-1}}$ and $\prod t_{\delta_{4i}}$ is preserved by the $(4\pi/(g-1))$-rotation  $\phi^2$, which freely acts on the surface $\Sigma_g^{2g-2}$.
Now we can take the quotient by $\phi^2$, which gives the surface $\Sigma_3^4$ and the factorization~\eqref{eq:CombinatorialSmith'sLP}.
In this way, we can think of the factorization~\eqref{eq:GeneralizedSmith'sLP} for odd $g$ as a $((g-1)/2)$-fold unbranched covering of the factorization~\eqref{eq:CombinatorialSmith'sLP}, i.e., an unbranched covering of Smith's pencil $f_3$.
\end{proof}

\begin{remark} \label{R:generalization_covering_generalization}
Lemma~\ref{L:generalization_coveringSmith} can be easily generalized to the claim that for $g_1$, $g_2$ with $g_1-1 | g_2-1$ the pencil $f_{g_2}$ is obtained as a $((g_2-1)/(g_1-1))$-fold unbranched covering of $f_{g_1}$.
On the other hand, for $g$ such that $g-1$ is prime the pencil $f_g$ cannot be obtained as a finite unbranched covering of any Lefschetz pencil of lower genus since the surface $\Sigma_g$ of such a genus $g$ cannot be the total space of an unbranched covering of any surface of lower genus other than $2$, which is easily excluded in any case.

\end{remark}

\begin{lemma}\label{L:divisibility_generalization}

The divisibility of $f_g$ is $1$. 

\end{lemma}

\begin{proof}
Let $a,b,a_i,b_i$ ($i=1,\ldots,g-1$) and $\delta_j$ ($j=1,\ldots,2g-2$) be oriented simple closed curves in $\Sigma_g^{2g-2}$ as shown in Figure~\ref{F:generator_homology_generalization}. 
\begin{figure}[htbp]
\centering
\includegraphics[width=60mm]{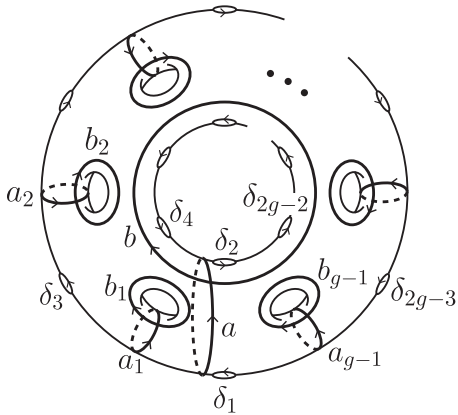}
\caption{The oriented curves in $\Sigma_g^{2g-2}$. 
The curves $\delta_1,\ldots,\delta_{2g-2}$ are on boundary components. }
\label{F:generator_homology_generalization}
\end{figure}
We take points $q_1,\ldots,q_{2g-2}\in \Pa \Sigma_g^{2g-2}$ so that the natural map $\pi_0(\{q_1,\ldots,q_{2g-2}\})\to \pi_0(\Pa\Sigma_g^{2g-2})$ is bijective. 
Let $D\subset \Int(\Sigma_g^{2g-2})$ be a disk sufficiently close to $\delta_1$, $\delta$ the simple closed curve $\Pa D$ with a suitable orientation, $\Sigma_g^{2g-1}=\Sigma_g^{2g-2}\setminus \Int(D)$, $q\in \Pa D$ and $Q=\{q,q_1,\ldots,q_{2g-2}\}$. 
We denote the homology classes in $H_1(\Sigma_g^{2g-1},Q;\Z)$ represented by $a,b,a_i,b_i,\delta,\delta_j$ by the same symbols $a,b,a_i,b_i,\delta,\delta_j$, respectively. 
It is easy to verify that the following equalities hold in $H_1(\Sigma_g^{2g-1},Q;\Z)$. 
{\allowdisplaybreaks
\begin{align*}
d_{1i} = & b_i-b_{i-1}+\delta_{2i-1}+\delta_{1,i}\delta, \\
d_{2i} = & a_i+b_i-a_{i-1}-b_{i-1}+\delta_{2i-1}+\delta_{1,i}\delta, \\
d_{3i} = & a_i-a_{i-1}+\delta_{2i-1}+\delta_{1,i}\delta, \\
d_{4i} = & b_i-b_{i-1}+\delta_{2i}+\delta_{1,i}\delta, \\
d_{5i} = & a_i+b_i-a_{i-1}-b_{i-1}+\delta_{2i}+\delta_{1,i}\delta, \\
d_{6i} = & a_i-a_{i-1}+\delta_{2i}+\delta_{1,i}\delta, 
\end{align*}
}
where $a_g=a_1$, $b_g=b_1$ and $\delta_{1,i}\in \{1,0\}$ denotes the Kronecker delta. 
We can take a handle decomposition of the blow-up $\tilde{X}_g$ of the total space $X_g$ of $f_g$ by applying the procedure explained in Appendix~\ref{A:2ndhomology}. 
Let $\widetilde{d_{ij}}\in C_2$ be the chain corresponding the vanishing cycle $d_{ij}$, $f\in C_2$ the chain corresponding a regular fiber and $\sigma_i\in C_2$ the chain represented by the $2$-handle in a neighborhood of the section corresponding the boundary component $\delta_i$. 
It is easy to see that the cycle group $Z_2$ is generated by the following elements: 
{\allowdisplaybreaks
\begin{align*}
&Z_1^i=\widetilde{d_{1i}}-\widetilde{d_{4i}},\hspace{.4em} Z_2^i=\widetilde{d_{2i}}-\widetilde{d_{5i}}, \hspace{.4em}Z_3^i=\widetilde{d_{3i}}-\widetilde{d_{6i}}\hspace{1em} (i=1,\ldots,g-1), \\
&W_i = d_{2i}-d_{1i}-d_{3i}\hspace{.5em}(i=1,\ldots,g-1), \hspace{.5em} X=\sum_{i=1}^{g-1}d_{1i}, \hspace{.5em} Y=\sum_{i=1}^{g-1}d_{3i}, \hspace{.5em} f,\hspace{.5em} \sigma_1,\ldots,\sigma_{2g-2}. 
\end{align*}
}
As we explained in Appendix~\ref{A:2ndhomology}, each $3$-handle in the handle decomposition of $\tilde{X}_g$ corresponds with a $1$-handle of $\Sigma_g$. 
We take a handle decomposition of $\Sigma_g$ so that each boundary component of $\Sigma_g^{2g-2}$ corresponds with a $0$-handle, $a_i, b_i\in H_1(\Sigma_g^{2g-1},Q;\Z)$ are represented by $1$-handles, and a regular neighborhood of a path $\gamma_i\subset \Sigma_g^{2g-1}\setminus (\cup_i (a_i\cup b_i))$ connecting $\delta_1$ with $\delta_{i+1}$ is a $1$-handle. 
Let $A_i,B_i\in C_3$ be the chain represented by the $3$-handle corresponding $a_i,b_i$, respectively, and $\widetilde{\gamma_i}\in C_3$ the chain represented by the $3$-handle corresponding $\gamma_i$. 
Using Lemma~\ref{L:boundary operator Pa3}, we can calculate the images of the $3$-chains under the boundary operator $\Pa_3$ as follows: 
{\allowdisplaybreaks
\begin{align*}
\Pa_3(A_i) = & Z_1^i-Z_2^i-Z_1^{i+1}+Z_2^{i+1} \hspace{.5em}(i=1\ldots,g-1), \\
\Pa_3(B_i) = & -Z_2^i+Z_3^i+Z_2^{i+1}-Z_3^{i+1} \hspace{.5em}(i=1\ldots,g-1), \\
\Pa_3(A_g)= & \Pa_3(B_g) =0,\\
\Pa_3(\widetilde{\gamma_{2j}}) = & W_1-W_{j+1}-f+\sigma_1-\sigma_{j+1} \hspace{.5em}(j=1\ldots,g-2), \\
\Pa_3(\widetilde{\gamma_{2j-1}}) = & W_1-W_{j}+Z_1^j-Z_2^j+Z_3^j-f+\sigma_1-\sigma_{j+1} \hspace{.5em}(i=1\ldots,g-1), 
\end{align*}
}
where $Z_k^g = Z_k^1$ for $k=1,2,3$. 
Thus, the following set is a basis of the cycle group $Z_2$: 
\[
\left\{\Pa_3(A_1),\Pa_3(B_1),\ldots,\Pa_3(A_{g-2}),\Pa_3(B_{g-2}),\Pa_3(\widetilde{\gamma_1}),\ldots,\Pa_3(\widetilde{\gamma_{2g-3}}),Z_1^1,Z_2^1, X,Y,f,W_1,\sigma_1,\ldots,\sigma_{2g-2}\right\}. 
\]
In particular, the homology group $H_2(X_g;\Z)$ is isomorphic to $\Z^6$ and $\{Z_1^1,Z_2^1,X,Y,f,W_1\}$ is a basis of $H_2(X_g;\Z)$. 
Since $f$ is represented by a regular fiber of $f_g$, the divisibility of $f_g$ is $1$. 
\end{proof}

Combining Lemmas~\ref{L:generalization_coveringSmith} and \ref{L:divisibility_generalization}, we eventually obtain: 

\begin{theorem}\label{T:polarization_generalization}

For odd $g$, $f_g$ is a holomorphic Lefschetz pencil on $T^4$ associated with a $(1,g-1)$-polarization. 

\end{theorem}

\noindent
According to Lemma~\ref{L:divisibility_generalization} and the observation followed by Lemma~\ref{L:generalization_coveringSmith}, it is natural to expect that the following conjecture holds:

\begin{conjecture}\label{C:polarization_generalization}

For even $g$, $f_g$ is a holomorphic Lefschetz pencil on $T^4$ associated with a $(1,g-1)$-polarization. 

\end{conjecture}

\noindent
Note that in order to prove Conjectures~\ref{C:polarization_generalization} it is sufficient to prove that $f_g$ is holomorphic for $g$ such that $g-1$ is prime by Remark~\ref{R:generalization_covering_generalization}.
If Conjectures~\ref{C:polarization_generalization} and \ref{C:uniqueness_holLP} hold, we can deduce the following from Lemma~\ref{L:relation_covering_polarization}:

\begin{conjecture} \label{C:covering_generalisedSmithsLP}

Let $f:T^4\dashedrightarrow \CP^1$ be a holomorphic Lefschetz pencil. 
There exists an unbranched covering $q:T^4\to T^4$ such that $f$ is isomorphic to the composition $f_g\circ q$ with $g-1$ prime. 

\end{conjecture}

\noindent
Note that this conjecture holds under the following assumptions:

\begin{itemize}

\item
the genus of $f$ is odd, and

\item
the genus of $f$ is greater than $5$ or that the divisibility of $f$ is greater than $1$.

\end{itemize}

\noindent
In this case we can take $q$ so that $g$ is equal to $3$ (see the observation following Lemma~\ref{L:relation_covering_polarization}).

\begin{remark}\label{R:relation_generalizationSmith_nonholLP}

If Conjecture~\ref{C:covering_generalisedSmithsLP} holds, it is theoretically possible to obtain monodromy factorizations of all the holomorphic Lefschetz pencils on the four-torus. 
In particular, a Lefschetz pencil on the four-torus is not holomorphic if the associated monodromy factorization is not Hurwitz equivalent to any of them (see also Remark~\ref{R:nonhol_reduciblefiber}).  

\end{remark}

\subsection{Symplectic Calabi-Yau Lefschetz pencils on homotopy $T^2$-bundles over $T^2$}\label{Subsection:T2bdloverT2}
We have seen explicit monodromy factorizations of the genus-$3$ Lefschetz pencil on \textit{the four-torus $T^4$} constructed by Smith in~\cite{Smith_2001_torus}, geometorically in Subsection~\ref{Subsection:Smith's_Pencil} and combinatorially in Subsection~\ref{Subsection:Combinatorial_Smith's_Pencil}.
Smith also mentioned that by modifying the pencil on $T^4$ one can construct Lefschetz pencils on the total spaces of \textit{$T^2$-bundles over $T^2$}, provided that the bundles admit sections.
In this subsection we will follow Smith's idea in a combinatorial way; for any $\alpha,\beta\in \Mod(\Sigma_1^1;U)$ with $[\alpha,\beta]=1$, we will construct a genus-$3$ Lefschetz pencil $f_{\alpha,\beta}$ by modifying the factorization~\eqref{eq:CombinatorialSmith'sLP}, and prove the following theorem (which was also stated in Introduction): 

\LPtorusbdl*

We first observe presentations for the fundamental groups of $T^2$-bundles over $T^2$ with sections.
Let $p:X \rightarrow T^2$ be a torus bundle over the torus which has a section $S\subset X$ and $D\subset T^2$ a small disk.
We take a meridian $m$ and a longitude $l$ of the base $T^2$. 
We denote the monodromy along $m$ and $l$ by $\alpha$ and $\beta$, respectively. 
Since $p$ has a section, $\alpha$ and $\beta$ can be considered as elements in $\Mod(\Sigma_1^1;U)$, where $U=\{u\}\subset \Pa \Sigma_1^1$. 
A tubular neighborhood $\nu S$ can be decomposed into a $0$-handle contained in $p^{-1}(D)$, two $1$-handles $a,b$ whose cores are lifts of $m$ and $l$, and a $2$-handle.  
The preimage $p^{-1}(D)$ also admits a handle decomposition with the $0$-handle of $\nu S$, two $1$-handles $c,d$ whose cores are a longitude and a merdian of a regular fiber $T$, and a $2$-handle (We take $c$ for the longitude, $d$ for the meridian for the convenience of later calculations). 
The total space $X$ can be obtained from the union $p^{-1}(D)\cup \nu S$ by attaching four $2$-handles, four $3$-handles and a $4$-handle. 
Two of the $2$-handles are contained in the preimage of a neighborhood of $m$, while the other $2$-handles are contained in the preimage of a neighborhood of $l$. 
We can eventually obtain a handle decomposition of $X$ and the associated cell decomposition of $X$. 
We denote the $1$-cells corresponding to $a,b,c,d$ endowed with suitable orientations by the same symbols. 
Since $X$ has only one $0$-cell, the $1$-cells $a,b,c,d$ represent elements in $\pi_1(X)$. 
Furthermore, $c$ and $d$ also represent elements in $\pi_1(T)$, especially we can describe $\alpha(c),\alpha(d),\beta(c),\beta(d)$ as words consisting of $c$ and $d$.  
Analyzing attaching maps of the $2$-cells, we can easily prove the following:

\begin{lemma}\label{L:FundamentalGroup_T^2bdl}
	The fundamental group $\pi_1(X)$ has the following presentation:
	\[ \pi_1(X) = \left\langle 
	\begin{matrix}
	a, b, \\
	c, d 
	\end{matrix}
	\left|~
	\begin{matrix}
	[a,b], ~ca\alpha(c)^{-1}a^{-1}, ~cb\beta(c)^{-1}b^{-1},\\
	[c,d], ~da\alpha(d)^{-1}a^{-1}, ~db\beta(d)^{-1}b^{-1}
	\end{matrix}
	\right. \right\rangle. \] 

\end{lemma}

In order to modify the factorization~\eqref{eq:CombinatorialSmith'sLP} we need a key observation on a symmetrical property of some subwords in~\eqref{eq:CombinatorialSmith'sLP} as mapping classes.
We set $X_1 = t_{d_{1}} t_{d_{2}} t_{d_{3}}$, $Y_1 = t_{d_{4}} t_{d_{5}} t_{d_{6}}$, $X_2 = t_{d_{7}} t_{d_{8}} t_{d_{9}}$ and $Y_2 = t_{d_{10}} t_{d_{11}} t_{d_{12}}$ in $\Mod(\Sigma_3^4;U)$ and consider their actions on the curves $L_i$, $R_i$ ($i=1,2$), $S_L$ and $S_R$ on $\Sigma_3^4$ that are as depicted in Figure~\ref{F:Subfactors_Properties_1}.
The actions can be read off from similar actions on curves on the surface $\Sigma_2^4$.
Set $X = t_{{B}_{0,1}^{\prime}}t_{{B}_{1,1}^{\prime}}t_{{B}_{2,1}^{\prime}}$ and $Y= t_{{B}_{0,2}^{\prime}}t_{{B}_{1,2}^{\prime}}t_{{B}_{2,2}^{\prime}}$ in $\Mod(\Sigma_2^4;U)$ and consider the curves $A_i$, $B_i$ ($i=1,2$) on $\Sigma_2^4$ that are as in Figure~\ref{F:Subfactors_Properties_2}.
\begin{figure}[htbp]
	\centering
	\subfigure[The curves $L_1$, $L_2$, $S_L$, $R_1$, $R_2$, $S_R$ on $\Sigma_3^4$. ]{\includegraphics[height=25mm]{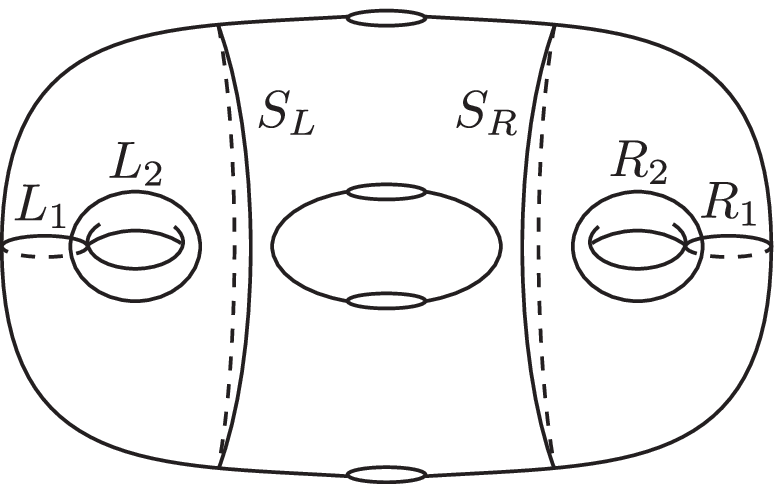}
		\label{F:Subfactors_Properties_1}}
	\hspace{.8em}
	\subfigure[The curves $A_1$, $A_2$, $B_1$, $B_2$ on $\Sigma_2^4$. ]{\includegraphics[height=25mm]{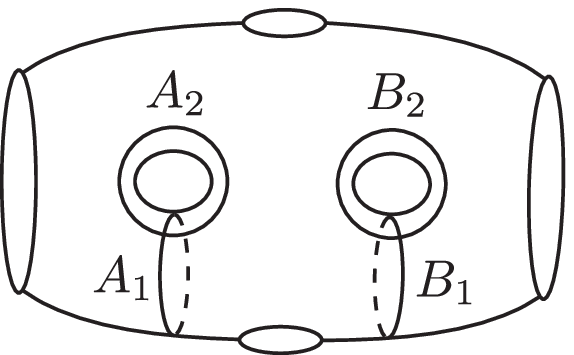}
		\label{F:Subfactors_Properties_2}}
	\hspace{.8em}
	\caption{Actions of $X$, $Y$, $X_1$, $Y_1$, $X_2$ and $Y_2$.} \label{F:ActionsOfSubfactors}
\end{figure}
\begin{lemma} \label{L:ActionsOfSubfactors}
	\textup{(1)} As a mapping class, each of $X$ and $Y$ in $\Mod(\Sigma_2^4;U)$ maps the $4$-tuple of simple closed curves $(A_1,A_2,B_1,B_2)$ to $(B_1,B_2,A_1,A_2)$ on $\Sigma_2^4$.\\
	\textup{(2)} Each of $X_1$, $Y_1$, $X_2$ and $Y_2$ in $\Mod(\Sigma_3^4;U)$ maps the $6$-tuple $(L_1,L_2,S_L,R_1,R_2,S_R)$ of simple closed curves to $(R_1,R_2,S_R,L_1,L_2,S_L)$ on $\Sigma_3^4$.
\end{lemma}
\begin{proof}
	It is simply routine work to check \textup{(1)}. 
	To see \textup{(2)}, we recall the embedding of $\Sigma_2^4$ to $\Sigma_3^4$ with which we dealt in Figures~\ref{F:SmithsLPConfiguration_1_1},~\ref{F:SmithsLPConfiguration_1_2},~\ref{F:SmithsLPConfiguration_1_3}.
	We can identify $X$, $Y$, $A_i$, $B_i$ with $X_1$, $Y_1$, $L_i$, $R_i$, respectively.
	Thus from \textup{(1)}, each of $X_1$ and $Y_1$ maps $(L_1,L_2,R_1,R_2)$ to $(R_1,R_2,L_1,L_2)$.
	Since $S_L$ and $S_R$ are the boundaries of the regular neighborhoods of $L_1 \cup L_2$ and $R_1 \cup R_2$, respectively, $X_1$ and $Y_1$ also switch $S_L$ and $S_R$.
	By considering the other embedding dealt with in Figures~\ref{F:SmithsLPConfiguration_2_1},~\ref{F:SmithsLPConfiguration_2_2},~\ref{F:SmithsLPConfiguration_2_3}, by which we can identify $X$, $Y$, $A_i$, $B_i$ with $X_2$, $Y_2$, $R_i$, $L_i$, respectively, we can verify the claims for $X_2$ and $Y_2$ in a similar manner.
\end{proof}

Now we construct a monodoromy factorization as a Lefschetz pencil corresponding to a given $T^2$-bundle over $T^2$ with an explicit monodromy factorization as a bundle.
We assume that the bundle has a section.
It is known that the section has to be of self-intersection number $0$ \footnote{One way to verify this is the following. Since a $T^2$-bundle over $T^2$ is a symplectic Calabi-Yau manifold, its canonical class $K$ is a torsion. Let $S$ be a section, which can be made symplectic. By the adjunction equality, we obtain $[S]^2=2\cdot 1-2+K(S)=0$.}.
Hence, the monodromy factorization has the form
\[
[\alpha,\beta]= t_{\delta}^0 = 1
\]
in $\Mod(\Sigma_1^1;U)$, where $\alpha$ and $\beta$ are the monodromies along the meridian $a$ and the londitude $b$ of the base torus, respectively, and $\delta$ is the boundary of the one-holed torus $\Sigma_1^1$.

We consider the two symmetrical embeddings $\varphi_L$ and $\varphi_R$ of $\Sigma_1^1$ into $\Sigma_3^4$ as shown in Figure~\ref{F:LPsOnHomotopyT2bundleoverT2}, one of which takes the meridian $d$ to $L_1$ and the longitude $c$ to $L_2$ and the other takes $d$ to $R_1$ and $c$ to $R_2$.
\begin{figure}[htbp]
	\centering
	\subfigure[The correspondence between the free loops.]{\includegraphics[height=70mm]{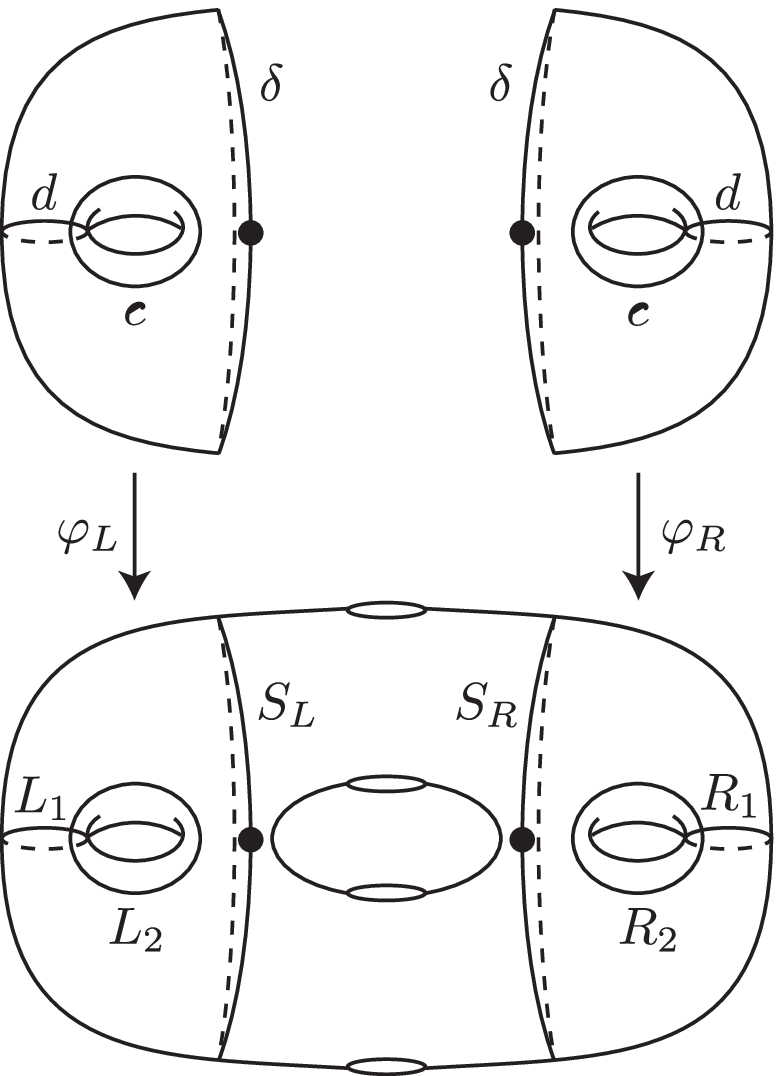}
		\label{F:LPsOnHomotopyT2bundleoverT2_Construction}}
	\hspace{3.8em}
	\subfigure[The correspondence between the based loops. Here $a_1$ stands for the loop $l_L \cdot \varphi_L(c) \cdot l_L^{-1}$. The loops $b_1$, $a_3$ and $b_3$ are similarly defined. ]{\includegraphics[height=70mm]{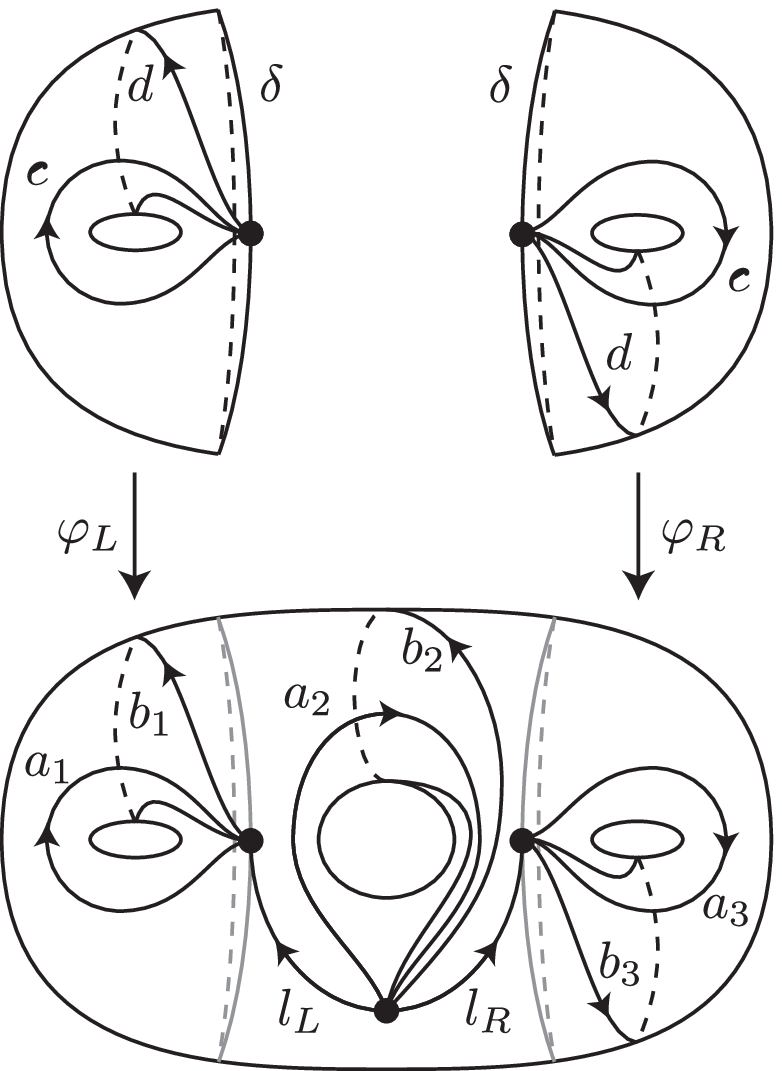}
		\label{F:LPsOnHomotopyT2bundleoverT2_pi1}}
	\caption{The embeddings $\varphi_L$ and $\varphi_R$.} \label{F:LPsOnHomotopyT2bundleoverT2}
\end{figure}
Then we can regard $\alpha$ and $\beta$ as elements in $\Mod(\Sigma_3^4;U)$ via those embeddings; for instance, take $\varphi_L \circ \alpha \circ \varphi_L^{-1}$ on the image of $\varphi_L$ and extend it as the inentity map on the complement $\Sigma_3^4 \setminus \varphi_L(\Sigma_1^1)$. 
Let $\alpha_L$ denote the resulting mapping class in $\Mod(\Sigma_3^4;U)$.
Similarly we have $\alpha_R$, $\beta_L$, $\beta_R$ corresponding to $\alpha$ via $\varphi_R$, $\beta$ via $\varphi_L$, $\beta$ via $\varphi_R$, respectively.
Note that we can deduce from the comutativity between $\alpha$ and $\beta$ that $\alpha_L$ and $\beta_L$ commute and that $\alpha_R$ and $\beta_R$ commute.
Obviously any other pair among $\alpha_L$, $\alpha_R$, $\beta_L$, $\beta_R$ also commutes.
Since $t_c$ and $t_d$ are generators of $\Sigma_1^1$, $\alpha$ and $\beta$ may be written as words of $t_c$ and $t_d$.
Fix such word expressions.
Then $\alpha_L$ and $\beta_L$ are written as the words of $t_{L_1}$ and $t_{L_2}$ corresponding to the fixed expressions, while $\alpha_R$ and $\beta_R$ are written as the corresponding words of $t_{R_1}$ and $t_{R_2}$.
By Lemma~\ref{L:ActionsOfSubfactors} the conjugation of $t_{L_i}$ by any of $X_1$, $Y_1$, $X_2$, $Y_2$ is $t_{R_i}$, hence, the conjugation of the fixed word for $\alpha_L$ by $X_i$ or $Y_i$ is exactly the fixed word for $\alpha_R$.
This simply means that the conjugation of $\alpha_L$ by $X_i$ or $Y_i$ is $\alpha_R$.
We can apply similar arguments to the conjugations of $\alpha_R$, $\beta_L$ and $\beta_R$.
In summary, we have the following \textit{switching property}:
\begin{align*}
	X_i \alpha_L X_i^{-1} &= \alpha_R, & X_i \alpha_R X_i^{-1} &= \alpha_L, & X_i \beta_L X_i^{-1} &= \beta_R, & X_i \beta_R X_i^{-1} &= \beta_L,\\
	Y_i \alpha_L Y_i^{-1} &= \alpha_R, & Y_i \alpha_R Y_i^{-1} &= \alpha_L, & Y_i \beta_L Y_i^{-1} &= \beta_R, & Y_i \beta_R Y_i^{-1} &= \beta_L,
\end{align*}
for $i=1,2$.

In order to create a desired factorization, we modify the factorization~\eqref{eq:CombinatorialSmith'sLP} by using $\alpha_R$, $\alpha_L$, $\beta_R$ and $\beta_L$ as follows:
{\allowdisplaybreaks
\begin{align*}
	t_{\delta_1}\cdots t_{\delta_4} &= t_{d_1} \cdots t_{d_{12}} \\ 
	&= X_1 Y_1 X_2 Y_2 \\ 
	&= \alpha_R \cdot X_1 Y_1 \cdot \alpha_R^{-1} \underline{\alpha_R \cdot X_2 \cdot \alpha_R^{-1}} ~\underline{\alpha_R \cdot Y_2 \cdot \alpha_R^{-1}}  \\
	&= \underline{\alpha_R X_1} Y_1  \alpha_R^{-1} {}_{\alpha_R}(X_2) {}_{\alpha_R}(Y_2)  \\ 
	&= X_1 \underline{\alpha_L Y_1}  \alpha_R^{-1} {}_{\alpha_R}(X_2) {}_{\alpha_R}(Y_2)  \\
	&= X_1  Y_1 \alpha_R  \alpha_R^{-1} {}_{\alpha_R}(X_2) {}_{\alpha_R}(Y_2)  \\
	\displaybreak[0] &= X_1  Y_1  {}_{\alpha_R}(X_2) {}_{\alpha_R}(Y_2)  \\
	&= X_1 \cdot \beta_R^{-1} \underline{\beta_R \cdot Y_1 \cdot \beta_R^{-1}} \beta_R \cdot {}_{\alpha_R}(X_2) {}_{\alpha_R}(Y_2)  \\
	&= \underline{X_1 \beta_R^{-1}} {}_{\beta_R}(Y_1) \underline{\beta_R {}_{\alpha_R}(X_2)} {}_{\alpha_R}(Y_2)  \\
	&= \beta_L^{-1} X_1  {}_{\beta_R}(Y_1) \underline{\beta_R \alpha_R X_2 \alpha_R^{-1}} {}_{\alpha_R}(Y_2)  \\
	&= X_1  {}_{\beta_R}(Y_1) \underline{\alpha_R \beta_R X_2 \alpha_R^{-1}} {}_{\alpha_R}(Y_2) \beta_L^{-1} \\
	&= X_1  {}_{\beta_R}(Y_1) \underline{\alpha_R X_2 \beta_L \alpha_R^{-1}} {}_{\alpha_R}(Y_2) \beta_L^{-1} \\
	&=  X_1  {}_{\beta_R}(Y_1) \underline{\alpha_R X_2 \alpha_R^{-1}} ~ \underline{\beta_L {}_{\alpha_R}(Y_2) \beta_L^{-1}} \\
	&=  X_1  {}_{\beta_R}(Y_1) {}_{\alpha_R}(X_2)   {}_{\beta_L \alpha_R}(Y_2),
\end{align*}
}
where the sign ``$=$'' above means equality as a mappig class, not as a factorization.
Here we freely used the switching property explained above as well as the commutativity among $\alpha_L$, $\alpha_R$, $\beta_L$, $\beta_R$ (and $t_{d_1}, \cdots, t_{d_4}$).
In other words, we have obtained the following factorization:
\begin{equation}
	t_{d_1}t_{d_2}t_{d_3} t_{\beta_R(d_4)}t_{\beta_R(d_5)}t_{\beta_R(d_6)} t_{\alpha_R(d_7)}t_{\alpha_R(d_8)}t_{\alpha_R(d_9)} t_{\beta_L\alpha_R(d_{10})}t_{\beta_L\alpha_R(d_{11})}t_{\beta_L\alpha_R(d_{12})} = t_{\delta_1}t_{\delta_2}t_{\delta_3}t_{\delta_4}. \label{eq:LPonT2bdl}
\end{equation}

Let $f_{\alpha,\beta} : X_{\alpha,\beta} \setminus B_{\alpha,\beta} \to \CP^1$ be 
the Lefschetz pencil corresponding to the monodromy factorization~\eqref{eq:LPonT2bdl}.
The pencil $f_{\alpha,\beta}$ has $12$ critical points and $4$ base points, hence the Euler characteristic of $X_{\alpha,\beta}$ is $0$.
The signature of $X_{\alpha,\beta}$ is also $0$ since we modified the factorization~\eqref{eq:CombinatorialSmith'sLP}, whose corresponding pencil has the signature $0$, by only using braid relations, which do not change the signature~\cite{EN_2005}.

\begin{lemma} \label{L:FundamentalGroupOfXalphabeta}
	The fundamental group $\pi_1(X_{\alpha,\beta})$ of the total space of the Lefschetz pencil $f_{\alpha,\beta}$ is isomorphic to that of the total space of the $T^2$-bundle over $T^2$ associated with the monodoromy factorization $[\alpha,\beta]=1$ in $\Mod(\Sigma_1^1;U)$.
\end{lemma}
\begin{proof}
It is a standard fact that the fundamental group $\pi_1(X)$ of the total space $X$ of a genus-$g$ Lefschetz pencil with a monodromy factorization $t_{c_n}\cdots t_{c_1}=t_{\delta_1}\cdots t_{\delta_p}$  is isomorphic to the quotient $\pi_1(\Sigma_g)/\left<c_1,\ldots,c_n\right>$, where $\left<c_1,\ldots,c_n\right>$ is the normal subgroup generated by the curves $c_1,\ldots,c_n$.
Let us begin with the easiest case that $\alpha=\beta=\id$, which is Smith's pencil itself.
We give an explicit presentation of $\pi_1(X_{\id,\id})$, which is of course $\pi_1(T^4) = \Z^4$, by deriving from the monodromy factorization~\eqref{eq:CombinatorialSmith'sLP}.
Starting from the standard generators $a_1,b_1,\cdots, a_3,b_3$ of $\pi_1(\Sigma_3)$ as in Figure~\ref{F:LPsOnHomotopyT2bundleoverT2_pi1}, we get a presentation of $\pi_1(X_{\id,\id})$ with the same generators and the following defining relators:
{\allowdisplaybreaks
\begin{align}
	& [a_1,b_1][a_2,b_2][a_3,b_3]=1, \tag{R-$0$}\\
	& a_1 [b_1,a_1] a_3^{-1} =1, \tag{R-$d_1$}\\
	& a_1 a_1 b_1^{-1} a_1^{-1} b_3 a_3^{-1} =1, \tag{R-$d_2$}\\
	& a_1 b_1^{-1} a_1^{-1} b_3 =1, \tag{R-$d_3$}\\
	& a_1 a_2 b_2 a_2^{-1} a_3^{-1} b_2^{-1} [a_3, b_3] =1, \tag{R-$d_4$}\\
	& a_1 b_1^{-1} a_2 b_2 a_2^{-1} b_3 a_3^{-1} b_2^{-1} [a_3,b_3] =1, \tag{R-$d_5$}\\
	& b_1 [b_3,a_3] b_2 b_3^{-1} a_2 b_2^{-1} a_2^{-1} =1, \tag{R-$d_6$}\\ 
	& a_1 a_2 [a_3, b_3] a_3^{-1} a_2^{-1} =1, \tag{R-$d_7$}\\
	& a_1 b_1^{-1} a_2 a_3 b_3 a_3^{-1} a_3^{-1} a_2^{-1} =1, \tag{R-$d_8$}\\
	& b_1 a_2 a_3 b_3^{-1} a_3^{-1} a_2^{-1} =1, \tag{R-$d_9$}\\
	& a_1 a_2 b_2^{-1} [a_3,b_3] a_3^{-1} b_2 a_2^{-1} =1,  \tag{R-$d_{10}$}\\
	& a_1 b_1^{-1} a_2 b_2^{-1} a_3 b_3 a_3^{-1} a_3^{-1} b_2 a_2^{-1} =1, \tag{R-$d_{11}$}\\
	& b_1 a_2 b_2^{-1} a_3 b_3^{-1} a_3^{-1} b_2 a_2^{-1} =1. \tag{R-$d_{12}$}
\end{align}
}
Here the relator (R-$d_i$) comes from the vanishing cycle $d_i$.
By substituting (R-$d_3$)  to (R-$d_2$) we obtain the relation $a_1 = a_3$, with which (R-$d_1$) implies that $[b_1,a_1]=1$.
Then (R-$d_3$) gives that $b_1 = b_3$.
Now we know that $[a_1,b_1]=[a_3,b_3]=1$, hence, (R-$0$) reduces to $[a_2,b_2]=1$.
Note that we so far have the set of the relations $a_1=a_3$, $b_1=b_3$, $[a_1,b_1]=[a_2,b_2]=1$, and that this is indeed equivalent to the set of the relations (R-$0$), (R-$d_1$), (R-$d_2$), (R-$d_3$).
With those new relations in mind, (R-$d_4$) becomes $[a_1, b_2]=1$, (R-$d_6$) becomes $[b_1,b_2]=1$, (R-$d_7$) becomes $[a_1,a_2]=1$ and (R-$d_9$) becomes $[b_1,a_2]=1$.
None of the other defining relators gives a new relation among $a_i$ and $b_j$.
Therefore, by renaming $a=a_2$, $b=b_2$, $c=a_1=a_3$, $d=b_1=b_3$, $\pi_1(X_{\id,\id})$ is the free abelian group generated by $a$, $b$, $c$ and $d$.

We can modify the presentation of $\pi_1(X_{\id,\id})$ to obtain a presentation of $\pi_1(X_{\alpha,\beta})$ with generators $a_1,b_1,\cdots, a_3,b_3$ and with defining relators (R-$0$), (R-$d_1$), (R-$d_2$), (R-$d_3$) and the following ones:
{\allowdisplaybreaks
	\begin{align}
		& a_1 a_2 b_2 a_2^{-1} \beta_R(a_3)^{-1} b_2^{-1} [\beta_R(a_3), \beta_R(b_3)] =1, \tag{R-$\beta_R(d_4)$}\\
		& a_1 b_1^{-1} a_2 b_2 a_2^{-1} \beta_R(b_3) \beta_R(a_3)^{-1} b_2^{-1} [\beta_R(a_3),\beta_R(b_3)] =1, \tag{R-$\beta_R(d_5)$}\\
		& b_1 [\beta_R(b_3),\beta_R(a_3)] b_2 \beta_R(b_3)^{-1} a_2 b_2^{-1} a_2^{-1} =1, \tag{R-$\beta_R(d_6)$}\\ 
		& a_1 a_2 [\alpha_R(a_3), \alpha_R(b_3)] \alpha_R(a_3)^{-1} a_2^{-1} =1, \tag{R-$\alpha_R(d_7)$}\\
		& a_1 b_1^{-1} a_2 \alpha_R(a_3) \alpha_R(b_3) \alpha_R(a_3)^{-1} \alpha_R(a_3)^{-1} a_2^{-1} =1, \tag{R-$\alpha_R(d_8)$}\\
		& b_1 a_2 \alpha_R(a_3) \alpha_R(b_3)^{-1} \alpha_R(a_3)^{-1} a_2^{-1} =1, \tag{R-$\alpha_R(d_9)$}\\
		& \beta_L(a_1) a_2 b_2^{-1} [\alpha_R(a_3),\alpha_R(b_3)] \alpha_R(a_3)^{-1} b_2 a_2^{-1} =1,  \tag{R-$\beta_L\alpha_R(d_{10})$}\\
		& \beta_L(a_1) \beta_L(b_1)^{-1} a_2 b_2^{-1} \alpha_R(a_3) \alpha_R(b_3) \alpha_R(a_3)^{-1} \alpha_R(a_3)^{-1} b_2 a_2^{-1} =1, \tag{R-$\beta_L\alpha_R(d_{11})$}\\
		& \beta_L(b_1) a_2 b_2^{-1} \alpha_R(a_3) \alpha_R(b_3)^{-1} \alpha_R(a_3)^{-1} b_2 a_2^{-1} =1. \tag{R-$\beta_L\alpha_R(d_{12})$}
	\end{align}
}
Again the relator (R-$*$) corresponds to each vanishing cycle.
As we discussed above, the relators (R-$0$), (R-$d_1$), (R-$d_2$), (R-$d_3$) imply that 
$a_1=a_3$, $b_1=b_3$ and $[a_1,b_1]=[a_2,b_2]=1$.
We regard $\pi_1(\Sigma_1)$ as the qutient $\pi_1(\Sigma_1^1) / \left< \delta \right> = \left< c, d \right> / \left< [c,d] \right> $ and then construct a homomorphism ${\varphi_L}_* : \pi_1(\Sigma_1) \to \pi_1(X_{\alpha,\beta})$ as follows: 
for an element $g\in \pi_1(\Sigma_1^1)$ we define ${\varphi_L}_*([g])$ to be $l_L \cdot \varphi_L(g) \cdot l_L^{-1}$, where $[g] \in \pi_1(\Sigma_1)=\pi_1(\Sigma_1^1)/\left< \delta \right>$ is an element represented by $g$. 
The map ${\varphi_L}_* $ is well-defined since $l_L\cdot {\varphi_L}(c)\cdot l_L^{-1} = a_1$, $l_L\cdot {\varphi_L}(d)\cdot l_L^{-1} =b_1$ and $[a_1,b_1]=1$.
Similarly we can define another homomorphism ${\varphi_R}_* : \pi_1(\Sigma_1) \to \pi_1(X_{\alpha,\beta})$ as $[g] \mapsto l_R \cdot \varphi_R(g) \cdot l_R^{-1}$.
However, ${\varphi_L}_*$ and ${\varphi_R}_*$ in fact coincide since ${\varphi_L}_*(c) = a_1 = a_3 = {\varphi_R}_*(c)$ and ${\varphi_L}_*(d) = b_1 = b_3 = {\varphi_R}_*(d)$.
It follows that $\alpha_L(a_1) = {\varphi_L}_*(\alpha(c)) = {\varphi_R}_*(\alpha(c)) = \alpha_R(a_3)$, and similarly $\alpha_L(b_1)=\alpha_R(b_3)$, $\beta_L(a_1) = \beta_R(a_3)$ and $\beta_L(b_1) = \beta_R(b_3)$.
Therefore we can identify $c$ with $a_1=a_3$ and $d$ with $b_1=b_3$ in a way that $\alpha(c)=\alpha_L(a_1)=\alpha_R(a_3)$ as a word of $c=a_1=a_3$ and $d=b_1=b_3$, and $\alpha(d)=\alpha_L(b_1)=\alpha_R(b_3)$, and so on.
We also rename $a=a_2$ and $b=b_2$.
Then (R-$\beta_R(d_4)$) becomes $c b \beta(c)^{-1} b^{-1}=1$,
(R-$\beta_R(d_6)$) becomes $d b \beta(d)^{-1} b^{-1}=1$,
(R-$\alpha_R(d_7)$) becomes $c a \alpha(c)^{-1} a^{-1}=1$ and 
(R-$\alpha_R(d_9)$) becomes $d a \alpha(d)^{-1} a^{-1}=1$.
No other defining relators give a new relation.
In conclusion, $\pi_1(X_{\alpha,\beta})$ is isomorphic to the group described in Lemma~\ref{L:FundamentalGroup_T^2bdl}, hence the fundanmental group of the $T^2$-bundle over $T^2$ associated with the monorodomy factorization $[\alpha,\beta]=1$, as desired.
\end{proof}

We are now ready to prove Theorem~\ref{T:LPtorusbdl}. 

\begin{proof}[Proof of Theorem~\ref{T:LPtorusbdl}]
The fundamental group of a $T^2$-bundle over $T^2$ cannot be isomorphic to that of a rational or ruled surface.
The theorem by Baykur-Hayano~\cite[Theorem 4.1]{BaykurHayano_multisection} impies that the Lefschetz pencil $f_{\alpha,\beta}$ is symplectic Calabi-Yau.
Furthermore, we can deduce from \cite[Corollary 3.3]{FriedlVidussi} that a symplectic Calabi-Yau four-manifold $M$ is homeomorphic to the total space $X$ of a $T^2$-bundle $X\to T^2$ with a section if and only if $\pi_1(M)$ is isomorphic to $\pi_1(X)$. 
By Lemma~\ref{L:FundamentalGroupOfXalphabeta} we can conclude that $X_{\alpha,\beta}$ is homeomorphic to the total space of $T^2$-bundle over $T^2$ with a section whose monodromy representation sends two elements generating $\pi_1(T^2)$ to $\alpha$ and $\beta$.
\end{proof}

We end this subsection with the following conjecture: 

\begin{conjecture}

The pencil $f_{\alpha,\beta}$ is isomorphic to that constructed by Smith~\cite{Smith_2001_torus}. 
In particular the total space $X_{\alpha,\beta}$ is \emph{diffeomorphic} to that of a $T^2$-bundle over $T^2$. 

\end{conjecture}



\appendix

\section{Homology groups of Lefschetz pencils from monodoromy factorizations}\label{A:2ndhomology}

Let $f:X\setminus B\to \CP^1$ be a genus-$g$ Lefschetz pencil and $t_{c_n}\cdots t_{c_1}=t_{\delta_1}\cdots t_{\delta_p}$ a monodromy factorization of $f$. 
As we used in the proof of Lemma~\ref{L:FundamentalGroupOfXalphabeta}, the fundamental group $\pi_1(X)$ is calculated from the vanishing cycles  $c_1,\ldots,c_n$.
In particular, we can easily calculate $H_1(X;\Z)$, which is isomorphic to the abelianization of $\pi_1(X)$. 
Since the Euler characteristic of $X$ is equal to $4-4g+n-p$, we can deduce from the universal coefficient theorem that $H_2(X;\Z)$ is isomorphic to $\Z^{r}\oplus T(H_1(X;\Z))$, 
where $r = 2-4g+n-p+2\rank(H_1(X;\Z))$ and $T(H_1(X;\Z))$ is the torsion part of $H_1(X;\Z)$. 
However, the above observation does not give any information on the fiber class of $f$, especially the divisibility of $f$. 
In this appendix, we will explain how to obtain the divisibility of $f$ by calculating the second homology of $X$ from a handlebody structure associated with $f$. 

Let $\tilde{X}$ be the blow-up of $X$ at the points in $B$ and $\tilde{f}$ the Lefschetz fibration on $\tilde{X}$ derived from $f$. 
The manifold $\tilde{X}$ can be decomposed as follows: 
\begin{equation}\label{E:decomposition_totalsp}
\tilde{X} = D^2\times \Sigma_g \cup \mbox{($2$-handles)} \cup D^2\times \Sigma_g, 
\end{equation}
where the two $D^2\times \Sigma_g$'s are tubular neighborhoods of regular fibers and each $2$-handle corresponds with a Lefschetz singularity of $\tilde{f}$ and are attached along a vanishing cycle with framing $(-1)$ with respect to the fiber framing (see \cite{Kas_1980}). 
We denote the $2$-handle in the above decomposition corresponding the vanishing cycle $c_i$ by $h_L^i$. 
The two $D^2\times \Sigma_g$'s above can be further decomposed as follows: 

\begin{itemize}

\item
the former $D^2\times \Sigma_g$ can be decomposed into a $0$-handle, $2g$ $1$-handles and a $2$-handle, which we denote by $h_F$, 

\item
the latter $D^2\times \Sigma_g$ can be decomposed into $p$ $2$-handles, $2g+p-1$ $3$-handles and a $4$-handle, where the cores of the $2$-handles are contained in the exceptional spheres arising in the blow-up. 
We denote the $2$-handle in this decomposition by $h_S^1,\ldots,h_S^p$. 

\end{itemize}

\noindent
The decompositions above gives rise to a handlebody structure of $\tilde{X}$, especially we can obtain a chain complex $\mathcal{C} =\{(C_i,\Pa_i)\}_i$ such that $C_i$ is a free abelian group generated by the $i$-handles above. 
The homology group $H_i(\mathcal{C})$ is isomorphic to $H_i(\tilde{X};\Z)$, and the fiber class of $\tilde{f}$ (resp.~the exceptional classes) are represented by $h_F$ (resp.~$h_S^1,\ldots,h_S^p$). 

The group $C_1$ can be identified with $H_1(\Sigma_g;\Z)$ in the obvious way. 
Under this identification, $\Pa_2:C_2\to C_1$ sends $h_L^i$ to the homology class of the corresponding vanishing cycle. 
Since $h_F$ and $h_S^j$ are contained in $Z_2=\ker \Pa_2$, we can obtain a generating set $\{z_1,\ldots,z_m,h_F,h_S^1,\ldots,h_S^p\}$ of $Z_2$ such that each $z_i$ is a linear combination of $h_L^1,\ldots,h_L^n$. 

The decomposition of the latter $D^2\times \Sigma_g$ in \eqref{E:decomposition_totalsp} is induced from the handle decomposition of $\Sigma_g$, which consists of $p$ $0$-handles, $2g+p-1$ $1$-handles and a $2$-handles. 
Indeed, we can regard an $i$-handle of this $D^2\times \Sigma_g$ as the product of $D^2$ and an $(i-2)$-handle of $\Sigma_g$. 
Let $\Sigma_g^p$ be the surface obtained by removing the $p$ $0$-handles of $\Sigma_g$. 
We take points $q_1,\ldots,q_p\in \Pa\Sigma_g^p$ so that the map $\pi_0(\{q_1,\ldots,q_p\})\to\pi_0(\Pa\Sigma_g^p)$ induced by the inclusion is bijective. 
Let $\delta_i$ be the simple closed curve in $\Sigma_g^p$ parallel to the boundary component containing $q_i$. 
We take a disk $D\subset \Int(\Sigma_g^p)$, a point $q\in \Pa D$ and a path $\alpha\subset \Sigma_g^p\setminus \Int(D)$ connecting $q_1$ and $q$. 
Let $\Sigma_g^{p+1}$ be the surface $\Sigma_g^p\setminus \Int(D)$ and $Q$ the set $\{q_1,\ldots,q_p,q\}$. 
It is easy to see that the homology group $H_1(\Sigma_g^{p+1},Q;\Z)$ is isomorphic to $\Z^{2g+2p}$ and is generated by elements represented by the cores of $1$-handles of $\Sigma_g$, $\delta_1,\ldots,\delta_p$ and $\alpha$. 
We can define an intersection pairing
\[
\varrho: H_1(\Sigma_g^{p+1},Q;\Z)\times H_1(\Sigma_g^{p+1};\Z)\to \Z
\]
which assigns $(\alpha,\beta)\in H_1(\Sigma_g^{p+1},Q;\Z)\times H_1(\Sigma_g^{p+1};\Z)$ to the algebraic intersection between a union of oriented paths in $\Pi(\Sigma_g^{p+1},Q)$ representing $\alpha$ and a closed curve representing $\beta$, where $\Pi(\Sigma_g^{p+1},Q)$ is the set of paths whose edges are points in $Q$. 

\begin{lemma}\label{L:boundary operator Pa3}

Let $\eta$ be a $3$-handle in the latter $D^2\times \Sigma_g$ in \eqref{E:decomposition_totalsp} and $\eta'$ the $1$-handle of $\Sigma_g$ corresponding $\eta$. 
Suppose that $\eta'$ is attached to two $0$-handles corresponding $h_S^i$ and $h_S^j$. 
We inductively take $k_i\in \Z$ and $\eta_i\in H_1(\Sigma_g^{p+1},Q;\Z)$ as follows: 

\begin{itemize}

\item
$\eta_0=\eta'$, 

\item
$k_i = \varrho(\eta_{i-1},{h_L^{i}}')$ and $\eta_i = \eta_{i-1}-k_i {h_L^i}'$, 

\end{itemize}

\noindent
where ${h_L^i}'$ is the $1$-handle corresponding $h_L^i$.
Then the element $\eta_n$ is equal to $\eta' + {h_S^i}'-{h_S^j}' - k_F[\Pa D]$ for some $k_F\in \Z$, and the following equality holds: 
\[
\Pa_3(\eta) = \sum_{l=1}^n k_l h_L^l + k_Fh_F + h_S^i-h_S^j. 
\]

\end{lemma}

\begin{proof}
Note first that the coefficients of $\Pa_3(\eta)$ are equal to the algebraic intersections between attaching sphere of $\eta$ and the belt spheres of the corresponding $2$-handles.  
The attaching sphere of $\eta$ is the union of 
\begin{itemize}

\item
the product of $D^2$ and the edges of the core of $\eta'$, and

\item
the product of $\Pa D^2$ and the core of $\eta'$. 

\end{itemize}
The former part intersects the belt spheres of $h_S^i$ and $h_S^j$ geometrically once (and the signs of these intersections are opposite).
It is easy to see that the latter part intersects the belt sphere of $h_L^i$ algebraically $k_i$ times. 
Let $\Mod(\Sigma_g^{p+1})$ be the set of isotopy classes of self-diffeomorphisms of $\Sigma_g^{p+1}$ \textit{which is not necessarily the identity map on the boundary}. 
The product $t_{c_n}\cdots t_{c_1}$ in $\Mod(\Sigma_g^{p+1})$ is a diffeomorphism obtained by pushing the disk $D$. 
We can verify that the path $t_{c_n}\cdots t_{c_1}(\mbox{the core of }\eta')$ goes though the center of $D$ algebraically $k_F$ times when we move it to $\eta$ by an isotopy. 
This observation implies that the product of $\Pa D^2$ and the core of $\eta'$ intersects the belt sphere of $h_F$ algebraically $k_F$ times. 
\end{proof}

Let $\{\alpha_1,\ldots,\alpha_{2g+p-1}\}$ be a generating set of $C_3$. 
Using Lemma~\ref{L:boundary operator Pa3} we can obtain the representation matrix $A_0$ of $\Pa_3$ with respect to the generating sets $\{\alpha_1,\ldots,\alpha_{2g+p-1}\}$ and $\{z_1,\ldots,z_m,h_F,h_S^1\ldots,h_S^p\}$: 
\[
(\Pa_3(\alpha_1),\ldots,\Pa_3(\alpha_{2g+p-1})) = (z_1,\ldots,z_m,h_F,h_S^1,\ldots,h_S^p) A_0. 
\]
Furthermore, applying fundamental operations to $A_0$ we can find another generating set $\{\alpha_1',\ldots,\alpha_{2g+p-1}'\}$ of $C_3$ and elements $z_1',\ldots,z_{m+1}'$ of $Z_2$ such that 

\begin{itemize}

\item
$z_i'$ is a linear combination of $z_1,\ldots,z_m,h_F$, 

\item
$\{z_1',\ldots,z_{m+1}',h_S^1,\ldots,h_S^p\}$ is a generating set of $C_2$, 

\item
the first $m+1$ rows of the representation matrix $A_1$ of $\Pa_3$ with respect to $\{\alpha_1',\ldots,\alpha_{2g+p-1}'\}$ and $\{z_1',\ldots,z_{m+1}',h_S^1,\ldots,h_S^p\}$ has non-zero entries on only its diagonal part, 

\item
let $e_1,\ldots,e_{m+1}$ be the first $m+1$ diagonal entries of $A_1$. 
There exists a positive integer $k$ such that $e_l =0$ for any $l\geq k$ and $e_i|e_{i+1}$ for any $i \leq m$. 

\end{itemize}

\begin{lemma}\label{L:property_coefficient_A1}

For any $i \leq k$ and $j \geq m+2$, all the entries in the $j$-th column of $A_1$ is eqaul to $0$ and the $(j,i)$-entry of $A_1$ is divisible by $e_i$. 

\end{lemma}

\begin{proof}
Let $a_{ij}$ be the $(i,j)$-entry of $A_1$. 
For any $i\geq m+2$, the cycle $\sum_{j\geq m+2} a_{ji}h_S^{j-m-1}$ is a boundary. 
Since the handle $h_S^{j}$ represents the exceptional class, $[h_S^1],\ldots,[h_S^p]$ is linearly independent in $H_2(\tilde{X})$. 
Thus, $a_{ji}=0$ if $i\geq m+2$. 
Since the cycle $e_iz_i' + \sum_{j\geq m+2} a_{ji}h_S^{j-m-1}$ is a boundary, $e_i[z_i']$ is equal to $\sum_{j\geq m+2} a_{ji}[h_S^{j-m-1}]$ in $H_2(\mathcal{C})$. 
The divisibility of $\sum_{j\geq m+2} a_{ji}[h_S^{j-m-1}]$ in $H_2(\tilde{X})$ is equal to $\gcd(\{a_{ji}\}_{j\geq m+2})$ since $[h_S^{j-m-1}]$ is the exceptional class. 
On the other hand, the divisibility of $e_i[z_i']$ is divisible by $e_i$. 
Thus, $e_i|a_{ji}$ for any $j\geq m+2$. 
\end{proof}

\noindent
Lemma~\ref{L:property_coefficient_A1} implies that we can find cycles $z_1'',\ldots,z_{m+1}''$ such that 
\begin{itemize}

\item
$\{z_1'',\ldots,z_{m+1}'',h_S^1,\ldots,h_S^p\}$ is a generating set of $C_2$, 

\item
the representation matrix $A_2$ of $\Pa_3$ with respect to $\{\alpha_1',\ldots,\alpha_{2g+p-1}'\}$ and $\{z_1'',\ldots,z_{m+1}'',h_S^1,\ldots,h_S^p\}$ has non-zero entries on only its diagonal part, 

\item
the first $k$ diagonal entries of $A_2$ are $e_1,\ldots,e_k$ and the other entries are $0$. 

\end{itemize}

\noindent
In particular, the second homology group $H_2(\tilde{X};\Z)$ is isomorphic to $(\oplus_i\Z/e_i\Z)\oplus \Z^{m+1-k}\oplus \Z^p$, where the former two components are generated by $z_1'',\ldots,z_{m+1}''$, while the last component is generated by $h_S^1,\ldots,h_S^p$. 
Furthermore, the natural homomorphism $H_2(\tilde{X};\Z) \to H_2(X;\Z)$ coincides with the projection $(\oplus_i\Z/e_i\Z)\oplus \Z^{m+1-k}\oplus \Z^p\to (\oplus_i\Z/e_i\Z)\oplus \Z^{m+1-k}$. 
The fiber class of $\tilde{f}$ is represented by $h_F$, which we can represents as a linear combination of $z_1'',\ldots,z_{m+1}'',h_S^1,\ldots,h_S^p$. 
Thus, the fiber class of $f$ is the image of $[h_F]$ under the projection $H_2(\tilde{X};\Z) \to H_2(X;\Z)$. 

In summary, the procedure above does not only give an isomorphism between $H_2(X;\Z)$ and an abelian group, but also determine which element in the abelian group corresponds with the fiber class of $f$. 
In particular we can calculate the divisibility of $f$. 
Moreover, in order to apply this procedure in practice we need only (possibly tedious) linear algebraic calculations, which we can do using a computer.

\vspace{1em}
\noindent
{\it Acknowledgements.} 
The authors would like to thank Refik \.{I}nan\c{c} Baykur and Tian-Jun Li for having fruitful discussions and making helpful comments on an earlier draft of this paper.  
The second author was supported by JSPS KAKENHI (26800027).

\end{document}